\documentclass[11pt]{article}
\usepackage[margin=1in]{geometry}
\usepackage{amsmath,amssymb,amsthm,mathtools,bm,dsfont}
\usepackage{algorithm,algorithmicx,algpseudocode}
\usepackage{booktabs}
\usepackage{authblk}
\usepackage{multirow}
\usepackage{microtype}
\usepackage{upgreek}
\usepackage{lmodern}
\usepackage[dvipsnames]{xcolor}
\usepackage{graphicx}
\usepackage{etoc}
\usepackage[font=footnotesize, margin=1cm]{caption}
\usepackage{tabularx,array}

\newcolumntype{L}[1]{>{\raggedright\arraybackslash}p{#1}}
\newcolumntype{Y}{>{\raggedright\arraybackslash}X}

\usepackage[numbers,sort]{natbib}
 \bibpunct[, ]{[}{]}{,}{n}{}{,}%

\usepackage[hidelinks]{hyperref}
\hypersetup{
    colorlinks=false
}

\newcommand{\E}{\mathbb{E}}
\newcommand{\R}{\mathbb{R}}
\newcommand{\D}{\mathbb{D}}
\newcommand{\Var}{\mathrm{Var}}
\newcommand{\Cov}{\mathrm{Cov}}
\newtheorem{theorem}{Theorem}
\newtheorem{proposition}{Proposition}
\newtheorem{lemma}{Lemma}
\newtheorem{corollary}{Corollary}

\newtheorem{assumption}{Assumption}
\newtheorem{remark}{Remark}
\newtheorem{example}{Example}

\title{Robust Queueing for Single-Server Queues with Abandonment}
\author{Wei You}

\affil{Department of Industrial Engineering and Decision Analytics, The Hong Kong University of Science and Technology, Hong Kong SAR, China, \url{weiyou@ust.hk}}

\date{\today}

\begin{document}

\maketitle
\begin{abstract}
Single-server queues with customer abandonment arise in call centers and other service systems, yet their steady-state performance is analytically tractable only in special cases.
We develop Robust Queueing approximations for the mean stationary virtual waiting time in the $GI/GI/1{+}GI$ model.
Our starting point is an exact reverse-time representation of the workload in terms of the net-input process, defined as the work brought by customers who eventually enter service minus cumulative service capacity.
RQ approximates this random process by its mean plus a robustness parameter times its standard deviation.
Abandonment makes both moments endogenous because the probability that a customer enters service depends on the waiting time observed upon arrival.
We resolve this dependence by imposing self-consistency through a deterministic trial approximation to the mean stationary virtual waiting time.
For the drift, the Poisson compensator identity is exact under Poisson arrivals; under renewal arrivals, the resulting mean Palm correction is controlled on the relevant heavy-traffic optimizer scales.
For the variance, we develop a deterministic-time-change surrogate and a refined finite-system surrogate.
The refined surrogate heuristically interpolates between two proved variance limits using a scale-dependent variance-reduction factor derived from a heavy-traffic diffusion limit.
Both constructions reduce to one-dimensional fixed-point equations that can be solved by bisection using the arrival index of dispersion for counts, the service-time squared coefficient of variation, and the patience distribution.
We establish heavy-traffic limits for the $GI/GI/1{+}GI$ model and for both RQ fixed points.
These limits calibrate the robustness parameter and provide theoretical support for the variance-reduction function.
Numerical experiments show that the refined approximation is accurate over a broad parameter range and identify its principal limitation in systems with short patience times and high abandonment.
\end{abstract}

\textbf{Keywords:} robust queueing, customer abandonment, virtual waiting time, indices of dispersion, heavy-traffic limits

\section{Introduction}

Customer impatience and abandonment (reneging) are defining features of many modern service systems, including call centers, healthcare delivery, and online service platforms.
In these settings, customers may leave without receiving service when delays are perceived as too long, altering both operational efficiency and quality of service.
Much of the call-center literature therefore models systems with abandonment; see, e.g., the many-server asymptotic analysis in \cite{zeltyn2005call}, the Erlang-A call-center model in \cite{garnett2002designing}, and the survey in \cite{dai2012many}.
While these applications often involve many servers, single-server models remain important as building blocks for more complex service networks (e.g., sequential service stages) and as primitives in decomposition approximations.

In this paper we study the classical first-come-first-served $GI/GI/1{+}GI$ queue with customer abandonment.
Arrivals follow a renewal process, service requirements are i.i.d.\ with a general distribution, and patience times are i.i.d.\ with a general distribution; a customer abandons if service has not begun by the time its patience expires.
We focus on the \emph{virtual waiting time} (also called the \emph{offered waiting time}), denoted by $Z(t)$, and in particular on its stationary mean $\E[Z(\infty)]$.
The virtual waiting time is a fundamental performance metric because it directly summarizes the system's congestion, underlies delay announcements, and can be used to approximate related quantities such as the probability of abandonment and mean queue length \cite{ward2012asymptotic,lee2020stationary,lee2021stationary}.

Despite its apparent simplicity, the $GI/GI/1{+}GI$ model is analytically challenging.
Abandonment creates a nonlinear feedback loop: the waiting time affects which customers remain in queue, which in turn changes the future workload seen by subsequent arrivals.
Exact steady-state descriptions are available only in special cases; early work establishing structural relations between actual and virtual waiting times includes \cite{stanford1979reneging,baccelli1984single}.
For general primitives, one typically relies on asymptotic approximations, numerical schemes, or simulation.

We develop new \emph{robust queueing} (RQ) approximations for $\E[Z(\infty)]$ that are fast, require only low-dimensional traffic descriptors, and remain accurate away from classical heavy-traffic regimes.
Our approach builds on the stochastic RQ methodology that approximates single-server performance using reverse-time supremum representations and variability summaries in the form of \emph{indices of dispersion} \cite{fendick1989measurements,whitt2018using,whitt2019advantage}.
These ideas have been extended to open networks via IDC-based flow propagation, yielding an RQ network analyzer analogous in spirit to the classical queueing network analyzer (QNA) \cite{whitt1983queueing}; see \cite{whitt2022robust} and references therein.

\subsection{Literature Review}\label{sec:lit_review}

\paragraph{Exact analysis and structural properties.}
In single-server queues with deadlines or patience times, early work analyzed reneging and established fundamental stability and distributional relations; see, e.g., \cite{stanford1979reneging,baccelli1984single}.
Even for Poisson arrivals, general patience times lead to integral-equation characterizations rather than closed forms, and tractable steady-state formulas are typically restricted to Markovian special cases.

\paragraph{Heavy-traffic diffusion approximations.}
A major line of work develops diffusion approximations for $GI/GI/1{+}GI$ queues for the offered waiting time.
\citet{ward2005diffusion} show that, under conventional heavy-traffic scaling, the offered waiting time can be approximated by a reflected Ornstein--Uhlenbeck (ROU) diffusion.
\citet{reed2008approximating} introduce \emph{hazard-rate scaling} so that the heavy-traffic diffusion limit incorporates the full patience-time distribution through a nonlinear drift term.
\citet{lee2011convergence} further establish heavy-traffic convergence for general patience-time distributions (allowing, e.g., state-dependent arrival intensities) and derive limits for queue length.
A broader perspective on asymptotic regimes for queues with reneging is surveyed in \cite{ward2012asymptotic}.

Using the stationary distribution of a diffusion limit as a proxy for the steady-state queue requires an \emph{interchange of limits} justification.
For the $GI/GI/1{+}GI$ model, \cite{lee2020stationary} establishes convergence of the scaled stationary offered-waiting-time distribution (and moments) to the stationary distribution of the limiting diffusion, resolving a question left open in \cite{ward2005diffusion}.
This result is extended under more general patience-time scaling (including hazard-rate scaling) in \cite{lee2021stationary}.

Beyond diffusion limits derived from specific heavy-traffic parameter scalings, \cite{huang2018beyond} develops universal performance bounds and diffusion-based approximations for the $M/GI/1{+}GI$ queue that are valid uniformly over families of patience distributions and across heavy-traffic regimes.
These results provide complementary support for diffusion proxies, but they still rely on Markovian arrivals and do not directly address non-renewal inputs arising endogenously in networks.

\paragraph{Robust queueing and indices of dispersion.}
Indices of dispersion for counts/work were introduced as variability summaries of offered traffic and were used to predict mean workload in single-server queues in \cite{fendick1989measurements}.
The robust queueing approach in \cite{whitt2018using} shows how to convert IDC/IDW information into accurate approximations for mean steady-state workload in the general $G/G/1$ queue, with asymptotic correctness in light and heavy traffic and the ability to capture temporal dependence.
Subsequent work emphasizes the value of IDC-based descriptions \cite{whitt2019advantage} and extends the methodology to open networks through IDC propagation \cite{whitt2022robust}.
Robust queueing has also been pursued from a robust-optimization perspective \cite{bandi2015robust}, but our focus is on developing stochastic-model performance approximations in the IDC-based RQ framework.

\subsection{Contributions and Organization}\label{sec:intro_contrib}

The paper makes five contributions.

\begin{itemize}
\item \textbf{A stationary RQ formulation for abandonment queues.}
Starting from the exact reverse-time representation, we express the approximation as a deterministic scalar fixed point.
The fixed point separates the construction into an effective-input drift, an effective-input variance, a robustness calibration, and a one-dimensional solve.

\item \textbf{A drift approximation with an explicit renewal correction.}
For Poisson arrivals, predictable thinning gives an exact compensator for effective arrivals.
For renewal arrivals, we identify the discrepancy as a Palm/time-average correction and show that its mean is negligible on the optimizer scales relevant to the RQ heavy-traffic limits.

\item \textbf{Two variance surrogates.}
The first RQ surrogate treats effective input as a renewal-reward process evaluated at a deterministic time change.
The refined surrogate introduces a scale-dependent variance-reduction function $w_{c,k}$ that quantifies how abandonment feedback suppresses effective-input variability over longer horizons.

\item \textbf{Asymptotic analysis and calibration.}
We prove heavy-traffic limits for the $GI/GI/1{+}GI$ model and for both RQ fixed points.
A single formulation recovers the underloaded, critical, and overloaded scales, with the boundary determined by the order $k$ of the first nonzero derivative of the patience distribution at the origin.
The critical limits calibrate the robustness parameter $b$; the underloaded calibration $b=\sqrt{2}$ is recovered as a limit, while the overloaded leading order is independent of $b$.

\item \textbf{Numerical and network evidence.}
The refined algorithm is compared with established benchmarks over a broad parameter grid.
We also give a heuristic tandem extension that feeds an approximation of the upstream departure IDC into the downstream RQ calculation.
\end{itemize}

The rest of the paper is organized as follows.
Section~\ref{sec:prelim} reviews ordinary RQ, develops the effective-input representation for queues with abandonment, and summarizes the approximation pipeline.
Section~\ref{sec:drift_approx} constructs the drift surrogate.
Section~\ref{sec:RQ_ab1} gives the first RQ algorithm and its calibration.
Section~\ref{sec:var} derives the variance-reduction function and the refined RQ algorithm.
Section~\ref{sec:numerical} reports the numerical experiments, and Section~\ref{sec:conclusion} concludes.
The appendices review benchmark methods, describe the tandem and secondary-measure heuristics, and contain the proofs.

\section{Preliminaries}\label{sec:prelim}

\subsection{Review of Robust Queueing for the \texorpdfstring{$GI/GI/1$}{GI/GI/1} Model}

Consider a stable single-server $GI/GI/1$ queue with infinite waiting room and a first-come-first-served service discipline.
Let $A(t)$ denote an arrival counting process with stationary and ergodic increments, rate $\lambda$, and $\Var(A(t))<\infty$ for all $t>0$.
Let $\{V_i\}_{i\ge1}$ be an i.i.d.\ sequence of service times, independent of $A(\cdot)$, with mean $1/\mu$ and finite variance.
Let $\rho\triangleq \lambda/\mu$, and assume $\rho<1$.

A convenient workload representation is based on the cumulative work process $ \tilde Y(t) \triangleq \sum_{i=1}^{A(t)} V_i$ and the associated net-input process, under unit service capacity, $\tilde N(t) \triangleq \tilde Y(t)-t$.
For a system that starts empty at time $0$, the workload process $\tilde Z(t)$ is given by the Skorokhod reflection mapping applied to $\tilde N(\cdot)$:
\begin{equation}\label{eq:reverse_GG1}
    \tilde Z(t) = \tilde N(t)-\inf_{0\le u\le t}\tilde N(u)
    = \sup_{0\le s\le t} \bigl\{ \tilde N(t)-\tilde N(t-s) \bigr\}.
\end{equation}
Thus the workload is the running supremum of reverse-time net-input increments.

Exact analysis of \eqref{eq:reverse_GG1} is typically intractable in non-Markovian settings.
Robust Queueing (RQ) replaces the stochastic reverse-time increment inside the supremum by a deterministic surrogate: its stationary mean plus a robustness parameter $b$ times its stationary standard deviation.
For the ordinary $GI/GI/1$ model, we have $\E\bigl[ \tilde N(t)-\tilde N(t-s) \bigr] = -(1-\rho)s$ and $\Var\bigl( \tilde N(t)-\tilde N(t-s) \bigr) = \Var\bigl( \tilde Y(t)-\tilde Y(t-s) \bigr)$.
Define the index of dispersion for work (IDW) by
\[
    \tilde I_w(s)
    \triangleq \frac{ \Var\bigl(\tilde Y(t)-\tilde Y(t-s)\bigr) }{ \E\bigl[\tilde Y(t)-\tilde Y(t-s)\bigr]\E[V_1] }
    = \frac{ \Var\bigl(\tilde Y(t)-\tilde Y(t-s)\bigr) }{ \rho s/\mu }.
\]
The RQ surrogate for the reverse-time increment is therefore $-(1-\rho)s + b\sqrt{\rho s \tilde I_w(s) / \mu}$.

The stationary RQ approximation is obtained as a deterministic optimization problem by applying this surrogate to the stationary infinite-horizon reverse-time representation:
\begin{equation}\label{eq:RQ_GG1}
    \tilde Z_{\mathrm{RQ}} \triangleq \sup_{s\ge0} \left\{ -(1-\rho)s + b\sqrt{\rho s \tilde I_w(s) / \mu} \right\}.
\end{equation}
When $\tilde I_w(s)$ is bounded above, the objective in \eqref{eq:RQ_GG1} has negative linear drift and at most square-root growth in its variability term, so the supremum is finite.

\subsection{The Dynamics of the \texorpdfstring{$GI/GI/1{+}GI$}{GI/GI/1+GI} Model}

We now turn to the $GI/GI/1{+}GI$ queue with abandonment.
Customers arrive according to a general right-continuous counting process $A(t)$.
Let $U$ denote a generic interarrival time.
Each customer has an i.i.d.\ service time $V$ and an i.i.d.\ patience time $D$.
A customer abandons if it has not entered service before its patience time expires; the notation $+GI$ indicates that the patience-time distribution is general.

The dynamics are most conveniently described in terms of the \emph{virtual waiting time process} $Z(t)$, defined as the waiting time at time $t$ of a hypothetical customer arriving at $t$ with infinite patience (also called the \emph{offered waiting time}).
Let $T_i \triangleq \inf\{t>0: A(t)=i\}$ be the arrival time of the $i$th customer.
Given $Z(\cdot)$, customer $i$ is offered waiting time $W_i \triangleq Z(T_i-)$, and is eventually served if and only if $D_i>W_i$.

In contrast to the $GI/GI/1$ model without abandonment, the evolution of $Z(t)$ is driven by the \emph{effective} workload that will eventually be processed by the server.
Define the \emph{effective arrival process}
\begin{equation}\label{eq:effective_arrival}
    A_0(t) \triangleq \sum_{i=1}^{A(t)} \mathds{1}\{D_i>W_i\},
\end{equation}
the number of arrivals by time $t$ who will eventually enter service.
Closely related is the effective total-input process
\[
    Y(t) \triangleq \sum_{i=1}^{A(t)} V_i \mathds{1}\{D_i>W_i\},
\]
which counts the total amount of work brought by customers arriving by time $t$ who do not abandon.
The corresponding effective net-input process is
\begin{equation}\label{eq:effective_net_input}
    N(t) \triangleq Y(t) - t.
\end{equation}
Following the same reflection argument that yields \eqref{eq:reverse_GG1}, assuming the system starts empty at time $0$, the virtual waiting time admits the reverse-time (supremum) representation
\begin{equation}\label{eq:reverse_ab}
    Z(t) = \sup_{0\le s\le t}\bigl\{N(t)-N(t-s)\bigr\}.
\end{equation}
Comparing \eqref{eq:reverse_GG1} and \eqref{eq:reverse_ab}, the essential distinction lies in the underlying net-input process.
For the $GI/GI/1{+}GI$ model, the object of interest is the \emph{effective} net-input process \eqref{eq:effective_net_input}, which depends on the abandonment indicators $\mathds{1}\{D_i>W_i\}$ and hence on the offered waiting times themselves.

We impose the following assumptions.

\begin{assumption}\label{assumption}
\begin{enumerate}
\item Arrivals occur one at a time.
The arrival process $A(\cdot)$ is a stationary renewal process with rate $\lambda$.
Let $U$ denote an interarrival time, assume $\E[U^2]<\infty$, and set $c_a^2\triangleq\lambda^2\Var(U)\in(0,\infty)$.
Let $I_a(t)$ denote the index of dispersion for counts (IDC):
\[
    I_a(t) \triangleq \frac{\Var(A(t))}{\lambda t}.
\]
Define the rate-one IDC by $I_a^{(1)}(t)\triangleq I_a(t/\lambda)$, which is the IDC of the stationary renewal process with interarrival law $\lambda U$.
Equivalently, $I_a(t)=I_a^{(1)}(\lambda t)$.
Assume that $I_a^{(1)}(t)\to c_a^2$ as $t\to\infty$, and that $\sup_{t\ge0}I_a^{(1)}(t)<\infty$.
\item The service times are i.i.d.\ with mean $1/\mu$, and finite squared coefficient of variation $c_s^2>0$.
\item The patience times are i.i.d.\ and admit the scaling representation $D=\alpha^{-1}\tilde D$, where $\tilde D$ is a nonnegative random variable with $\E[\tilde D]=1$.
Under this scaling convention, the CDF of $D$ is $F_\alpha(t)=F(\alpha t)$, where $F$ denotes the cumulative distribution function of $\tilde D$.
We assume $F(0)=0$, $F$ is twice continuously differentiable on $[0,\infty)$, has support $[0,\infty)$, and its density $f=F'$ is bounded.
\item The arrival process, service times, and patience times are mutually independent.
\end{enumerate}
Throughout this paper, we set $c_x^2\triangleq c_a^2+c_s^2$.
\end{assumption}

\begin{remark}[Stationary version]\label{rmk:stationary}
The identity \eqref{eq:reverse_ab} is a finite-time identity for the system initialized empty.
The steady-state quantities used below are obtained by passing to the associated stationary version of the queue with impatience.
This is justified by the standard stability theory for $GI/GI/1{+}GI$ queues.
In particular, \citet[Section~2.2]{ward2005diffusion} state that a sufficient stability condition for the offered-waiting-time process in the reneging model to have a nondegenerate limiting distribution is $\rho \mathbb P(D=\infty)<1$, and attribute the offered-waiting-time result to \citet[Lemma~2]{baccelli1984single}.
Under Assumption~\ref{assumption}, patience times are finite almost surely, so this condition is automatically satisfied for every finite traffic intensity $\rho$.

For $u<t$, define the stationary effective input and net-input increments by
\[
    Y^{\mathrm{st}}(u,t]\triangleq \sum_{i:T_i\in(u,t]}V_i\mathds{1}\{D_i>Z^{\mathrm{st}}(T_i-)\},
    \qquad N^{\mathrm{st}}(u,t]\triangleq Y^{\mathrm{st}}(u,t]-(t-u),
\]
where $Z^{\mathrm{st}}(t)$ is the stationary virtual waiting time.
Applying \eqref{eq:reverse_ab} to a system started empty at time $-r$ and sending $r\to \infty$ along the stationary construction gives the reverse-time representation
\[
    Z^{\mathrm{st}}(t)=\sup_{s\ge0}N^{\mathrm{st}}(t-s,t],\qquad t\in \R.
\]
To lighten notation, when steady-state quantities are considered we write $Z(t)$ for $Z^{\mathrm{st}}(t)$ and $N(t)-N(t-s)$ for $N^{\mathrm{st}}(t-s,t]$.
For each finite horizon, the stationary increment means and variances used below are finite because the thinning indicators are bounded, service times have finite second moment, and $A(t)-A(u)$ has finite variance under Assumption~\ref{assumption}.
\end{remark}

In stationarity, the reverse-time representation motivates the generic RQ approximation
\begin{equation}\label{eq:RQ_raw_formulation}
    Z_{\mathrm{RQ}} = \sup_{s\ge0} \left\{ m(s)+b\sqrt{v(s)} \right\},
\end{equation}
where $m(s) = \E[N(t)-N(t-s)]$ and $v(s) = \Var(N(t)-N(t-s))$.
By stationarity, these terms do not depend on $t$.

\subsection{Road Map and Status of the RQ Approximations}
\label{sec:rq_roadmap}

The construction follows the same four-step logic for both algorithms.
First, start from the exact stationary reverse-time representation.
Second, for a deterministic trial approximation $z$ to the mean stationary virtual waiting time, replace the endogenous increment mean and variance by surrogates $m_z(s)$ and $v_z(s)$.
Third, maximize the deterministic envelope $m_z(s)+b\sqrt{v_z(s)}$ over the look-back horizon $s$.
Fourth, impose self-consistency by setting the resulting supremum equal to the trial value $z$.
The result is a scalar fixed point.

Table~\ref{tab:rq_roadmap} separates exact identities, approximation steps, asymptotic guidance, and numerical implementation.

\begin{table}[!ht]
\centering
\caption{End-to-end RQ construction and theoretical status.}
\label{tab:rq_roadmap}
\footnotesize
\renewcommand{\arraystretch}{1.14}
\setlength{\tabcolsep}{4pt}
\begin{tabularx}{\textwidth}
{@{}L{0.14\textwidth}L{0.25\textwidth}YL{0.18\textwidth}@{}}
\toprule
Stage & Object or input & Status & Location \\
\midrule

Inputs and output
&
$\lambda$, $\mu$, $c_s^2$, and $F_\alpha$; the first RQ uses the rate-one work IDW $I_w$, while the refined RQ uses the arrival IDC $I_a$, its limit $c_a^2$, $\alpha$, $k$, $\beta$, and precomputed tables for $w_{\tilde c_\alpha,k}$ and $b$.
The output is an approximation to the mean stationary virtual waiting time.
&
\emph{Model input/offline preprocessing.}
The variance-reduction function is tabulated before the scalar RQ solve.
&
Assumptions~\ref{assumption}--\ref{assumption:F};
Section~\ref{sec:compute_w};
Equations~\eqref{eq:RQ_ab_1} and~\eqref{eq:RQ_ab_2}.
\\
\addlinespace

Effective drift
&
Poisson-surrogate $m_z$.
&
\emph{Exact and approximate components.}
The Poisson compensator and renewal Palm identity are exact.
Omitting the Palm term and imposing self-consistency through the trial value $z$ are approximation steps; the mean of the omitted term is controlled on the RQ optimizer scales.
&
Section~\ref{sec:drift_approx};
Lemmas~\ref{lm:drift_poisson}--\ref{lm:correction};
Proposition~\ref{prop:palm_optimizer_scales}.
\\
\addlinespace

First variance
&
$v_z^{(1)}$ from $\lambda$, $\mu$, $F_\alpha$, and the rate-one work IDW $I_w$.
&
\emph{Approximation.}
Deterministic time-change renewal-reward surrogate used in the first RQ. This serves as a baseline method.
&
Section~\ref{sec:RQ_ab1_algo};
Equations~\eqref{eq:m_v1_def}--\eqref{eq:RQ_ab_1}.
\\
\addlinespace

Refined variance
&
$v_z^{(2)}$ from $\lambda$, $\mu$, $c_s^2$, $F_\alpha$, $I_a$, $c_a^2$, $\alpha$, $k$, $\beta$, and $w_{\tilde c_\alpha,k}$.
&
\emph{Asymptotically guided heuristic.}
The surrogate interpolates the proved fixed-horizon long-patience and heavy-traffic variance limits and is used in the numerical experiments.
&
Theorem~\ref{thm:HT_limit};
Sections~\ref{sec:def_wck}, \ref{sec:HT_var},
and~\ref{sec:RQ_ab2_algo};
Corollary~\ref{thm:HT_var};
Lemma~\ref{lm:var_lim};
Equations~\eqref{eq:IDW_effective_approx}--\eqref{eq:RQ_ab_2}.
\\
\addlinespace

Calibration and solve
&
$b$ and the selected pair $(m_z,v_z^{(j)})$.
&
\emph{Asymptotic calibration and numerical solution.}
Critical heavy-traffic matching calibrates $b$; the underloaded limit gives $b=\sqrt{2}$, and the overloaded leading-order limit is insensitive to $b$.
The refined implementation uses offline tabulation and interpolation based on canonical $M/M/1{+}M$ and $M/M/1{+}E_k$ models. For fixed $b$, the unique scalar fixed point is solved by bisection.
&
Sections~\ref{sec:calibrate_first_RQ}
and~\ref{sec:calibrate_refined_RQ};
Equations~\eqref{eq:RQ_ab_1} and~\eqref{eq:RQ_ab_2}.
\\
\bottomrule
\end{tabularx}
\end{table}

The tandem construction and the additional performance-measure formulas are heuristic extensions rather than stages of the main RQ algorithm; see Appendices~\ref{sec:appendix_tandem}
and~\ref{sec:heuristic}, respectively.

\section{The Drift of the Effective Net-Input Process}\label{sec:drift_approx}

To evaluate the mean of the effective net-input process, it is convenient to exploit martingale representations of counting processes.
For the Poisson-arrival argument below, let $\mathbb F=\{\mathcal F_t\}_{t\ge0}$ be the usual augmentation of the natural filtration generated by the arrival process and the service- and patience-time marks revealed at their arrival epochs.
By the Doob--Meyer decomposition \citep[Theorem~10.5]{kallenberg2021foundations}, $A(t)-\Lambda(t)$ is an $\mathbb F$-martingale, where $\Lambda(\cdot)$ is the (predictable) compensator of $A(\cdot)$.
This representation is particularly tractable when $A(\cdot)$ is a homogeneous Poisson process, in which case $\Lambda(t)=\lambda t$.
We begin with this Poisson setting.

\subsection{Queues with Poisson arrivals}

Let $\bar F_{\alpha}(t)\triangleq 1-F_{\alpha}(t)$ denote the complementary CDF of the patience time.
Recall the effective arrival process $A_0(\cdot)$ defined in \eqref{eq:effective_arrival}, and the effective net-input process $N(t)$ defined in \eqref{eq:effective_net_input}.
In particular, $Z(t-)$ is $\mathbb F$-predictable.

Under Poisson arrivals, consider an arrival occurring at time $u$.
Conditional on the pre-arrival history $\mathcal F_{u-}$, the offered waiting time $Z(u-)$ is known, while the patience time $D$ is independent of $\mathcal F_{u-}$.
Therefore, $\mathbb P\left(D>Z(u-)\mid \mathcal F_{u-}\right)=\bar F_{\alpha} \bigl(Z(u-)\bigr)$.
In other words, relative to the filtration $\mathbb F$, the process $A_0(\cdot)$ is obtained from the Poisson arrivals via \emph{predictable thinning} with retention probability $\bar F_{\alpha}(Z(u-))$.
Consequently, $A_0(\cdot)$ has $\mathbb F$-intensity $\lambda \bar F_{\alpha}(Z(u-))$, and its compensator is $\Lambda_0(t) = \lambda\int_0^t \bar F_{\alpha} \bigl(Z(u-)\bigr) du$.
Equivalently, $M_0(t)\triangleq A_0(t)-\Lambda_0(t)$ is an $\mathbb F$-martingale.

For the effective work process, conditional on $\mathcal F_{u-}$, the service time mark $V$ is independent of $\mathcal F_{u-}$ and independent of $D$, so $ \E \left[V \mathds{1}\{D>Z(u-)\}\mid \mathcal F_{u-}\right] =\E[V] \bar F_{\alpha} \bigl(Z(u-)\bigr)=\frac{1}{\mu}\bar F_{\alpha} \bigl(Z(u-)\bigr). $

\begin{lemma}\label{lm:drift_poisson}
Under Assumption~\ref{assumption}, suppose that $A(\cdot)$ is a homogeneous Poisson process with rate $\lambda$.
Then, for any $0\le s\le t$,
\begin{align*}
    \E\left[A_0(t)-A_0(t-s)\right] &= \lambda \E\left[\int_{t-s}^{t}\bar F_{\alpha}\bigl(Z(u-)\bigr) du\right],\\
    \E\left[N(t)-N(t-s)\right]
    &= \frac{\lambda}{\mu} \E\left[\int_{t-s}^{t}\bar F_{\alpha}\bigl(Z(u-)\bigr) du\right] - s.
\end{align*}
Moreover, if $\{Z(u)\}_{u\in \R}$ is strictly stationary, then for all $0\le s\le t$,
\begin{align*}
    \E\left[A_0(t)-A_0(t-s)\right] &= \lambda s \E\left[\bar F_{\alpha}\bigl(Z(0)\bigr)\right],\\
    \E\left[N(t)-N(t-s)\right] &= \left(\frac{\lambda}{\mu}\E\left[\bar F_{\alpha}\bigl(Z(0)\bigr)\right]-1\right)s.
\end{align*}
\end{lemma}

\subsection{Queues with a General Renewal Arrival Process}

Motivated by Lemma~\ref{lm:drift_poisson}, we introduce the Poisson surrogate
\begin{equation}\label{eq:Lambda}
    \Lambda_t(s) \triangleq \lambda \int_{t-s}^{t} \bar F_{\alpha}\bigl(Z(u-)\bigr) du,
\end{equation}
which is the compensator increment of the effective arrival process in the Poisson case.
For a general renewal arrival process, the corresponding effective-arrival integral admits the exact decomposition
\[
    \int_{t-s}^{t} \bar F_{\alpha}\bigl(Z(u-)\bigr)dA(u) = \Lambda_t(s) + \delta_t(s),
    \quad \text{where} \quad \delta_t(s)
    \triangleq \int_{t-s}^{t} \bar F_{\alpha}\bigl(Z(u-)\bigr)d\bigl(A(u)-\lambda u\bigr).
\]
When $A(\cdot)$ is Poisson, $A(u)-\lambda u$ is a martingale and the predictability of $Z(u-)$ implies $\E[\delta_t(s)]=0$, recovering Lemma~\ref{lm:drift_poisson}.
For a general renewal process, $A(u)-\lambda u$ is not a martingale, and the correction term need not have mean zero.

\subsubsection{Stationary Palm Correction}

The RQ fixed point uses stationary mean increments, so define
\[
    \delta_\alpha(s)\triangleq\int_{(0,s]}\bar F_\alpha\bigl(Z(u-)\bigr)d\bigl(A(u)-\lambda u\bigr),\qquad s\ge0,
\]
where $Z(\cdot)$ is the stationary offered-waiting-time process.
Let $\mathbb P^0$ and $\E^0$ denote the arrival Palm law and expectation.
Under $\mathbb P^0$, let $U$ be the next interarrival time and let $Z_{\mathrm{pre}}$ and $Z_{\mathrm{post}}$ be the offered waiting times immediately before and after the arrival at time $0$.
Thus $Z_{\mathrm{post}}=Z_{\mathrm{pre}}+V_0\mathds{1}\{D_0>Z_{\mathrm{pre}}\}$.
At a deterministic time under $\mathbb P$, the stationary workload is denoted by $Z(0)$.
The identities used below are taken from \cite{bremaud1993stationary}.

\begin{lemma}[Palm form of the renewal correction]\label{lm:correction}
Under Assumption~\ref{assumption}, for every $s\ge0$,
\begin{equation}\label{eq:palm_delta}
    \E[\delta_\alpha(s)]=\lambda s\Delta_\alpha,
    \qquad
    \Delta_\alpha\triangleq\E^0[\bar F_\alpha(Z_{\mathrm{pre}})]-\E[\bar F_\alpha(Z(0))].
\end{equation}
Moreover,
\begin{equation}\label{eq:Delta_alpha_density}
    \Delta_\alpha
    =\alpha\E^0\left[\int_0^{U\wedge Z_{\mathrm{post}}}\{1-\lambda(U-v)\}f\bigl(\alpha(Z_{\mathrm{post}}-v)\bigr)dv\right],
\end{equation}
and
\begin{equation}\label{eq:palm_linear_bound}
    |\E[\delta_\alpha(s)]|\le\frac{3+c_a^2}{2}\|f\|_\infty\alpha s,
    \qquad s\ge0.
\end{equation}
\end{lemma}

Lemma~\ref{lm:correction} identifies the renewal-arrival drift correction as a Palm/time-average discrepancy.
The exact stationary effective net-input drift is
\[
    \E[N(s)-N(0)]=\{\rho\E[\bar F_\alpha(Z(0))]-1\}s+\rho s\Delta_\alpha.
\]
The Poisson-surrogate drift approximation neglects $\rho s\Delta_\alpha$:
\begin{equation}\label{eq:mean_drift}
    \E[N(s)-N(0)]\approx\{\rho\E[\bar F_\alpha(Z(0))]-1\}s.
\end{equation}
Equivalently, the surrogate effective-arrival mean is $\lambda s\E[\bar F_\alpha(Z(0))]$.

\begin{remark}[Control of the mean Palm correction]\label{rem:control_palm}
The bound \eqref{eq:palm_linear_bound} is uniform over the heavy-traffic family because $c_a^2$ is fixed.
Since the stationary RQ formulation takes a supremum over $s\ge0$, a linear bound alone does not control growing optimizing horizons.
Proposition~\ref{prop:palm_optimizer_scales} verifies negligibility after the relevant horizon scales have been identified.
\end{remark}

\section{A First Robust Queueing Algorithm}\label{sec:RQ_ab1}

We now propose our first approximation for the variance of the effective net-input increment $N(t)-N(t-s)$ based on a (deterministic) time-change of the renewal arrival process.
Combining this variance approximation with the drift approximation from Section~\ref{sec:drift_approx} yields our first RQ algorithm for the virtual waiting time.

Recall that the effective net-input process is obtained by state-dependent thinning of renewal rewards: an arrival contributes service work if and only if its patience time exceeds the offered waiting time at arrival.
Although patience times are independent of the system primitives, the thinning decision $\mathds{1}\{D_i>W_i\}$ is correlated with the arrival process through the offered waiting time $W_i=Z(T_i-)$.
As an approximation, we treat this correlation as negligible and model the effective input over $(t-s,t]$ via a stationary renewal reward surrogate evaluated at a deterministically rescaled time.

\subsection{The First RQ Algorithm}\label{sec:RQ_ab1_algo}

Let $\tilde A(\cdot)$ denote a \emph{rate-one} version of the renewal arrival process, i.e., if $A(\cdot)$ has i.i.d.\ interarrival times $U$ with $\E[U]=1/\lambda$, then $\tilde A(\cdot)$ is the renewal counting process with interarrival times $\lambda U$ and $\E[\tilde A(t)]=t$.
Motivated by the Poisson case and the drift surrogate $\Lambda_t(s)$ in \eqref{eq:Lambda}, we approximate the effective net-input increment by
\[
    N(t)-N(t-s) = \sum_{i=A(t-s)+1}^{A(t)} V_i\mathds{1}\{D_i>W_i\} - s
    \approx \sum_{i=1}^{\tilde A(\Lambda_t(s))} V_i - s,
\]
where $\Lambda_t(s)=\lambda\int_{t-s}^{t}\bar F_\alpha(Z(u-))du$ is defined in \eqref{eq:Lambda}.
For every deterministic $\ell\ge0$,
\begin{equation}\label{eq:renewal_reward_variance}
    \Var\left(\sum_{i=1}^{\tilde A(\ell)}V_i\right)
    =\frac{\ell}{\mu^2}I_w(\ell),
\end{equation}
where $I_w(\cdot)$ is the IDW associated with rate-one $\tilde A$ and the service times $\{V_i\}$.
Since $\tilde A(t)\stackrel{d}{=}A(t/\lambda)$ and the service times are independent of the arrival process, conditioning on $\tilde A(t)$ gives $I_w(t)=I_a^{(1)}(t)+c_s^2$.
In particular, under Assumption~\ref{assumption} the IDW is well defined, continuous, and bounded, with $I_w(\infty)=\lim_{t\to\infty}I_w(t)=c_x^2\in(0,\infty)$.
We use \eqref{eq:renewal_reward_variance} at the random clock
$\ell=\Lambda_t(s)$ as a conditional variance surrogate:
\begin{equation}\label{eq:approx_var_RQ1}
    \Var\bigl(N(t)-N(t-s)\mid\Lambda_t(s)\bigr)
    \approx
    \frac{\Lambda_t(s)}{\mu^2}
    I_w\bigl(\Lambda_t(s)\bigr).
\end{equation}

Note that $\Lambda_t(s)$ defined in \eqref{eq:Lambda} depends on the state process $Z(\cdot)$, so we do not yet have a deterministic RQ approximation.
The remaining approximation is a stationary self-consistency step.
Let $z$ denote a deterministic trial approximation to the mean stationary virtual waiting time.
Over a reverse-time interval of length $s$, we approximate the abandonment probability $\bar F_\alpha(Z(u-))$ by the constant $\bar F_\alpha(z)$.
This replaces the random effective-arrival mean $\Lambda_t(s) = \lambda \int_{t-s}^{t} \bar F_{\alpha}\bigl(Z(u-)\bigr) du$ by $\lambda \bar F_\alpha(z)s$.

Combining the drift approximation \eqref{eq:mean_drift} with the variance surrogate \eqref{eq:approx_var_RQ1}, the RQ surrogate for the stationary reverse-time increment $N(t)-N(t-s)$, conditional on the trial value $z$, is $m_z(s)+b\sqrt{v^{(1)}_z(s)}$, where
\begin{equation}\label{eq:m_v1_def}
    m_z(s) \triangleq \bigl(\rho\bar F_\alpha(z)-1\bigr)s,
    \qquad v^{(1)}_z(s) \triangleq \frac{\lambda \bar F_\alpha(z)s}{\mu^2} I_w\bigl(\lambda \bar F_\alpha(z)s\bigr).
\end{equation}
Substituting this deterministic surrogate into the stationary RQ supremum representation \eqref{eq:RQ_raw_formulation} gives the first stationary RQ fixed-point equation
\begin{align}
    Z_{\mathrm{RQ}_1} & = \sup_{s\ge 0} \left\{ m_{Z_{\mathrm{RQ}_1}}(s)+b\sqrt{v^{(1)}_{Z_{\mathrm{RQ}_1}}(s)} \right\}
    = \sup_{u\ge 0} \left\{ \rho u - \frac{u}{\bar F_\alpha(Z_{\mathrm{RQ}_1})} + b\sqrt{\frac{\rho u}{\mu} I_w(\lambda u)} \right\}, \label{eq:RQ_ab_1}
\end{align}
where the second equality uses the change of variables $u=\bar F_\alpha(Z_{\mathrm{RQ}_1})s$, so that $\lambda\bar F_\alpha(Z_{\mathrm{RQ}_1})s=\lambda u$.
This equation defines the steady-state RQ approximation directly as a scalar fixed point.

Define the mapping
\[
    \Psi(z)\triangleq \sup_{u\ge 0} \left\{ \rho u - \frac{u}{\bar F_\alpha(z)} + b \sqrt{\frac{\rho u}{\mu} I_w(\lambda u)} \right\}.
\]

\begin{remark}[Existence of the first RQ fixed point]\label{rem:RQ1_existence}
Assumption~\ref{assumption} provides a simple sufficient condition for the existence of a unique solution to \eqref{eq:RQ_ab_1}.
In particular, $z\mapsto \bar F_\alpha(z)$ is continuous, strictly decreasing, and satisfies $\bar F_\alpha(z)\downarrow0$ as $z\to \infty$.
Furthermore, $I_w$ is continuous, bounded above, and satisfies $c_x^2=\lim_{t\to \infty}I_w(t)\in(0,\infty)$.
For any $z$ such that $\rho\bar F_\alpha(z)<1$, the objective defining $\Psi(z)$ is bounded above by
\[
    -\left(\frac{1}{\bar F_\alpha(z)}-\rho\right)u + b\sqrt{\frac{\rho\|I_w\|_\infty}{\mu}u},
\]
and hence $\Psi(z)<\infty$.
Thus the fixed-point equation can only have a finite solution in the region $\rho\bar F_\alpha(z)<1$.
On any compact subset of this region, the negative linear coefficient $1/\bar F_\alpha(z)-\rho$ is bounded away from zero, so the maximizer in the definition of $\Psi(z)$ is localized to a compact interval in $u$, uniformly over $z$ in that subset.
It follows that $\Psi$ is continuous on the region $\rho\bar F_\alpha(z)<1$.
Moreover, since $\bar F_\alpha(z)\downarrow0$, the same upper bound implies $\Psi(z)\to 0$ as $z\to \infty$.
If $\rho\bar F_\alpha(0)<1$, then $\Psi(0)\ge0$, while $\Psi(z)-z\to -\infty$ as $z\to \infty$, so continuity gives a solution.
If $\rho\bar F_\alpha(0)\ge1$, let $z_\rho$ be the unique value satisfying $\rho\bar F_\alpha(z_\rho)=1$.
As $z\downarrow z_\rho$, the negative linear coefficient $1/\bar F_\alpha(z)-\rho$ tends to zero, and the condition $c_x^2>0$ implies $\Psi(z)\to \infty$.
Since $\Psi(z)-z\to -\infty$ as $z\to \infty$, continuity again gives a solution.
Finally, because $z\mapsto\Psi(z)$ is nonincreasing, $z\mapsto\Psi(z)-z$ is strictly decreasing.
Therefore, under Assumption~\ref{assumption}, \eqref{eq:RQ_ab_1} admits a unique finite solution.
\end{remark}

In practice, the unique solution $Z_{\mathrm{RQ}_1}=Z_{\mathrm{RQ}_1}(b;\lambda,\mu,F_\alpha,I_w)$ can be computed efficiently by bisection.
We discuss calibration of $b$ in Section~\ref{sec:calibrate_first_RQ}.

\subsection{Heavy-Traffic Limit for the First RQ Algorithm}

The choice of the robustness parameter $b$ is central to the accuracy of the RQ approximation.
To motivate our calibration of $b$, we establish heavy-traffic limits for the steady-state RQ fixed point in \eqref{eq:RQ_ab_1} and compare these limits with the corresponding heavy-traffic asymptotics for the canonical $M/M/1{+}GI$ model.
A key theme is how the scaling of the RQ solution (and, by comparison, the mean offered waiting time) depends on the local behavior of the patience-time distribution near the origin.
Recall from Assumption~\ref{assumption} that the patience-time scaling is $F_\alpha(t)=F(\alpha t)$ and $\bar F_\alpha(t)=1-F_\alpha(t)=\bar F(\alpha t)$, where $F$ is the CDF of $\tilde D$ with mean $1$.

To formalize the relevant local behavior of $F$ at $0$, we impose the following regularity assumption.

\begin{assumption}\label{assumption:F}
There exists an integer $k=k(F)\ge 1$ such that $F$ is $k$ times continuously differentiable on $[0,\infty)$ and
\[
    F^{(j)}(0)=0 \ \text{for } j=0,1,\dots,k-1, \qquad F^{(k)}(0)>0,
\]
where $F^{(j)}$ denotes the $j$th derivative.
Setting $\beta\triangleq F^{(k)}(0)/k! > 0$, then
\[
    F(x)=\beta x^k + o(x^k), \qquad x\downarrow 0.
\]
\end{assumption}

We consider the long-patience--heavy-traffic limit indexed by the abandonment rate $\alpha\downarrow0$.
The service rate $\mu$ and the standardized service-time law of $\mu V^\alpha$ are fixed across the family.
The arrival rate is $\lambda_\alpha=\rho_\alpha\mu$, and the standardized interarrival law of $\lambda_\alpha U^\alpha$ is also fixed across the family.
Consequently, $c_a^2$, $c_s^2$, $I_a^{(1)}$, and $I_w=I_a^{(1)}+c_s^2$ do not depend on $\alpha$.
We assume $\alpha^{-\gamma}(\rho_\alpha-1)\to c$ for constants $\gamma>0$ and $c\in\R$.
Thus $\rho_\alpha\to1$ as $\alpha\downarrow0$.
Define the threshold
\begin{equation}\label{eq:threshold}
    h=h(F)\triangleq \frac{k}{k+1}.
\end{equation}

In this regime, the RQ solution exhibits three distinct scalings, which are determined by the threshold $h$: (i) \textbf{Underloaded}: if $c<0$ and $\gamma<h$, abandonment becomes asymptotically negligible and $Z_{\mathrm{RQ}_1}$ scales as $(1-\rho)^{-1}$, consistent with the expected steady-state workload of a $GI/GI/1$ queue without abandonment. (ii) \textbf{Critically loaded:} if $\gamma\ge h$, then abandonment enters the solution and $Z_{\mathrm{RQ}_1}$ scales as $\alpha^{-h}$.
This corresponds to the case where refined diffusion models are required, e.g. ROU process \cite{ward2005diffusion} when $F'(0) \neq 0$ and hazard rate scaling \cite{reed2008approximating} when $F'(0) = 0$. (iii) \textbf{Overloaded:} if $c>0$ and $\gamma<h$, then $Z_{\mathrm{RQ}_1}$ grows faster, on the scale $\alpha^{-(1-\gamma/k)}$.
The case with $\gamma = 0$ corresponds to the overloaded queue studied in \cite{jennings2012overloaded}.

These scalings are summarized below.
A unified proof of the first and refined RQ heavy-traffic limits is given in Section~\ref{sec:proof_RQ_HT}.

\begin{theorem}[Heavy-traffic limit for first RQ]\label{Thm:RQ_HT}
Consider the heavy-traffic family described above under Assumptions~\ref{assumption} and~\ref{assumption:F}.
Fix $\mu>0$ and $b>0$, and let $\lambda_\alpha=\rho_\alpha\mu$ with $\alpha^{-\gamma}(\rho_\alpha-1)\to c$ for some $\gamma>0$ and $c\in \R$.
Recall $c_x^2$ from Assumption~\ref{assumption} (equivalently, $c_x^2=I_w(\infty)<\infty$).
For each $\alpha>0$, let $Z_{\mathrm{RQ}_1,b}^\alpha$ denote the RQ solution of \eqref{eq:RQ_ab_1}.
\begin{enumerate}
\item \textbf{(Underloaded)} If $c<0$ and $\gamma<h$, then
\[
    \lim_{\alpha\downarrow 0} (1-\rho_{\alpha}) Z_{\mathrm{RQ}_1,b}^\alpha
    = \frac{1}{\mu}\cdot\frac{b^2}{2}\cdot\frac{c_x^2}{2},
    \qquad\text{equivalently} \qquad \lim_{\alpha\downarrow 0} (-c)\mu\alpha^\gamma Z_{\mathrm{RQ}_1,b}^\alpha
    =\frac{b^2}{2}\cdot\frac{c_x^2}{2}.
\]
Moreover,
\[
    \lim_{\alpha\downarrow 0}\frac{Z_{\mathrm{RQ}_1,b}^\alpha}{\E[Z_{M/M/1}]} =\frac{b^2}{2}\cdot\frac{c_x^2}{2},
    \qquad \E[Z_{M/M/1}] \triangleq \frac{\rho_{\alpha}}{\mu(1-\rho_{\alpha})},
\]
where $\E[Z_{M/M/1}]$ is the mean steady-state workload of an $M/M/1$ queue with arrival rate $\lambda_\alpha$ and service rate $\mu$ (without abandonment).

\item \textbf{(Critically loaded)} If $\gamma\ge h$, then there exists a finite constant $\hat Z_{\mathrm{RQ}_1,b}>0$ such that
\[
    \lim_{\alpha\downarrow 0}\alpha^{h}Z_{\mathrm{RQ}_1,b}^\alpha=\hat Z_{\mathrm{RQ}_1,b}.
\]
Furthermore, $\hat Z_{\mathrm{RQ}_1,b}$ is the unique positive root of
\begin{equation}\label{eq:RQ_HT_equation}
    -\mathds{1}\{\gamma=h\} c \hat Z_{\mathrm{RQ}_1,b} + \beta \hat Z_{\mathrm{RQ}_1,b}^{k+1} =\frac{c_x^2 b^2}{4\mu}.
\end{equation}

\item \textbf{(Overloaded)} If $c>0$ and $\gamma<h$, then
\[
    \lim_{\alpha\downarrow 0}\alpha^{1-\gamma/k}Z_{\mathrm{RQ}_1,b}^\alpha =\left(\frac{c}{\beta}\right)^{1/k}.
\]
In particular, the leading-order limit is independent of $b$ and of the work-variability parameter $c_x^2$, and depends on the patience-time distribution only through $F^{(k)}(0)$.
Equivalently,
\[
    \lim_{\alpha\downarrow 0}\frac{\rho_{\alpha}}{\rho_{\alpha}-1} F\bigl(\alpha Z_{\mathrm{RQ}_1,b}^\alpha\bigr)=1.
\]
\end{enumerate}
\end{theorem}

Part~(1) of Theorem~\ref{Thm:RQ_HT} identifies the parameter range in which patience times are asymptotically long relative to the load gap, so that abandonment becomes negligible and the system behaves as a single-server queue without abandonment.
Part~(2) of Theorem~\ref{Thm:RQ_HT} characterizes the regime in which the patience-time distribution influences the heavy-traffic scaling.
In the canonical Markovian case $M/M/1{+}M$, one has $k=1$ (so $h=1/2$) and $F'(0)\neq 0$.
Diffusion limits for queues with abandonment were established for the Markovian model in \citet{ward2003diffusion} and generalized to $GI/GI/1{+}GI$ in \citet{ward2005diffusion}.
We restate the relevant result below.
Let $Z^\alpha(\cdot)$ denote the (steady-state) virtual waiting time process under abandonment scaling parameter $\alpha$, and define the diffusion-scaled process $\tilde Z^\alpha(t)\triangleq \alpha^{1/2} Z^\alpha(\alpha^{-1}t)$.

\begin{proposition}[Theorem 1, \citet{ward2005diffusion}]\label{Prop:WG05}
Suppose $\alpha^{-1/2}(\rho_{\alpha}-1)\to c$ for some finite constant $c$, and assume $\tilde Z^\alpha(0)\Rightarrow \tilde Z(0)$ as $\alpha\downarrow 0$.
Then $\tilde Z^\alpha \Rightarrow \tilde Z$ as $\alpha\downarrow 0$, where $\tilde Z$ is a ROU process with drift $c-F'(0)z$ and infinitesimal variance $c_x^2/\mu$ with $c_x^2=c_a^2+c_s^2$.
If $F'(0)>0$, the ROU process has a unique stationary distribution, which is the law of a normal random variable truncated to $[0,\infty)$
\[
    \tilde Z(\infty)\ \stackrel{d}{=}\ \mathcal{N}\left(\frac{c}{F'(0)}, \frac{c_x^2}{2\mu F'(0)}\right)\ \bigg|\ \left\{\mathcal{N}\left(\frac{c}{F'(0)}, \frac{c_x^2}{2\mu F'(0)}\right)\ge 0\right\},
\]
where $\mathcal{N}(\mu,\sigma^2)$ is a normal random variable.
\end{proposition}

When $F'(0)>0$, the stationary distribution above is a truncated normal with mean
\begin{equation}\label{eq:ROU_expectation}
    \E[\tilde Z(\infty)]
    = \frac{c}{F'(0)} + \frac{\phi\left(-\frac{c}{F'(0)\sigma}\right)}{1-\Phi\left(-\frac{c}{F'(0)\sigma}\right)} \sigma,
    \qquad \sigma^2\triangleq \frac{c_x^2}{2\mu F'(0)},
\end{equation}
where $\phi$ and $\Phi$ are the standard normal density and distribution functions.
If $F'(0)=0$, the diffusion approximation in Proposition~\ref{Prop:WG05} degenerates to a reflected Brownian motion, where the patience-time distribution vanishes from the limit.
This fails to reveal the subtle scaling of the heavy-traffic limit when the system load is heavier than that in the canonical $\alpha^{1/2}$ scaling; see Theorem~\ref{Thm:HT_exact}.

The exact stationary mean formula for the $M/M/1{+}GI$ model is the single-server specialization of \citet[Eq.~(9.9)]{zeltyn2005call}.
The next theorem shows that the exact mean offered waiting time exhibits the same three scaling regimes as the RQ solution.

\begin{theorem}[Heavy-traffic limit for $M/M/1{+}GI$]\label{Thm:HT_exact}
Under Assumption \ref{assumption:F}, let $\E[Z_{\alpha}]$ be the mean stationary virtual waiting time in the $M/M/1{+}GI$ model with service rate $\mu$ and arrival rate $\lambda_{\alpha}=\rho_{\alpha}\mu$ such that $\alpha^{-\gamma}(\rho_{\alpha}-1)\to c$ for some $\gamma>0$ and $c\in \R$.
Let $h$ be defined by \eqref{eq:threshold}.
\begin{enumerate}
\item \textbf{(Underloaded)} If $c<0$ and $\gamma<h$, then
\[
    \lim_{\alpha\downarrow 0} (-c)\mu\alpha^\gamma \E[Z_{\alpha}]
    = \lim_{\alpha\downarrow 0}\frac{\E[Z_{\alpha}]}{\E[Z_{M/M/1}]} =1,
\]
where $\E[Z_{M/M/1}]=\rho_{\alpha}/(\mu(1-\rho_{\alpha}))$.

\item \textbf{(Critically loaded)} If $\gamma\ge h$, then
\begin{equation}\label{eq:HT_exact}
    \lim_{\alpha\downarrow 0}\alpha^{h}\E[Z_{\alpha}]
    = \frac{\int_0^{\infty} x\exp\left\{c\mu x \mathds{1}\{\gamma = h\} - \frac{\mu\beta}{k+1} x^{k+1}\right\}dx} {\int_0^{\infty} \exp\left\{c\mu x \mathds{1}\{\gamma = h\} - \frac{\mu\beta}{k+1} x^{k+1}\right\}dx}
    \triangleq z_{\mathrm{HT}}.
\end{equation}
In particular, $z_{\mathrm{HT}}=z_{\mathrm{HT}}(c;k,\beta,\mu)$ depends on the patience distribution only through $k$ and $\beta$.

\item \textbf{(Overloaded)} If $c>0$ and $\gamma<h$, then
\[
    \lim_{\alpha\downarrow 0}\alpha^{1-\gamma/k}\E[Z_{\alpha}] =\left(\frac{c}{\beta}\right)^{1/k}.
\]
\end{enumerate}
\end{theorem}

\subsection{Calibration of the parameter \texorpdfstring{$b$}{b}}\label{sec:calibrate_first_RQ}

We now discuss calibration of the robustness parameter $b$.
We follow a procedure similar to \citet{whitt2019time}: we select $b$ by matching the heavy-traffic limits of the RQ approximation in Theorem~\ref{Thm:RQ_HT} with the corresponding heavy-traffic limits of the original system in Theorem~\ref{Thm:HT_exact}, specialized to the $M/M/1{+}GI$ model.
Moreover, Part~(3) of Theorem~\ref{Thm:RQ_HT} shows that the steady-state RQ solution matches the exact leading-order constant in the overloaded heavy-traffic regime, irrespective of the choice of $b$.
Consequently, we work in the critically-loaded scaling $\gamma=h$ with $c \triangleq \alpha^{-h}(\rho-1)$.

Throughout this subsection we set $\mu=1$.
For an $M/M/1$ input, the long-run index of dispersion for work satisfies $c_x^2=I_w(\infty)=2$.
Matching the limiting constants in \eqref{eq:RQ_HT_equation} and \eqref{eq:HT_exact} by setting $\hat Z_{\mathrm{RQ}_1,b}=z_{\mathrm{HT}}$ therefore yields
\begin{equation}\label{eq:b}
    b(c) \triangleq \sqrt{2\left|-c z_{\mathrm{HT}}+\beta z_{\mathrm{HT}}^{k+1}\right|},
\end{equation}
where $z_{\mathrm{HT}}$ is the constant defined in \eqref{eq:HT_exact} with $\gamma=h$ and $\mu=1$.

The next lemma shows that this calibration automatically recovers the classical underloaded calibration $b=\sqrt{2}$ as a limiting case; see \cite{whitt2018using}.

\begin{lemma}\label{lm:b_to_sqrt2}
For any $k\ge 1$ and $F^{(k)}(0)>0$, we have $\lim_{c\to -\infty} b(c) = \sqrt{2}$.
\end{lemma}

\begin{proof}
As $c\to -\infty$, the integrals defining $z_{\mathrm{HT}}$ in \eqref{eq:HT_exact} are dominated by a neighborhood of $0$, and one obtains $z_{\mathrm{HT}}\sim -1/c$, so that $-cz_{\mathrm{HT}}\to 1$ and $z_{\mathrm{HT}}^{k+1}=o(1)$.
Substituting into \eqref{eq:b} yields $b(c)\to \sqrt{2}$.
\end{proof}

\begin{remark}[Universal calibration across all heavy-traffic regimes]
Recall that $c\to -\infty$ corresponds to the underloaded long-patience regime (part~(1) of Theorem~\ref{Thm:RQ_HT} and Theorem~\ref{Thm:HT_exact}), because if $\gamma<h$ then $\alpha^{-h}(\rho-1)\to -\infty$ as $\alpha\downarrow 0$.
Lemma~\ref{lm:b_to_sqrt2} therefore implies that the critically-loaded calibration \eqref{eq:b} subsumes the underloaded case as a special limit.
Moreover, in the overloaded regime the leading-order scaling is asymptotically insensitive to $b$ (Theorem~\ref{Thm:RQ_HT}(3)).
In summary, \eqref{eq:b} provides a single calibration rule that is consistent across all three heavy-traffic regimes.
\end{remark}

\begin{remark}[Closed form for the $M/M/1{+}M$ model]
For the canonical $M/M/1{+}M$ model with $\mu=1$, one has $k=1$, $h=1/2$, and $F'(0)=1$.
Writing $c=\alpha^{-1/2}(\rho_{\alpha}-1)$, the constant $z_{\mathrm{HT}}$ in \eqref{eq:HT_exact} equals the mean of a truncated normal distribution (see Proposition~\ref{Prop:WG05}), namely
\[
    z_{\mathrm{HT}} = c + \frac{\phi(-c)}{1-\Phi(-c)}.
\]
Substituting this value into \eqref{eq:b} yields the explicit calibration
\[
    b(c) = \sqrt{ 2\left(c + \frac{\phi(-c)}{1-\Phi(-c)}\right)\frac{\phi(-c)}{1-\Phi(-c)} }.
\]
\end{remark}

\section{A Refined Robust Queueing Algorithm}\label{sec:var}

The first RQ algorithm in Section~\ref{sec:RQ_ab1} uses the variance surrogate \eqref{eq:approx_var_RQ1}, which is motivated by the heuristic that the dependence between the thinning rule $\mathds{1}\{D_i>W_i\}$ and the arrival process is negligible.
In this section, we develop a refined approximation for the variance function of the effective net-input process based on a heavy-traffic limit.
The resulting limit suggests a more accurate structure for the variability term in the RQ formulation, and it serves as the basis for a refined RQ algorithm.

\subsection{Heavy-Traffic Limit for the Effective Net-Input Process}
Under the conventional diffusion scaling of Proposition~\ref{Prop:WG05}, the limit depends on the patience-time distribution only through $f(0)$ and degenerates to a reflected Brownian motion when $f(0)=0$.
Here we consider an alternative heavy-traffic regime under which the limiting diffusion has a \emph{nonlinear} drift determined by the first nonzero derivative of $F$ at the origin (Assumption~\ref{assumption:F}).
Unlike the hazard-rate scaling of \citet{reed2008approximating}, the limit retains only this local order of $F$; see Remark~\ref{rem:why_not_full_hazard}.

\paragraph{System sequence and scaling.}
Consider a sequence of $GI/GI/1{+}GI$ queues indexed by the patience scaling parameter $\alpha\downarrow 0$.
In the $\alpha$th system, the patience-time CDF is $F_\alpha(x)=F(\alpha x)$, where $F$ is a base CDF with $F(0)=0$ and finite mean, satisfying Assumption~\ref{assumption:F}.
Let $\lambda_{\alpha}$ and $\mu_{\alpha}$ denote the arrival and service rates, and write $\rho_{\alpha}\triangleq \lambda_{\alpha}/\mu_{\alpha}$.
Let $h=k/(k+1)$ be defined as in \eqref{eq:threshold}.
We assume $\mu_{\alpha}\to\mu$ and $\alpha^{-h}(\rho_{\alpha}-1)\to c$ as $\alpha\downarrow0$.
This is the critically-loaded regime with $\gamma=h$ in Theorem~\ref{Thm:RQ_HT}.
Recall $\beta$ from Assumption~\ref{assumption:F}.
Each system in the sequence satisfies the service-time assumptions with rate $\mu_{\alpha}$ and common service SCV $c_s^2$.
Hence $\Var(V_i^\alpha)=c_s^2/(\mu_{\alpha})^2$ and $\sup_\alpha \E[(V_i^\alpha)^2]<\infty$.

Let $A^\alpha(\cdot)$ be the arrival counting process, and let $Z^\alpha(\cdot)$ be the virtual waiting time process.
Let $S^\alpha$ denote the centered service-requirement partial-sum process indexed by customer count:
\[
    S^\alpha(t)\triangleq\sum_{i=1}^{\lfloor t\rfloor}\left(V_i^\alpha-\frac1{\mu_{\alpha}}\right),\qquad t\ge0.
\]
For customer $i$, let $T_i^\alpha$ be the arrival epoch and let $W_i^\alpha=Z^\alpha(T_i^\alpha-)$ be the offered waiting time.
Define the effective arrival and effective work-input processes as
\[
    A_0^\alpha(t)\triangleq \sum_{i=1}^{A^\alpha(t)} \mathds{1}\{D_i^\alpha>W_i^\alpha\},
    \qquad Y^\alpha(t)\triangleq \sum_{i=1}^{A^\alpha(t)} V_i^\alpha \mathds{1}\{D_i^\alpha>W_i^\alpha\}.
\]
We use the time scaling $t\mapsto\alpha^{-2h}t$ and space scaling $x\mapsto\alpha^hx$.
Define the fluid-scaled arrival process $\bar A^\alpha(t)\triangleq \alpha^{2h}A^\alpha(\alpha^{-2h}t)$ and the diffusion-scaled processes
\begin{align}
    \tilde A^\alpha(t) &\triangleq \alpha^h\left[A^\alpha(\alpha^{-2h}t)-\alpha^{-2h}\lambda_{\alpha} t\right],\notag\\
    \tilde S^\alpha(t) &\triangleq \alpha^h S^\alpha(\alpha^{-2h}t),\notag\\
    \tilde A_0^\alpha(t)
    &\triangleq \alpha^h\left[A_0^\alpha(\alpha^{-2h}t)-\alpha^{-2h}\lambda_{\alpha} t\right],\notag\\
    \tilde Y^\alpha(t)
    &\triangleq \alpha^h\left[Y^\alpha(\alpha^{-2h}t)-\alpha^{-2h}\rho_{\alpha} t\right],\label{eq:scaled_Y}\\
    \tilde Z^\alpha(t) &\triangleq \alpha^h Z^\alpha(\alpha^{-2h}t),\notag\\
    \tilde L^\alpha(t) &\triangleq \alpha^h L^\alpha(\alpha^{-2h}t),\notag
\end{align}
where $L^\alpha(\cdot)$ is the cumulative idle time in the identity
\[
    Z^\alpha(t)=Z^\alpha(0)+Y^\alpha(t)-t+L^\alpha(t),\qquad t\ge0.
\]
The prelimit regulator satisfies $Z^\alpha(t)\ge0$, $L^\alpha$ is nondecreasing, $L^\alpha(0)=0$, and $\int_0^\infty\mathds{1}\{Z^\alpha(t)>0\}dL^\alpha(t)=0$.

\paragraph{Heavy-traffic limit.}
Let $B_a$ and $B_s$ be independent standard Brownian motions and write $e(t)=t$ for the identity map.
Recall that $c_a^2$ and $c_s^2$ are the asymptotic variability parameters of the arrival and service primitives, and $c_x^2=c_a^2+c_s^2$.

\begin{theorem}\label{thm:HT_limit}
Assume the functional CLT
\begin{equation}\label{eq:FCLT_primitives}
    (\tilde A^\alpha,\tilde S^\alpha,\tilde Z^\alpha(0))
    \Rightarrow \left(c_a B_a\circ(\mu e),\mu^{-1}c_s B_s,Z^*(0)\right), \qquad \alpha\downarrow0.
\end{equation}
Assume that $Z^*(0)$ is independent of the future increments of $B_a$ and $B_s$.
Then
\[
    (\tilde Z^\alpha,\tilde L^\alpha,\tilde Y^\alpha,\tilde A_0^\alpha) \Rightarrow (Z^*,L^*,Y^*,A_0^*),
    \qquad \alpha\downarrow0,
\]
where $(Z^*,L^*)$ is the unique solution to the reflected integral equation
\begin{equation}\label{eq:HT_limit_Zstar}
    Z^*(t)=Z^*(0)+\mu^{-1}c_aB_a(\mu t)+\mu^{-1}c_sB_s(\mu t)-\beta\int_0^t\bigl(Z^*(s)\bigr)^kds+ct+L^*(t),
\end{equation}
with $Z^*(t)\ge0$, $L^*$ nondecreasing, $L^*(0)=0$, and $\int_0^\infty\mathds{1}\{Z^*(t)>0\}dL^*(t)=0$.
Moreover,
\begin{align*}
    A_0^*(t) &=c_aB_a(\mu t)-\mu\beta\int_0^t\bigl(Z^*(s)\bigr)^kds,\\
    Y^*(t) &=\frac1\mu A_0^*(t)+\frac1\mu c_sB_s(\mu t)
    = \mu^{-1}c_aB_a(\mu t)+\mu^{-1}c_sB_s(\mu t)-\beta\int_0^t\bigl(Z^*(s)\bigr)^kds.
\end{align*}
\end{theorem}

For Robust Queueing, we will ultimately use a stationary approximation for increments of the effective net-input process (see, e.g., \citet[Section~5.2]{whitt2018using}).
Accordingly, we assume henceforth that $Z^*(0)$ is distributed according to the unique stationary distribution of $Z^*$.

\begin{remark}
When $k=1$, Theorem~\ref{thm:HT_limit} reduces to the ROU limit in Proposition~\ref{Prop:WG05}.
For general $k\ge 1$, the limiting diffusion \eqref{eq:HT_limit_Zstar} has polynomial drift.
Its stationary density is given by
\begin{equation}\label{eq:stationary_dist}
    \pi_k(x) = \frac{1}{G_k}\exp\left\{\frac{2\mu}{c_x^2}\left( c x-\frac{\beta}{k+1}x^{k+1} \right)\right\}\mathds{1}\{x\ge 0\},
\end{equation}
where $G_k$ is the normalizing constant; see, e.g., \citet[Section~3]{browne1995piecewise}.
\end{remark}

\subsection{The Variance Function of the Stationary Heavy-Traffic Limit}

Recall that the heavy-traffic limit in Theorem~\ref{thm:HT_limit} is characterized by the parameter tuple
\[
    \Xi \triangleq (c,k,\mu,c_a^2,c_s^2,\beta).
\]
Let $(Z^*,L^*)$ be the stationary reflected diffusion in Theorem~\ref{thm:HT_limit}, and note that the limit total-input process is given by
\[
    Y^*(t) = Z^*(t)-Z^*(0)-ct-L^*(t).
\]
We define the variance function of the stationary heavy-traffic limit as
\[
    v(t;\Xi) \triangleq \Var\bigl(Y^*(t)\bigr) = \Var\bigl(Z^*(t)-Z^*(0)-L^*(t)\bigr),
\]
where $Z^*(0)$ has density \eqref{eq:stationary_dist}.
The current subsection develops $v(t;\Xi)$; Section~\ref{sec:HT_var} constructs the finite-system interpolation, and Section~\ref{sec:RQ_ab2_algo} inserts it into the refined RQ fixed point.

The central effect is negative feedback.
An upward input fluctuation first increases workload, but the resulting congestion raises subsequent abandonment and removes some future work.
The feedback is negligible over an infinitesimal horizon and accumulates over longer horizons.
We measure it by a dimensionless ratio $w_{c,k}(t)$: the numerator is effective-input variance in a normalized diffusion, and the denominator is the variance of its Brownian input without feedback.
Thus $w_{c,k}(t)=1$ means no variance reduction at horizon $t$, whereas a smaller value means stronger suppression by abandonment.

The construction proceeds in three steps.
We first reduce every parameter tuple $\Xi$ to a normalized base diffusion.
We then derive a response-function representation that separates future-noise and stationary-initial-state contributions.
Finally, we use that representation to establish qualitative properties and to compute a reusable offline table.

\subsubsection{The Base Diffusion and the Definition of \texorpdfstring{$w_{c,k}$}{wck}}\label{sec:def_wck}

Calculating the variance function $v(t;\Xi)$ for general model primitives is challenging, because it depends on the entire parameter tuple $\Xi$.
However, we now present a scaling argument, and show that all such variance functions can be expressed in terms of a normalized base model.
The base model is obtained by setting $\mu=c_a^2=c_s^2=\beta=1$ while changing $c$ and $k$.

For fixed $c\in \R$ and integer $k\ge1$, define the base reflected diffusion
\begin{equation}\label{eq:base_reflected_SDE}
    Z^{c,k}(t) = Z^{c,k}(0) + \sqrt{2}B(t) + \int_0^t \bigl(c-(Z^{c,k}(s))^k\bigr)ds + L^{c,k}(t), \qquad t\ge0.
\end{equation}
The process is reflected at zero, and we initialize it in stationarity.
Its stationary density is
\begin{equation}\label{eq:pi_ck}
    \pi_{c,k}(x)
    = \frac{\exp\left\{ cx-x^{k+1} / (k+1) \right\} }{ \int_0^\infty \exp\left\{ cy-y^{k+1} / (k+1) \right\}dy }\mathds{1}\{x\ge0\}.
\end{equation}
The associated effective-input fluctuation is
\begin{equation}\label{eq:base_Y_def}
    Y^{c,k}(t) = Z^{c,k}(t)-Z^{c,k}(0)-ct-L^{c,k}(t) = \sqrt{2}B(t)-\int_0^t (Z^{c,k}(s))^kds .
\end{equation}
The Brownian term $\sqrt{2}B(t)$ has variance $2t$.
The integral term is the abandonment-induced feedback that reduces the effective-input variability.
We therefore define
\begin{equation}\label{eq:w}
    v_{c,k}(t) \triangleq \Var(Y^{c,k}(t)), \qquad t\ge0, \qquad w_{c,k}(t) \triangleq \frac{v_{c,k}(t)}{2t}, \qquad t>0.
\end{equation}
For notational convenience at the origin, set $w_{c,k}(0)\triangleq1$.

The next lemma explains why it is enough to calculate $w_{c,k}$ for the base model.
All primitive distributions enter the heavy-traffic variance through a rescaling of the time argument and the load parameter $c$.

\begin{lemma}[Scaling representation of the heavy-traffic variance]\label{lm:var_expression}
Define
\[
    \tau \triangleq \left(\frac{c_x^2}{2\mu}\right)^{\frac{k-1}{k+1}} \beta^{\frac{2}{k+1}},
    \qquad \tilde c \triangleq c \left(\frac{c_x^2}{2\mu}\right)^{-\frac{k}{k+1}} \beta^{-\frac{1}{k+1}}.
\]
Then, for all $t\ge0$,
\[
    v(t;\Xi) = \frac{c_x^2}{\mu}t w_{\tilde c,k}(\tau t).
\]
\end{lemma}

Thus, once $w_{c,k}$ has been computed for a grid of base parameters $(c,k)$, the heavy-traffic variance function for any model primitives can be evaluated by rescaling $c$ and $t$.

\subsubsection{Response Functions and Variance Decomposition}

To analyze and compute $w_{c,k}$, we introduce two response functions associated with the base diffusion.
The first tracks the part of an input perturbation that remains after abandonment feedback, while the second records the expected cumulative feedback generated from a given initial workload.

Define $q_k(x)\triangleq kx^{k-1}$ for $x\ge0$.
For $x\ge0$, let $\mathbb P_x$ denote the law of the reflected diffusion \eqref{eq:base_reflected_SDE} started from $Z^{c,k}(0)=x$, and let $T_0\triangleq\inf\{t\ge0:Z^{c,k}(t)=0\}$.
For $t\ge0$ and $x\ge0$, define
\begin{equation}\label{eq:psi_def}
    \psi_{c,k}(t,x)
    \triangleq
    \E_x\left[
        \exp\left\{-\int_0^{t\wedge T_0}q_k(Z^{c,k}(s))ds\right\}
    \right].
\end{equation}
The exponential factor discounts a perturbation while the diffusion remains above the reflecting boundary.
When the workload is high, $q_k(Z^{c,k}(s))$ is large and the perturbation is damped more rapidly.

Define also
\begin{equation}\label{eq:varphi_def}
    \varphi_{c,k}(t,x)
    \triangleq
    \E_x\left[\int_0^t(Z^{c,k}(s))^kds\right].
\end{equation}
By \eqref{eq:base_Y_def}, $\E_x[Y^{c,k}(t)]=-\varphi_{c,k}(t,x)$.
Thus $\varphi_{c,k}$ records the expected cumulative abandonment feedback generated from the initial state.

The next proposition provides the backward characterization used both in the analysis and in the numerical computation.
Its derivative identity is the link between the two response functions.

\begin{proposition}[Backward characterization of $\psi_{c,k}$ and $\varphi_{c,k}$]\label{prop:psi_h_pde}
The function $\psi_{c,k}$ defined in \eqref{eq:psi_def} is bounded and continuous, is uniquely characterized among bounded continuous functions by the corresponding stopped Feynman--Kac representation, and satisfies
\begin{equation}\label{eq:psi_pde}
    \begin{cases}
        \partial_t \psi(t,x) = \partial_{xx}\psi(t,x) + (c-x^k)\partial_x\psi(t,x) - q_k(x)\psi(t,x), & t>0,\ x>0,\\
        \psi(0,x)=1, & x\ge0,\\
        \psi(t,0)=1, & t\ge0.
    \end{cases}
\end{equation}
The function $\varphi_{c,k}$ defined in \eqref{eq:varphi_def} is continuous, is uniquely characterized in the stated polynomial-growth class by the corresponding stopped Feynman--Kac representation, and satisfies
\begin{equation}\label{eq:varphi_pde}
    \begin{cases}
        \partial_t \varphi(t,x) = \partial_{xx}\varphi(t,x) + (c-x^k)\partial_x\varphi(t,x) + x^k, & t>0,\ x>0,\\
        \varphi(0,x)=0, & x\ge0,\\
        \partial_x\varphi(t,0)=0, & t\ge0,
    \end{cases}
\end{equation}
that satisfies $\sup_{t\in[0,T]}\sup_{x\ge0}|\varphi(t,x)|/(1+x^k)<\infty$ for every $T<\infty$.
For every $t>0$, the map $x\mapsto\varphi_{c,k}(t,x)$ is continuously differentiable on $[0,\infty)$, and
\begin{equation}\label{eq:varphi_x_psi}
    \partial_x\varphi_{c,k}(t,x)=1-\psi_{c,k}(t,x),
    \qquad x\ge0.
\end{equation}
\end{proposition}

Equation~\eqref{eq:varphi_x_psi} says that the marginal increase in expected cumulative abandonment feedback is $1-\psi_{c,k}$.
Hence $\psi_{c,k}$ is the residual fraction of a unit initial perturbation after that feedback is accounted for.
The localized It\^o argument in Lemma~\ref{lm:w_variance_rep} then separates the effective-input variance into future-noise and stationary-initial-state contributions.

\begin{lemma}[Variance representation]\label{lm:w_variance_rep}
For each $t>0$,
\begin{equation}\label{eq:w_variance_rep}
    w_{c,k}(t)
    = \frac{1}{t}\int_0^t \E_{\pi_{c,k}}\left[\psi_{c,k}(u,Z)^2\right]du
    + \frac{1}{2t}\Var_{\pi_{c,k}}\left(\varphi_{c,k}(t,Z)\right),
    \qquad Z\sim\pi_{c,k}.
\end{equation}
\end{lemma}

\subsubsection{Qualitative Properties of the Variance-Reduction Function}

The variance representation and the sensitivity identity yield the properties needed in the refined RQ analysis.

\begin{proposition}\label{prop:w}
The variance-reduction function has the following properties.
\begin{enumerate}
\item For every $c\in\R$, integer $k\ge1$, and $t>0$, one has $0<w_{c,k}(t)\le1$, and $\lim_{t\downarrow0}w_{c,k}(t)=1=w_{c,k}(0)$.
\item For each fixed $c$ and $k$, the map $t\mapsto w_{c,k}(t)$ is strictly decreasing, and
\begin{equation}\label{eq:w_infty_formula}
    w_{c,k}(\infty)
    \triangleq \lim_{t\to\infty}w_{c,k}(t)
    = \pi_{c,k}(0)^2\int_0^\infty
    \frac{\overline\Pi_{c,k}(x)^2}{\pi_{c,k}(x)}dx
    >0,
    \qquad
    \overline\Pi_{c,k}(x)\triangleq\int_x^\infty\pi_{c,k}(y)dy.
\end{equation}
\item For each fixed $k$, the mapping $(c,t)\mapsto w_{c,k}(t)$ is jointly continuous on $\R\times[0,\infty)$, and hence uniformly continuous on compact subsets.
\item The long-time limit satisfies
\[
    \lim_{c\to-\infty}w_{c,k}(\infty)=1,
    \qquad
    \lim_{c\to\infty}w_{c,k}(\infty)=0.
\]
\end{enumerate}
\end{proposition}

The short-time limit says that abandonment has no first-order variance effect over an infinitesimal horizon.
The strict decrease in $t$ says that the feedback accumulates over longer horizons.
Joint continuity permits the finite-system index $\tilde c_\alpha$ to approach its heavy-traffic limit uniformly on compact scaled-time intervals.
The load endpoints describe the long-horizon variance reduction in extreme underload and overload.

\subsubsection{Numerical Evaluation of the Variance-Reduction Function}\label{sec:compute_w}

By Lemma~\ref{lm:var_expression}, evaluating $v(t;\Xi)$ reduces to computing the base function $w_{c,k}$.
Proposition~\ref{prop:psi_h_pde} and Lemma~\ref{lm:w_variance_rep} reduce this computation to two one-dimensional parabolic PDEs followed by stationary quadrature.
The function $w_{c,k}$ does not generally have a closed form because it depends on the transient law of the stationary reflected diffusion \eqref{eq:base_reflected_SDE}.

\begin{remark}[Numerical evaluation of $w_{c,k}$]\label{rmk:est_w}
For each pair $(c,k)$, we solve \eqref{eq:psi_pde} and \eqref{eq:varphi_pde} on a truncated interval $[0,x_{\max}]$.
The truncation point is chosen adaptively using the stationary density \eqref{eq:pi_ck}.
Specifically, $x_{\max}$ is increased until the unnormalized stationary density at $x_{\max}$ is at most $10^{-14}$ of its value at the mode.
The experiments use a uniform spatial grid with $2000$ subintervals.

The diffusion term is discretized by second-order finite differences.
The drift term $c-x^k$ is discretized by a centered difference when the local mesh satisfies $|c-x^k|\Delta x\le2$ and by generator-form upwinding otherwise.
At $x=0$, we impose $\psi(t,0)=1$ for \eqref{eq:psi_pde} and the reflecting boundary condition $\partial_x \varphi(t,0)=0$ for \eqref{eq:varphi_pde}.
At $x=x_{\max}$, we impose a zero-gradient boundary condition after placing $x_{\max}$ in the far stationary tail.

The PDEs are advanced in time using a Crank--Nicolson scheme.
The output time grid is logarithmic, with default range $10^{-4}\le t\le10^8$ and $1200$ positive grid points, together with the value $w_{c,k}(0)=1$.
The internal time step is adaptive and satisfies $\Delta t \le 0.03\max\{t,10^{-4}\}$.
At each output time, the expectations and variances with respect to $\pi_{c,k}$ in \eqref{eq:w_variance_rep} are evaluated by the trapezoidal rule on the spatial grid.
The time integral of $\E_{\pi_{c,k}}[\psi_{c,k}(u,Z)^2]$ is also accumulated by the trapezoidal rule.
When both the Brownian-noise term and the initial-state term vary by less than $10^{-9}$ over $25$ consecutive time steps, the remaining tail is extrapolated using the limiting values.

The table for $w_{c,k}$ is computed offline.
In the reported computations, we tabulate $c\in[-20,20]$ with spacing $0.1$ for each relevant value of $k$.
During the RQ fixed-point computation, $w_{c,k}(t)$ is evaluated from this table.
For positive $t$, we use shape-preserving cubic interpolation in $\log t$ along each fixed-$c$ row and linear interpolation in $c$.
For $t=0$, we set $w_{c,k}(0)=1$.
For $c$ outside the tabulated range, we apply the implementation's exponential tail convention, which extrapolates toward $1$ on the underloaded side and toward $0$ on the overloaded side.
This convention is used only outside the reported $c$-grid, while Proposition~\ref{prop:w} supplies the corresponding long-horizon endpoint values.
Figure~\ref{fig:wck} shows the numerically computed $w_{c,k}(t)$ for $k=1,2,3$ and selected values of $c$.
\end{remark}

\begin{figure}[htbp]
    \centering
    \includegraphics[width=\textwidth]{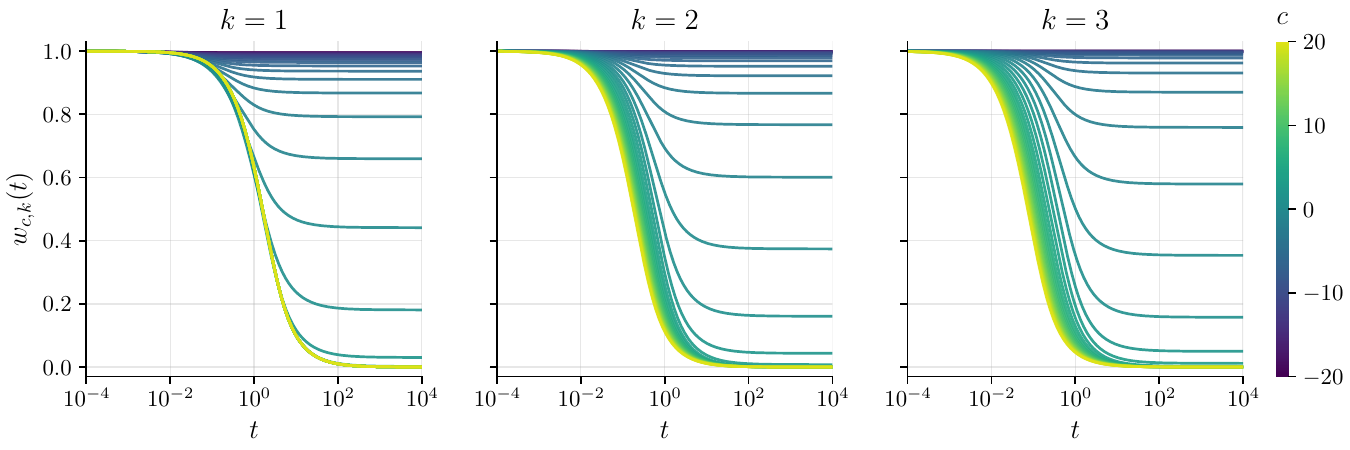}
    \caption{The numerically computed $w_{c,k}(t)$ for $k=1,2,3$ and selected values of $c$.}\label{fig:wck}
\end{figure}

Combining Lemma~\ref{lm:var_expression} with the precomputed table for $w_{c,k}$ gives an efficient way to evaluate the heavy-traffic variance function for any model primitives.
For a parameter tuple $\Xi$, we compute $\tilde c$ and $\tau$ from Lemma~\ref{lm:var_expression}, evaluate $w_{\tilde c,k}(\tau t)$ by interpolation, and then set
\[
    v(t;\Xi) = \frac{c_x^2}{\mu}t w_{\tilde c,k}(\tau t).
\]
Thus the online cost of the refined RQ approximation is only a table lookup and interpolation.

\begin{remark}[Why not use full hazard-rate scaling]\label{rem:why_not_full_hazard}
The full hazard-rate scaling of \citet{reed2008approximating} was developed to address a limitation of the conventional Ward--Glynn heavy-traffic regime: when the patience density at the origin vanishes, the usual ROU approximation no longer captures the abandonment effect at diffusion scale.
Under hazard-rate scaling, the limiting drift depends on the cumulative hazard $H(x)=\int_0^x h(u)du$, so the full patience-time distribution enters the limiting diffusion.
Our construction is motivated by \cite{reed2008approximating}, but retains only a two-parameter summary of it.

Adopting the full hazard-rate scaling in the RQ variance calibration would replace the polynomial feedback term $x^k$ in \eqref{eq:base_reflected_SDE} by a general cumulative hazard $H(x)$.
The resulting variance-reduction function would then be indexed by the entire function $H$ rather than by the two parameters $(c,k)$.
Consequently, both the variance table and the calibration of $b$ would have to be recomputed for each patience-time distribution.
Moreover, the scaling representation in Lemma~\ref{lm:var_expression} relies on the homogeneity of the polynomial drift $x^k$, so the same rescaling and table-lookup implementation is no longer available for a general cumulative hazard.

The present approximation is therefore a compromise between fidelity and implementability.
It retains the hazard-rate insight that the conventional regime must be modified when $f(0)=0$, but keeps only the local order $k$ and coefficient $\beta=F^{(k)}(0)/k!$, so that all model primitives enter the variance surrogate through the rescaling in Lemma~\ref{lm:var_expression} and the precomputed table for $w_{c,k}$.
\end{remark}

\subsection{Two Variance Limits and Their Interpolation}\label{sec:HT_var}

The finite-system variance surrogate combines two asymptotic statements that describe different time scales.
The first is a stationary heavy-traffic limit on horizons of order $\alpha^{-2h}$.
The second is a fixed-horizon long-patience limit away from exact critical loading.
We state the two endpoints first and introduce their interpolation only afterward.

\paragraph{Stationary heavy-traffic endpoint.}
Assume the $\alpha$th queue is in equilibrium at time $0$, and let $\E_e$ and $\Var_e$ denote expectation and variance under that stationary law.
Subtracting the deterministic centering term does not affect variance.
Recall from \eqref{eq:scaled_Y} that
\[
    \tilde Y^\alpha(t)
    =\alpha^h\left[Y^\alpha(\alpha^{-2h}t)-\alpha^{-2h}\rho_\alpha t\right].
\]
Hence
\[
    \Var_e\bigl(\tilde Y^\alpha(t)\bigr)
    =\alpha^{2h}\Var_e\bigl(Y^\alpha(\alpha^{-2h}t)\bigr).
\]
Weak convergence from Theorem~\ref{thm:HT_limit}, together with the explicit uniform-integrability condition below, yields convergence of these second moments.

\begin{corollary}[Heavy-traffic limit of the variance function]\label{thm:HT_var}
Assume the conditions of Theorem~\ref{thm:HT_limit}.
Assume further that the $\alpha$th system is in equilibrium at time $0$ and that for each fixed $t\ge 0$ the family $\{|\tilde Y^\alpha(t)|^2:\alpha>0\}$ is uniformly integrable (e.g., $\sup_{\alpha}\E_e[|\tilde Y^\alpha(t)|^{2+\delta}]<\infty$ for some $\delta>0$).
Then for each fixed $t\ge 0$,
\[
    \Var_e(\tilde Y^\alpha(t)) \longrightarrow v(t;\Xi) = \Var\bigl(Y^*(t)\bigr),\qquad \alpha\downarrow 0,
\]
where $Y^*$ is the limiting process in Theorem~\ref{thm:HT_limit} and $\Xi=(c,k,\mu,c_a^2,c_s^2,\beta)$ is the parameter tuple.
Moreover, with $\tilde c$ and $\tau$ defined in Lemma~\ref{lm:var_expression},
\[
    v(t;\Xi) = \frac{c_x^2}{\mu} t w_{\tilde c,k}(\tau t), \qquad t\ge 0.
\]
\end{corollary}

\paragraph{Finite-system form suggested by the heavy-traffic endpoint.}
For each finite $\alpha$, write $\mu$ and $V_1$ for $\mu_\alpha$ and $V_1^\alpha$, respectively, and define the normalized load index
\begin{equation}\label{eq:ctilde_alpha}
    \tilde c_\alpha
    \triangleq
    \alpha^{-h}(\rho_\alpha-1)
    \left(\frac{c_x^2}{2\mu}\right)^{-\frac{k}{k+1}}
    \beta^{-\frac1{k+1}}.
\end{equation}
Under the critical scaling in Corollary~\ref{thm:HT_var},
$\tilde c_\alpha\to\tilde c$.
Joint continuity of $w$ in Proposition~\ref{prop:w} therefore permits the finite-system index $\tilde c_\alpha$ to replace $\tilde c$, uniformly over compact scaled-time intervals.
For a single finite system, we henceforth suppress the subscript on $\rho_\alpha$.

Writing $s$ for physical time, Corollary~\ref{thm:HT_var} suggests
\begin{equation}\label{eq:V_approx_large_t}
    \Var_e\bigl(Y^\alpha(s)\bigr)
    \approx
    \frac{c_x^2}{\mu}s\,
    w_{\tilde c_\alpha,k}\bigl(\alpha^{2h}\tau s\bigr),
    \qquad s=O(\alpha^{-2h}).
\end{equation}
On this scale, the stationary effective-work rate converges to one, so
$\E_e[Y^\alpha(s)]\approx s$.
It is convenient to express \eqref{eq:V_approx_large_t} through the dimensionless effective IDW
\begin{equation}\label{eq:effective_IDW}
    I_w^{\mathrm{ab}}(s)
    \triangleq
    \frac{\Var_e(Y^\alpha(s))}
         {\E[V_1]\E_e[Y^\alpha(s)]}
    =\mu\frac{\Var_e(Y^\alpha(s))}{\E_e[Y^\alpha(s)]}.
\end{equation}
The heavy-traffic endpoint then becomes
\[
    I_w^{\mathrm{ab}}(s)
    \approx
    c_x^2 w_{\tilde c_\alpha,k}\bigl(\alpha^{2h}\tau s\bigr),
    \qquad s=O(\alpha^{-2h}).
\]
The factor $c_x^2$ is the limiting variability parameter; $w$ is the additional reduction generated by abandonment feedback.

\begin{remark}[Comparison with a queue without abandonment]
For a stationary $GI/GI/1$ marked-renewal input
$\widetilde Y(s)=\sum_{i=1}^{A(s)}V_i$, the work IDW converges to $c_x^2$ as $s\to\infty$.
For the abandonment model, Corollary~\ref{thm:HT_var} gives instead
\[
    \lim_{t\to\infty}\lim_{\alpha\downarrow0}
    \frac{\Var_e(Y^\alpha(\alpha^{-2h}t))}
         {\E[V_1]\E_e[Y^\alpha(\alpha^{-2h}t)]}
    =c_x^2w_{\tilde c,k}(\infty)
    \le c_x^2.
\]
Thus $w_{\tilde c,k}(\infty)$ is the limiting long-horizon attenuation factor.
\end{remark}

\paragraph{Fixed-horizon long-patience endpoint.}
The heavy-traffic result describes horizons that grow like $\alpha^{-2h}$.
For a fixed physical horizon, long patience produces a different limit.
When $\rho<1$, the effective input approaches the original marked renewal input because abandonment vanishes.
When $\rho>1$, stationary flow balance requires the limiting fraction of arrivals that are eventually served to be $1/\rho$.
The next lemma makes this statement precise under a stationary concentration assumption and shows that the first two effective-input moments converge to those of Bernoulli thinning with retention probability $1/(\rho\vee1)$.
Recall the stationary arrival IDC
\[
    I_a(t)\triangleq\frac{\Var(A(t))}{\lambda t}.
\]
\begin{lemma}\label{lm:var_lim}
Fix $t>0$ and consider a stationary $GI/GI/1{+}GI$ queue with arrival rate $\lambda$, service rate $\mu$, and patience distribution $F_\alpha(x)$.
Let $\rho=\lambda/\mu\ne1$.
If $\rho>1$, let $\xi_\rho>0$ be the unique solution of $\bar F(\xi_\rho)=1 / \rho$, and assume that $\alpha Z^\alpha(0)\Rightarrow\xi_\rho$ under the stationary law.
Then, as $\alpha\downarrow0$,
\[
    \E_e[Y^\alpha(t)] \longrightarrow (\rho\wedge 1)t,
    \qquad
    \frac{\Var_e(Y^\alpha(t))}{\E_e[Y^\alpha(t)]}
    \longrightarrow \frac{1}{\mu}\left( \frac{I_a(t)}{\rho\vee 1} + 1-\frac{1}{\rho\vee 1} + c_s^2 \right).
\]
\end{lemma}

\paragraph{Heuristic interpolation used by the refined RQ algorithm.}
Corollary~\ref{thm:HT_var} and Lemma~\ref{lm:var_lim} are proved endpoints, but neither supplies a uniform approximation over all horizons.
We connect them with the following explicit heuristic.
Define
\begin{equation}\label{eq:IDW_refined_RQ}
    \hat I_w(s)
    \triangleq
    \frac{I_a(s)}{\rho\vee1}
    +\left(1-\frac{1}{\rho\vee1}\right)
    +c_s^2.
\end{equation}
For fixed $s$ and $\rho\ne1$, Lemma~\ref{lm:var_lim} gives
$I_w^{\mathrm{ab}}(s)\to\hat I_w(s)$ as $\alpha\downarrow0$.
Near critical loading on the growing heavy-traffic horizon, the first factor approaches $c_x^2$, and Corollary~\ref{thm:HT_var} supplies the attenuation factor $w$.
We therefore use
\begin{equation}\label{eq:IDW_effective_approx}
    I_w^{\mathrm{ab}}(s)
    \approx
    \hat I_w(s)
    w_{\tilde c_\alpha,k}\bigl(\alpha^{2h}\tau s\bigr).
\end{equation}
Equation~\eqref{eq:IDW_effective_approx}, not the two endpoint limits themselves, is the finite-system modeling approximation used below.
It is consistent with the heavy-traffic endpoint because
$\hat I_w(\infty)\to I_a(\infty)+c_s^2=c_x^2$ as $\rho\to1$.

The factorization has a useful interpretation.
The term $\hat I_w(s)$ describes variability from finite-horizon arrival counts, Bernoulli-type effective-arrival thinning, and service times.
Within it,
\[
    \frac{I_a(s)}{\rho\vee1}
    +\left(1-\frac{1}{\rho\vee1}\right)
\]
is the fixed-horizon limiting IDC of effective arrivals: a load-dependent convex combination of the original IDC and the Poisson value $1$.
The factor $w_{\tilde c_\alpha,k}$ then applies the additional dynamic reduction caused by state-dependent abandonment feedback.

\begin{example}
To illustrate the interpolation, we consider the $H_2(4)/M/1{+}M$ and $H_2(4)/M/1{+}E_2$ models.
Here $H_2(4)$ denotes a balanced-means hyperexponential distribution with SCV $4$, and $E_2$ denotes an Erlang distribution with shape parameter $2$.
For $E_2$, the patience-time CDF satisfies Assumption~\ref{assumption:F} with $k=2$, and hence $h=2/3$.
We set $\mu=1$ and $\lambda_{\alpha}=\rho_{\alpha}=1+c\alpha^{2/3}$ with $c=2$ and $\alpha=2^{-j}$ for $j\in\{0,3,6,9,12\}$.
Figure~\ref{fig:Var_approx} compares simulation estimates of the effective IDW defined in \eqref{eq:effective_IDW} (solid curves) with the approximation in \eqref{eq:IDW_effective_approx} (dashed curves).
These estimates obtained from a single simulation run with a warm-up period of $10^7$ time units and a data-collection period of $3\times 10^8$ time units, using the IDW estimation procedure described in \citep[Section~2.1.2]{whitt2022robust}.
The variance-reduction function $w_{\tilde c_\alpha,2}(\cdot)$ is evaluated using the procedure described in Remark~\ref{rmk:est_w}.
The comparison indicates how accurately the interpolation joins the fixed-horizon and growing-horizon regimes over the displayed patience levels.
\begin{figure}[htbp]
    \centering
    \includegraphics[width=0.495\textwidth]{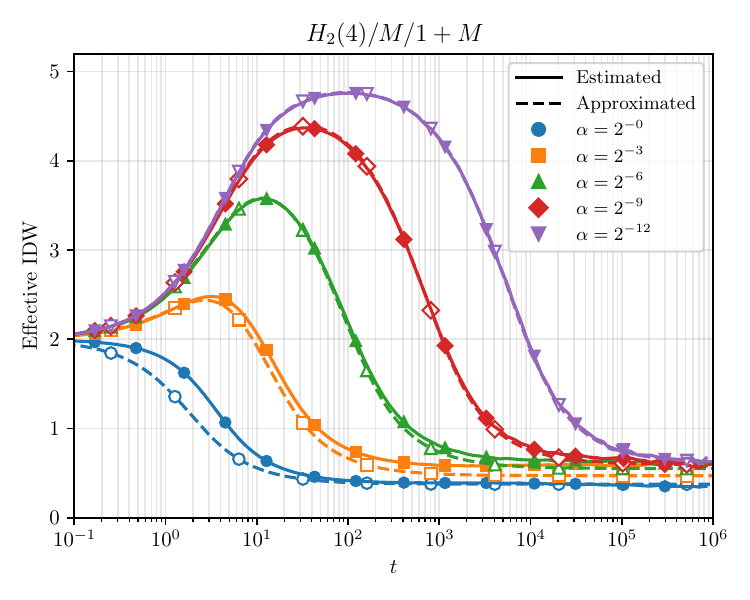}
    \includegraphics[width=0.495\textwidth]{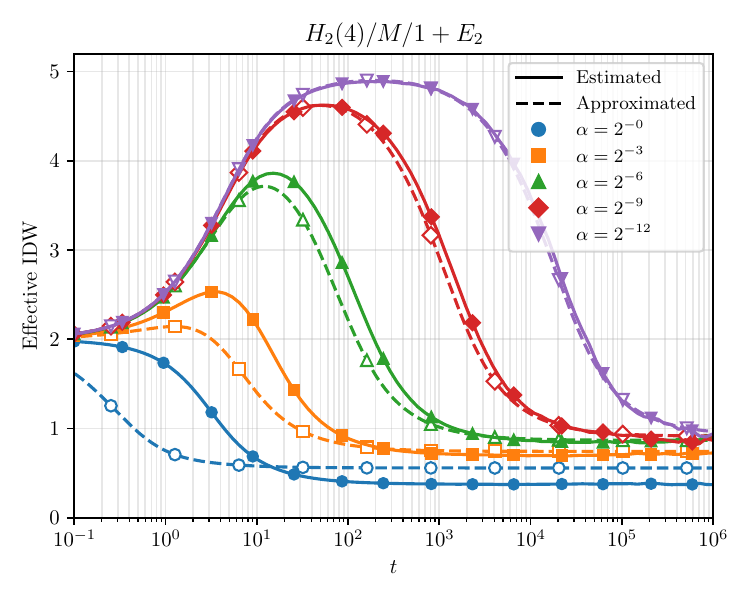}
    \caption{Simulation estimates (solid) and approximations \eqref{eq:IDW_effective_approx} (dashed) of the effective IDW in the $H_2(4)/M/1{+}M$ model (left), and the $H_2(4)/M/1{+}E_2$ model (right) with $\mu=1$, $c=2$, and $\alpha=2^{-j}$ for $j\in\{0,3,6,9,12\}$.}
    \label{fig:Var_approx}
\end{figure}
\end{example}

\subsection{Robust Queueing Algorithm}\label{sec:RQ_ab2_algo}

With the drift approximation in \eqref{eq:mean_drift} and the heuristic effective-IDW surrogate \eqref{eq:IDW_effective_approx}, we now propose a refined robust queueing formulation for the virtual waiting time.

The definition of $\Lambda_t(s)$ in \eqref{eq:Lambda} depends on the state process $Z(\cdot)$, so we apply the same stationary self-consistency step as in Section~\ref{sec:RQ_ab1_algo}: with $z$ a deterministic trial approximation to the mean stationary virtual waiting time, we replace $\bar F_\alpha(Z(u-))$ by $\bar F_\alpha(z)$ over the reverse-time interval, giving the drift surrogate $\lambda\bar F_\alpha(z)s$ and the mean approximation $\E_e[Y^\alpha(s)]\approx\lambda\bar F_\alpha(z)s/\mu$.
Combining this with the effective-IDW approximation \eqref{eq:IDW_effective_approx} gives the stationary drift and variance surrogates
\begin{equation}\label{eq:m_v2_def}
    m_z(s) = \bigl(\rho\bar F_\alpha(z)-1\bigr)s,
    \qquad v^{(2)}_z(s)
    \triangleq \hat I_w(s) w_{\tilde c_\alpha,k}\bigl(\alpha^{2h}\tau s\bigr) \frac{\lambda \bar F_{\alpha}(z)s}{\mu^2}.
\end{equation}

The resulting RQ surrogate for the effective net-input increment $N(t)-N(t-s)$, evaluated at the trial value $z$, is
\begin{equation}\label{eq:rq_ab2}
    m_z(s)+b\sqrt{v^{(2)}_z(s)}.
\end{equation}
Substituting \eqref{eq:rq_ab2} into the stationary RQ formulation \eqref{eq:RQ_raw_formulation} and imposing self-consistency $z=Z_{\mathrm{RQ}_2}$ gives the refined stationary RQ fixed-point equation
\begin{align}
    Z_{\mathrm{RQ}_2}
    &= \sup_{s\ge 0} \left\{ m_{Z_{\mathrm{RQ}_2}}(s)+b\sqrt{v^{(2)}_{Z_{\mathrm{RQ}_2}}(s)} \right\} \notag\\
    &= \sup_{s\ge 0} \left\{ \bigl(\rho\bar F_{\alpha}(Z_{\mathrm{RQ}_2})-1\bigr)s +b \sqrt{ \hat I_w(s) w_{\tilde c_\alpha,k}\bigl(\alpha^{2h}\tau s\bigr) \frac{\lambda \bar F_{\alpha}(Z_{\mathrm{RQ}_2})s}{\mu^2} } \right\}. \label{eq:RQ_ab_2}
\end{align}
Equation \eqref{eq:RQ_ab_2} defines the refined steady-state RQ approximation directly as a scalar fixed point.

By arguments analogous to those in Remark~\ref{rem:RQ1_existence}, \eqref{eq:RQ_ab_2} admits a unique finite solution under Assumption~\ref{assumption}.
Indeed, $z\mapsto\bar F_\alpha(z)$ is continuous, strictly decreasing, and vanishes at infinity, while $s\mapsto\hat I_w(s)\,w_{\tilde c_\alpha,k}(\alpha^{2h}\tau s)$ is continuous and bounded, and is bounded away from zero for all sufficiently large $s$ because $\hat I_w(s)>0$ and $w_{\tilde c_\alpha,k}(\alpha^{2h}\tau s)\to w_{\tilde c_\alpha,k}(\infty)>0$ by Proposition~\ref{prop:w}.

We denote the unique solution by $Z_{\mathrm{RQ}_2,b}$.
In practice, $Z_{\mathrm{RQ}_2,b}$ can be computed efficiently by bisection.
We discuss calibration of the robustness parameter $b$ in Section~\ref{sec:calibrate_refined_RQ}.

\subsection{Heavy-Traffic Limit for Robust Queueing}

Recall the threshold $h$ defined in \eqref{eq:threshold}.
We study the scaling of the refined RQ fixed point in the long-patience heavy-traffic regime $\alpha\downarrow 0$ with $\rho_{\alpha}\to 1$.

\begin{theorem}[Heavy-traffic limit for refined RQ]\label{Thm:RQ_refined_HT}
Consider the heavy-traffic family described above under Assumptions~\ref{assumption} and~\ref{assumption:F}.
Fix $b>0$ and $\mu>0$, and let $\lambda_{\alpha}=\rho_{\alpha}\mu$, with $\alpha^{-\gamma}(\rho_{\alpha}-1)\to c$ for some $\gamma>0$ and $c\in \R$.
Recall $c_x^2$ from Assumption~\ref{assumption} (and $\beta$ from Assumption~\ref{assumption:F}).
For each $\alpha>0$, let $Z_{\mathrm{RQ}_2,b}^\alpha$ denote the unique finite solution to the refined steady-state RQ equation \eqref{eq:RQ_ab_2}.
\begin{enumerate}
\item \textbf{Underloaded.}
If $0<\gamma<h$ and $c<0$, then
\[
    \lim_{\alpha\downarrow 0}(-c)\mu\alpha^\gamma Z_{\mathrm{RQ}_2,b}^\alpha
    = \lim_{\alpha\downarrow 0}\frac{Z_{\mathrm{RQ}_2,b}^\alpha}{\E[Z_{M/M/1}]} = \frac{b^2}{2}\frac{c_x^2}{2},
\]
where $\E[Z_{M/M/1}] = \rho_{\alpha}\bigl(\mu(1-\rho_{\alpha})\bigr)^{-1}$.

\item \textbf{Critically loaded.}
If $\gamma\ge h$, then there exists a finite constant $\hat Z_{\mathrm{RQ}_2,b}>0$ such that
\[
    \lim_{\alpha\downarrow 0}\alpha^h Z_{\mathrm{RQ}_2,b}^\alpha=\hat Z_{\mathrm{RQ}_2,b}.
\]
Moreover, $\hat Z_{\mathrm{RQ}_2,b}$ is the unique positive solution to
\begin{equation}\label{eq:RQ_HT_equation_general}
    \hat Z_{\mathrm{RQ}_2,b}
    = \sup_{u\ge 0} \left\{ \left(\mathds{1}\{\gamma=h\}c-\beta\hat Z_{\mathrm{RQ}_2,b}^k\right)u + b\sqrt{\frac{c_x^2}{\mu}w_{\tilde c_\gamma,k}(\tau u)u} \right\},
\end{equation}
where
\[
    \tilde c_\gamma \triangleq \mathds{1}\{\gamma=h\}c \left(\frac{c_x^2}{2\mu}\right)^{-k/(k+1)} \beta^{-1/(k+1)}
    = \lim_{\alpha\downarrow0}\tilde c_\alpha .
\]

\item \textbf{Overloaded.}
If $0<\gamma<h$ and $c>0$, then
\[
    \lim_{\alpha\downarrow 0}\alpha^{1-\gamma/k} Z_{\mathrm{RQ}_2,b}^\alpha = \left(\frac{c}{\beta}\right)^{1/k}.
\]
In particular, the leading-order limit is independent of $b$ and of the arrival and service variability parameters, and it depends on the patience-time distribution only through $F^{(k)}(0)$.
Equivalently,
\[
    \lim_{\alpha\downarrow 0}\frac{\rho_{\alpha}}{\rho_{\alpha}-1}F\bigl(\alpha Z_{\mathrm{RQ}_2,b}^\alpha\bigr)=1.
\]
\end{enumerate}
\end{theorem}

Theorem~\ref{Thm:RQ_refined_HT} is proved together with Theorem~\ref{Thm:RQ_HT} in Section~\ref{sec:proof_RQ_HT}.
The underloaded and overloaded limits in Theorem~\ref{Thm:RQ_refined_HT} coincide with those of the first RQ algorithm in Theorem~\ref{Thm:RQ_HT}.
The key difference lies in the critically-loaded regime.
Both algorithms yield the same scaling $Z_{\mathrm{RQ}}^\alpha=O(\alpha^{-h})$, but the refined limit depends nontrivially on the variance-reduction function $w_{\tilde c_\gamma,k}$.

\subsubsection{Control of the Mean Palm Correction}\label{sec:control_palm}

As noted in Remark~\ref{rem:control_palm}, the linear bound \eqref{eq:palm_linear_bound} must be complemented by control on the horizons selected by the RQ suprema.
The following proposition provides this control on the optimizer scales of Theorems~\ref{Thm:RQ_HT} and~\ref{Thm:RQ_refined_HT}.

\begin{proposition}[Mean Palm correction on RQ optimizer scales]\label{prop:palm_optimizer_scales}
For every fixed $M<\infty$, the mean correction is negligible at the workload and horizon scales used in Theorems~\ref{Thm:RQ_HT} and~\ref{Thm:RQ_refined_HT}:
\begin{align*}
    (1-\rho_\alpha)\sup_{0\le s\le M(1-\rho_\alpha)^{-2}}|\E[\delta_\alpha(s)]|&\longrightarrow0, && c<0,\ \gamma<h,\\
    \alpha^h\sup_{0\le s\le M\alpha^{-2h}}|\E[\delta_\alpha(s)]|&\longrightarrow0, && \gamma\ge h,\\
    \alpha^{1-\gamma/k}\sup_{0\le s\le M\alpha^{-1-(k-1)\gamma/k}}|\E[\delta_\alpha(s)]|&\longrightarrow0, && c>0,\ 0<\gamma<h.
\end{align*}
\end{proposition}

Proposition~\ref{prop:palm_optimizer_scales} is not used to prove the two RQ limits.
It instead verifies, on the localized horizon ranges in their proofs, that the omitted net-input term $\mu^{-1}\E[\delta_\alpha(s)]$ is of smaller order than the scaled workload.

\subsubsection{Calibration of the Parameter \texorpdfstring{$b$}{b}} \label{sec:calibrate_refined_RQ}

In the underloaded regime of Theorem~\ref{Thm:RQ_refined_HT}, setting $b=\sqrt{2}$ recovers the original RQ algorithm in \cite{whitt2018using}.
In the overloaded regime of Theorem~\ref{Thm:RQ_refined_HT}, the value of $b$ is immaterial.

The remaining challenge is the critically loaded case.
In the critically loaded case $\gamma=h$, the limit index is $\tilde c_\gamma=c(c_x^2/(2\mu))^{-k/(k+1)}\beta^{-1/(k+1)}$.
Because of the scale-dependent factor $w_{\tilde c_\gamma,k}(\tau u)$, the fixed-point equation in \eqref{eq:RQ_HT_equation_general} does not yield a closed-form calibration of $b$.
We therefore calibrate $b$ numerically by matching the exact stochastic heavy-traffic limit in \eqref{eq:HT_exact} with the heavy-traffic limit of the refined RQ approximation in \eqref{eq:RQ_HT_equation_general}.
Specifically, we use the $M/M/1{+}M$ model for $k=1$ and the $M/M/1{+}E_k$ model for $k>1$.
This gives a calibrated value of $b$ for each pair $(k,\tilde c_\gamma)$.

In the numerical experiments, we precompute a table of calibrated $b$ values over the required values of $k$ and over a discrete calibration grid for the normalized load index.
For a given stochastic model, we determine $k$ from the patience distribution as the order of the first nonzero derivative at the origin and compute the finite-instance index $\tilde c_\alpha$ from \eqref{eq:ctilde_alpha}, using $\rho=\lambda/\mu$.
We then interpolate the precomputed table in $\tilde c_\alpha$ for each $k$ to obtain the calibrated value of $b$ used in \eqref{eq:RQ_ab_2}.

\section{Numerical Experiments}\label{sec:numerical}

This section evaluates the refined fixed point \eqref{eq:RQ_ab_2} across loading, patience, and variability regimes.
All reported RQ values use the instance-dependent calibration from Section~\ref{sec:calibrate_refined_RQ} and the variance surrogate $v_z^{(2)}$ in \eqref{eq:m_v2_def}.
The first RQ algorithm is retained for exposition and asymptotic comparison; the numerical study focuses on the refined method.
Code and figure-reproduction scripts are available at
\url{https://github.com/cnyouwei/RQ_ab_toolkit}.

\paragraph{Inputs and numerical solution.}
The arrival input is the IDC
$I_a(t)=\Var(A(t))/\E[A(t)]$.
For renewal input, it can be computed by numerical inversion of renewal-transform formulas; for observed non-renewal input, it can be estimated over windows of length $t$ by the sample variance-to-mean ratio of arrival counts \citep[Section~2.1.2]{whitt2022robust}.
For each trial value of $z$, the supremum in \eqref{eq:RQ_ab_2} is evaluated on $800$ logarithmically spaced positive horizons over $[10^{-4},10^8]$, with $s=0$ added.
The outer fixed point is solved by bisection.
The initial bracket is $[0,1]$; the upper endpoint may be doubled at most $80$ times, and the solve uses at most $200$ bisection iterations.

\paragraph{Benchmark approximations.}
We compare refined RQ with three diffusion-based methods.
The Ward--Glynn approximation (WG) \citep{ward2003diffusion,ward2005diffusion} is
$\alpha^{-1/2}\E[\tilde Z(\infty)]$, with the expectation in \eqref{eq:ROU_expectation}, when $f(0)>0$.
When $f(0)=0$, we instead report the hazard-rate scaling approximation of \citet{reed2008approximating}.
We also report the Huang--Gurvich approximation (HG) \citep{huang2018beyond}.
Appendix~\ref{sec:other_methods} summarizes these methods and the minor extension used to apply HG to general arrival processes; that extension does not inherit the guarantee proved for the original $M/GI/1{+}GI$ setting.
The WG and hazard-rate methods are designed primarily for near-critical, long-patience systems, whereas HG directly covers a wider loading range under Poisson arrivals.

\paragraph{Test grid and error metric.}
We normalize the mean service time to one.
The arrival-rate grid is
\[
    \lambda\in
    \{1-2^{-j}:j=1,\ldots,10\}
    \cup
    \{1+2^{-j}:j=-2,-1,\ldots,10\},
\]
and the patience-scale grid is
$\alpha\in\{2^{-j}:j=0,1,\ldots,13\}$.
Because the base patience distribution has mean one, the corresponding mean patience times are
$1/\alpha\in\{1,2,4,\ldots,8192\}$.
The $23\times14=322$ combinations cover underload, near-critical loading, and overload.
Each heat map reports
\[
    \frac{\text{approximation}-\text{reference}}{\text{reference}}.
\]
Blue denotes overestimation and red denotes underestimation; darker shading indicates larger magnitude, and the display is clipped at $\pm30\%$.
For the tractable $M/M/1{+}GI$ models, the reference is the exact stationary mean described below.
For all other single-queue and tandem models, the reference is estimated by Monte Carlo simulation using $250$ independent replications, each with a warm-up period of $10^6$ time units and a data-collection period of $2\times 10^7$ time units.

\subsection{The \texorpdfstring{$M/M/1{+}GI$}{M/M/1+GI} Models}\label{sec:numerical_mg1gi}

The $M/M/1{+}GI$ class provides an exact baseline: its mean stationary virtual waiting time follows from the single-server specialization of \citet[Eq.~(9.9)]{zeltyn2005call}.
Figure~\ref{fig:MM1_GI} considers exponential patience (top row), Erlang-$2$ patience (middle row), and balanced hyperexponential patience with SCV $4$ (bottom row).
We write these distributions as $M$, $E_2$, and $H_2(4)$, respectively.
The exponential and $H_2(4)$ distributions have $k=1$.
For $E_2$, $k=2$, $f(0)=0$, and $f'(0)=4$ under the unit-mean normalization; consequently, the center panel in the middle row uses hazard-rate scaling rather than WG.

\begin{figure}[htbp]
    \centering
    \includegraphics[width=\textwidth]{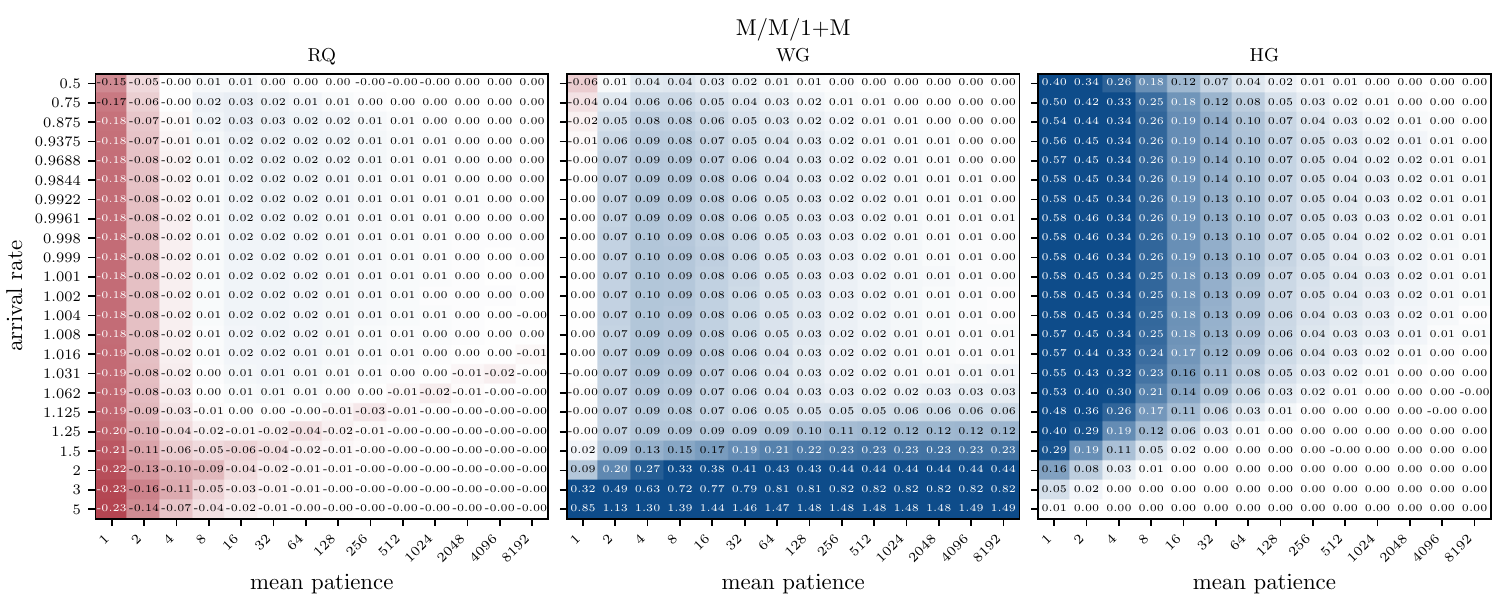}
    \includegraphics[width=\textwidth]{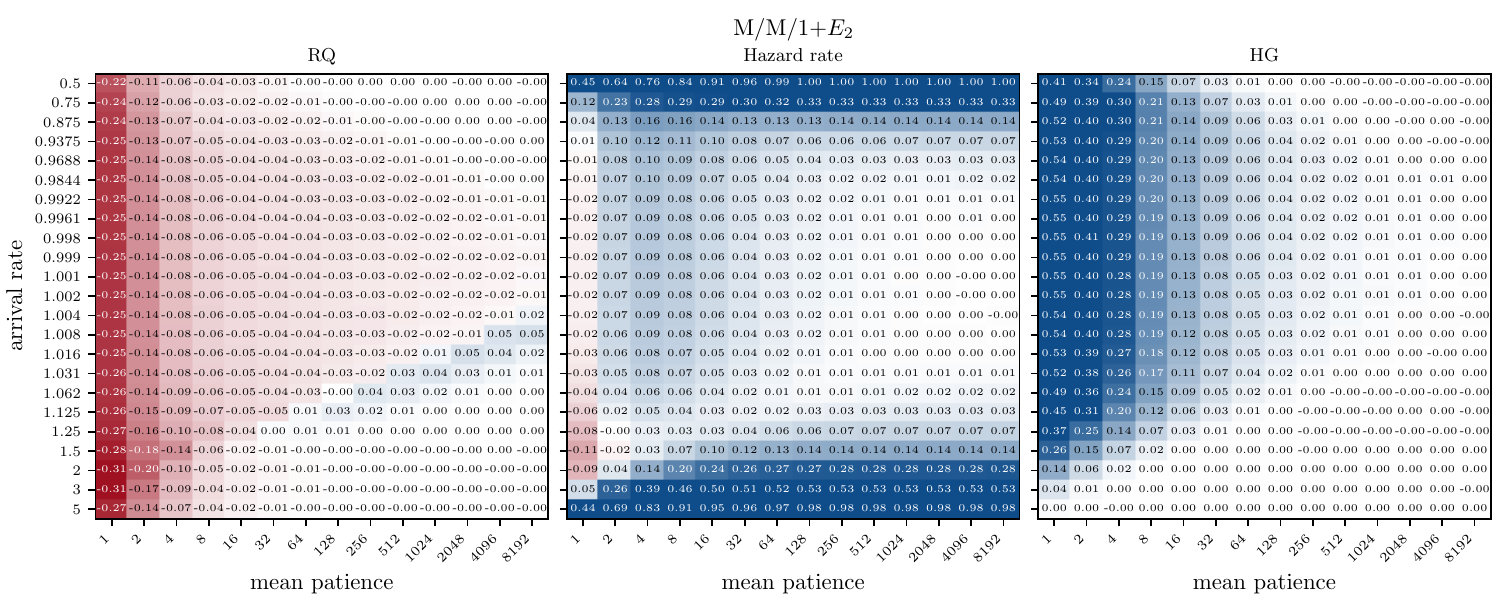}
    \includegraphics[width = \textwidth]{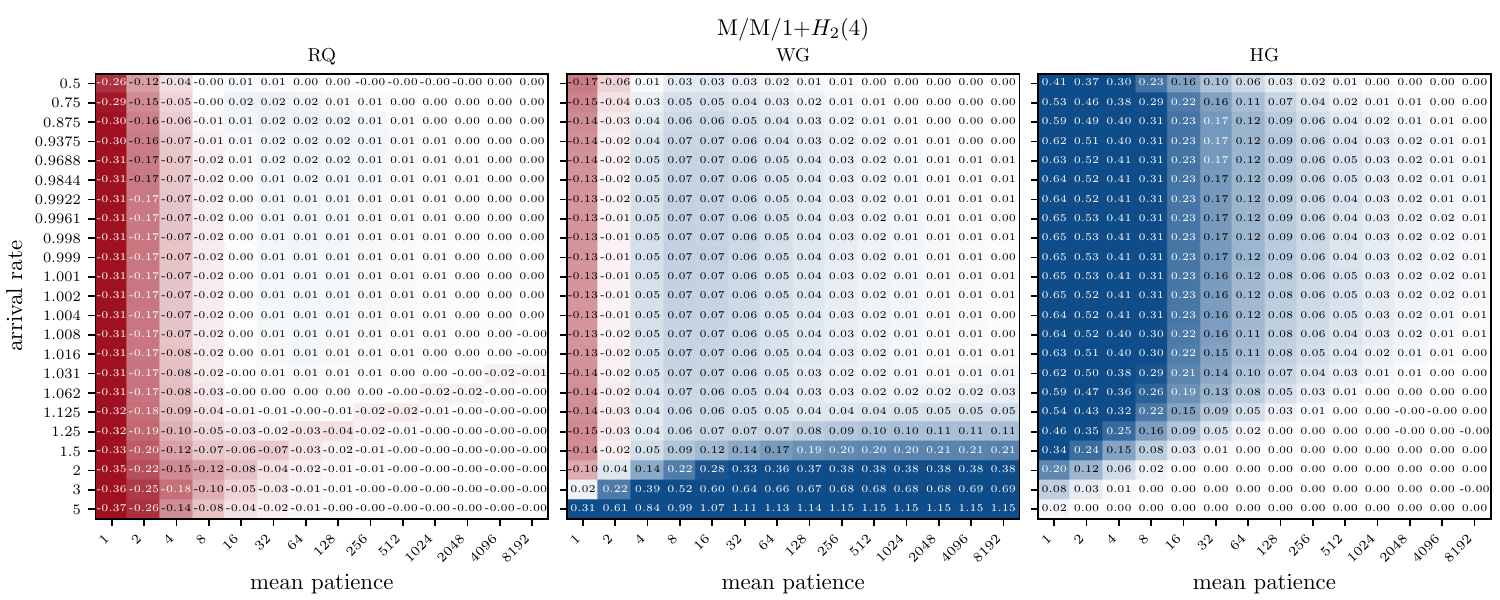}
    \caption{Signed relative errors for refined RQ (left), WG for $f(0)>0$ and hazard-rate scaling for $f(0)=0$ (center), and HG (right). Rows correspond to $M/M/1{+}M$ (top), $M/M/1{+}E_2$ (middle), and $M/M/1{+}H_2(4)$ (bottom).}\label{fig:MM1_GI}
\end{figure}

Figure~\ref{fig:MM1_GI} shows that refined RQ has small error over most of the grid and changes smoothly across loading regimes.
For the models in Figure~\ref{fig:MM1_GI}, the calibration is based only on the critical long-patience limits of $M/M/1{+}M$ and $M/M/1{+}E_2$.
For $H_2(4)$ patience, we use the $k=1$ calibration table; the figure therefore tests whether the local-order calibration transfers to a different patience shape.

For $M/M/1{+}E_2$, the small diagonal sign change under mild overload has displayed errors no larger than about $5\%$.
The WG and hazard-rate approximations are highly accurate near $\rho=1$ with long patience, but can deteriorate in deep underload or overload.
HG is particularly accurate in deep overload, including some short-patience cases, but can overestimate the mean virtual waiting time when patience is short outside deep overload.

The clearest systematic limitation of refined RQ is short patience combined with high abandonment, where it tends to underestimate the mean virtual waiting time.
The error generally decreases as the mean patience time increases.
For mean patience times between $8$ and $32$, most displayed errors are in the single-digit percentage range, although high-variability cases can reach roughly $10\%$--$20\%$ on parts of the grid.
At longer patience times, the displayed errors are generally small, with some residual nonnegligible error in the highest-variability cases.

\subsection{The \texorpdfstring{$GI/GI/1{+}GI$}{GI/GI/1+GI} Models}

When the arrival process is Poisson, steady-state performance is often relatively insensitive to higher-order features of the service-time distribution beyond its mean and variance (provided the third moment is not excessively large).
Consequently, approximation accuracy for $M/GI/1{+}GI$ models typically does not degrade substantially relative to the $M/M/1{+}GI$ baseline.
Figure~\ref{fig:MGI1_GI} illustrates this robustness for lognormal service times: the refined RQ approximation remains accurate for both the $M/LN(1,4)/1{+}H_2(4)$ and $M/LN(1,4)/1{+}E_2$ models over most of the parameter grid.
Here $LN(1,4)$ denotes the lognormal distribution with mean $1$ and SCV $4$.

\begin{figure}[ht]
    \centering
    \includegraphics[width=\textwidth]{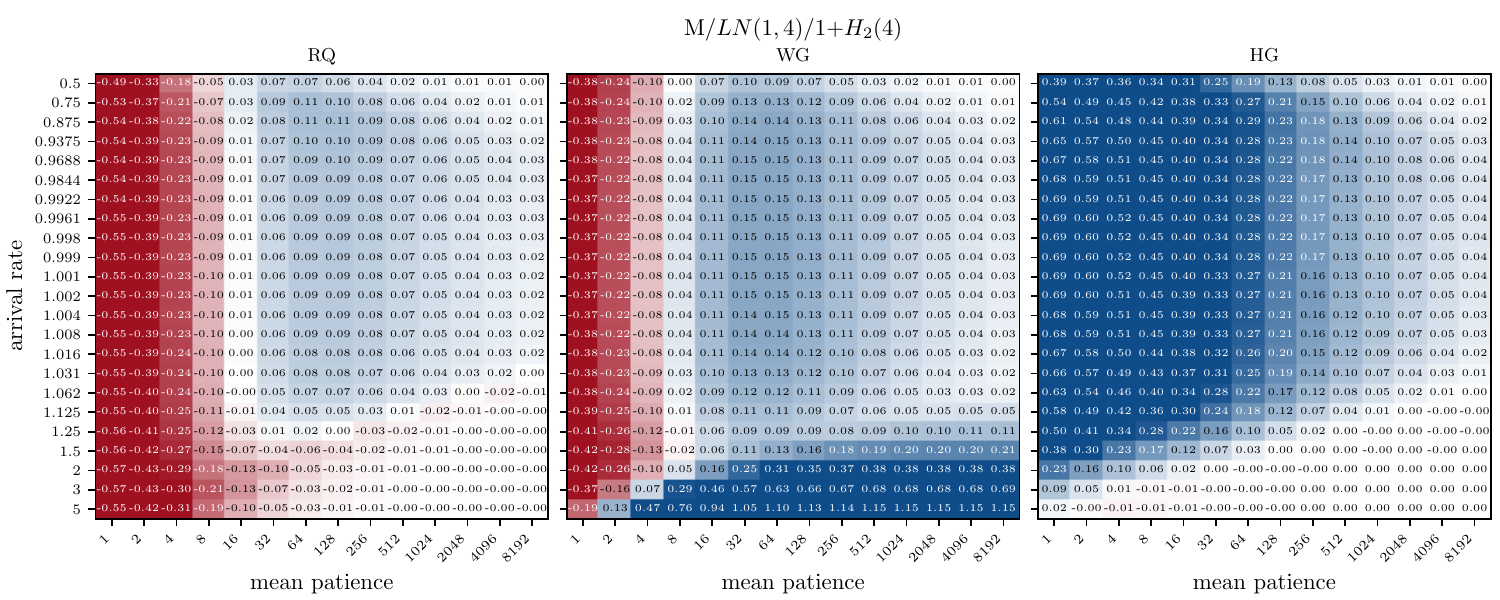}
    \includegraphics[width=\textwidth]{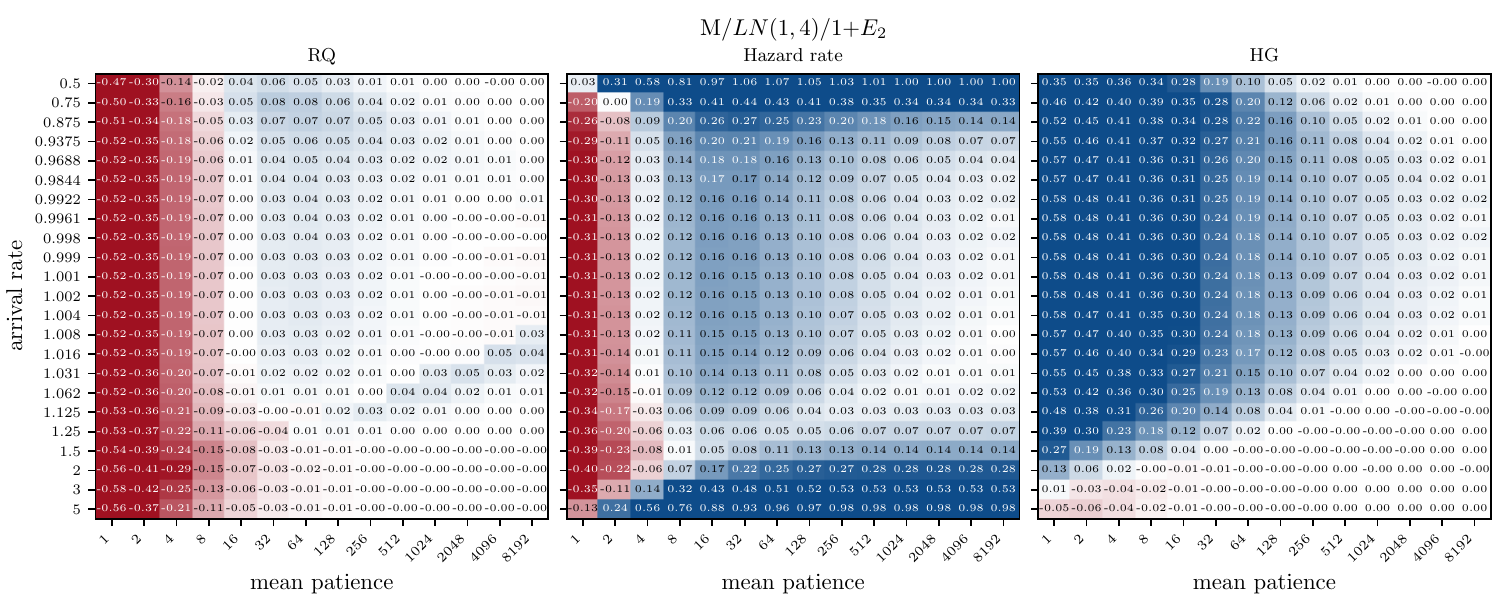}
    \caption{Signed relative error heat maps for the refined RQ approximation (left column), the Ward--Glynn approximation \cite{ward2003diffusion,ward2005diffusion} (center column; replaced by the hazard-rate scaling approximation \cite{reed2008approximating} when $f(0)=0$), and the Huang--Gurvich approximation \cite{huang2018beyond} (right column), for the $M/LN(1,4)/1{+}H_2(4)$ model (top row) and the $M/LN(1,4)/1{+}E_2$ model (bottom row).}
    \label{fig:MGI1_GI}
\end{figure}

Models with non-Poisson renewal arrival processes are usually more challenging, even when the service times are exponential.
Classical approaches often describe renewal input via two moments (rate and SCV $c_a^2$), which is justified by heavy-traffic limits where $c_a^2$ appears explicitly.
Away from heavy traffic, however, the finite-horizon arrival IDC relevant to steady-state performance can differ substantially from its long-run limit $c_a^2$, so a two-moment characterization can be too crude.

Figure~\ref{fig:GIGI1_GI} reports results for the $E_2/LN(1,2)/1{+}E_2$ model (top), the $H_2(4)/LN(1,2)/1{+}H_2(4)$ model (mid), and the $H_2(4)/LN(1,2)/1{+}E_2$ model (bottom).
Here $LN(1,2)$ denotes the lognormal distribution with mean $1$ and SCV $2$.
The benchmark methods behave as expected: the Ward--Glynn and hazard-rate scaling approximations are most accurate near critical loading (where they are theoretically justified), while the Huang--Gurvich approximation can deteriorate when patience times are short.
By contrast, the refined RQ approximation remains stable across regimes.

\begin{figure}[htbp]
    \centering
    \includegraphics[width=\textwidth]{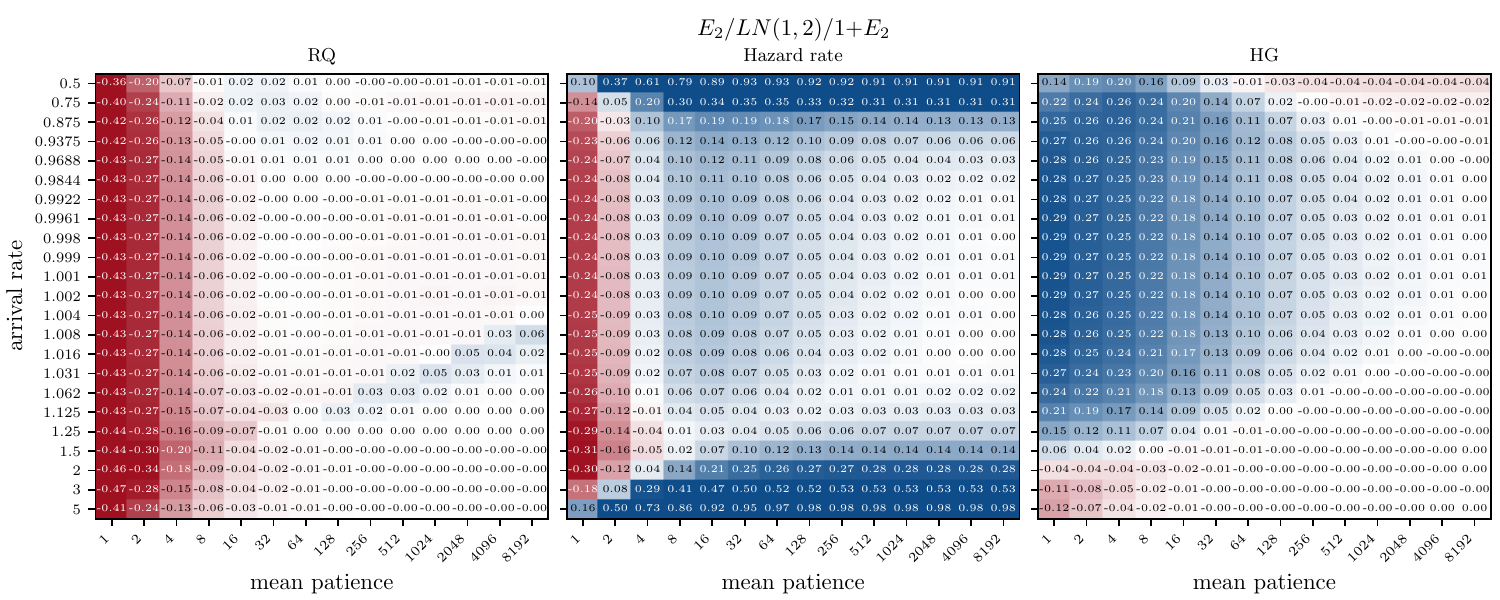}
    \includegraphics[width=\textwidth]{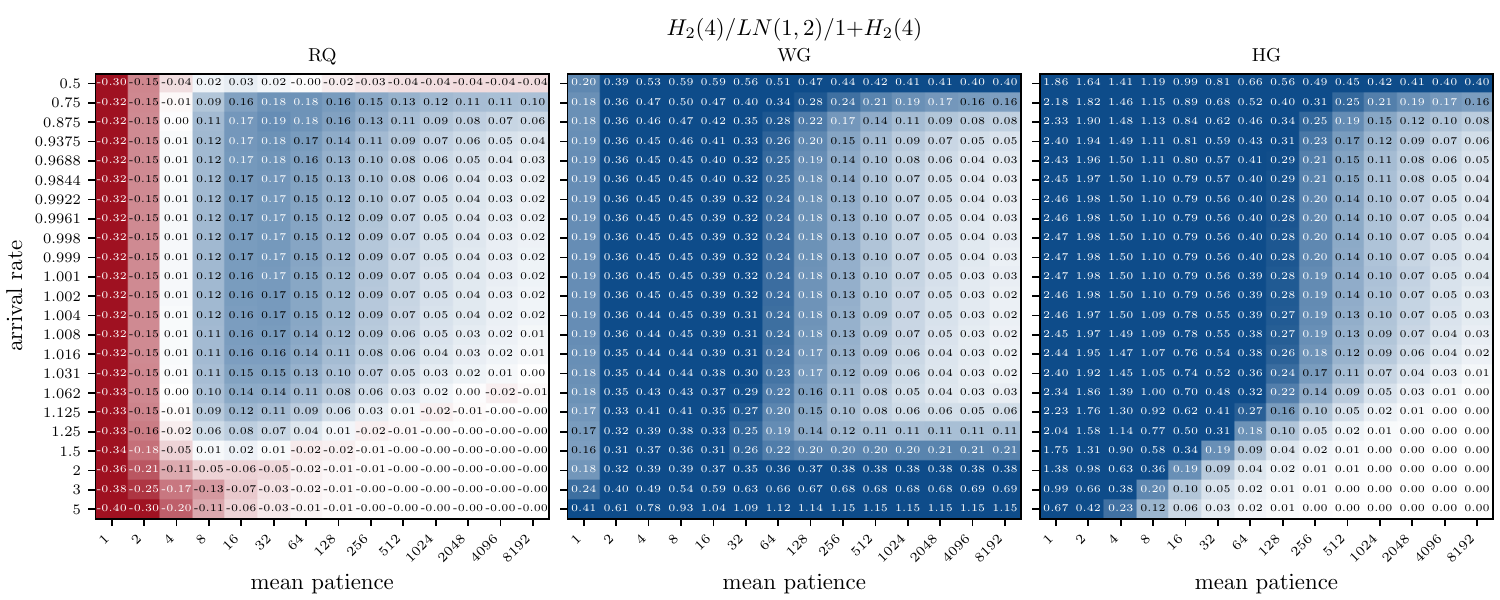}
    \includegraphics[width=\textwidth]{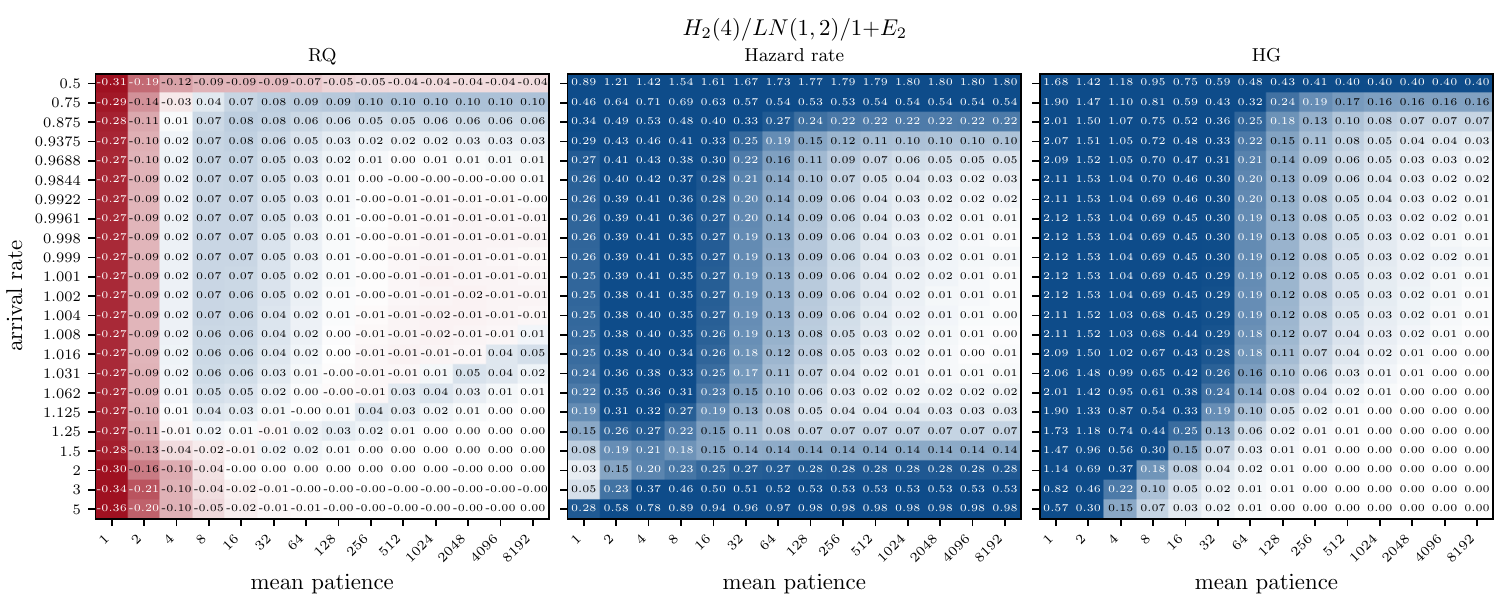}
    \caption{Signed relative error heat maps for the refined RQ approximation (left column), the Ward--Glynn approximation \cite{ward2003diffusion,ward2005diffusion} (center column, when $f(0)>0$; replaced by the hazard-rate scaling approximation \cite{reed2008approximating} when $f(0)=0$), and the Huang--Gurvich approximation \cite{huang2018beyond} (right column), for the $E_2/LN(1,2)/1{+}E_2$ (top row), $H_2(4)/LN(1,2)/1{+}H_2(4)$ (mid row), and $H_2(4)/LN(1,2)/1{+}E_2$ (bottom row) models.}
    \label{fig:GIGI1_GI}
\end{figure}

\subsection{A Heuristic Tandem Extension with Non-Renewal Input}\label{sec:non_renewal}

We close the experiments with a two-node tandem example.
This extension is heuristic: the single-queue theory assumes renewal input, whereas the departure stream feeding the second node is generally non-renewal.
Queue~1 is a stable single-server queue without abandonment; each departure immediately joins Queue~2, where abandonment is allowed.
The target is the mean stationary virtual waiting time at Queue~2.

The extension is possible because the RQ fixed point uses the arrival process through its IDC function, including its long-run limit, rather than through a renewal law directly.
We replace $I_a(t)$ in \eqref{eq:IDW_refined_RQ} by an approximation of the Queue~1 departure IDC obtained from the network propagation method of \citet{whitt2022robust}, and use its long-run limit as $c_{a,2}^2$ in $c_{x,2}^2=c_{a,2}^2+c_{s,2}^2$, $\tilde c_\alpha$, and $\tau$; Appendix~\ref{sec:appendix_tandem} gives the formula.
This substitution preserves finite-horizon dependence that would be lost by reducing the downstream input to only a rate and an asymptotic SCV, but it does not have the same theoretical status as the renewal-input results above.

\begin{figure}[ht]
    \centering
    \includegraphics[width = \textwidth]{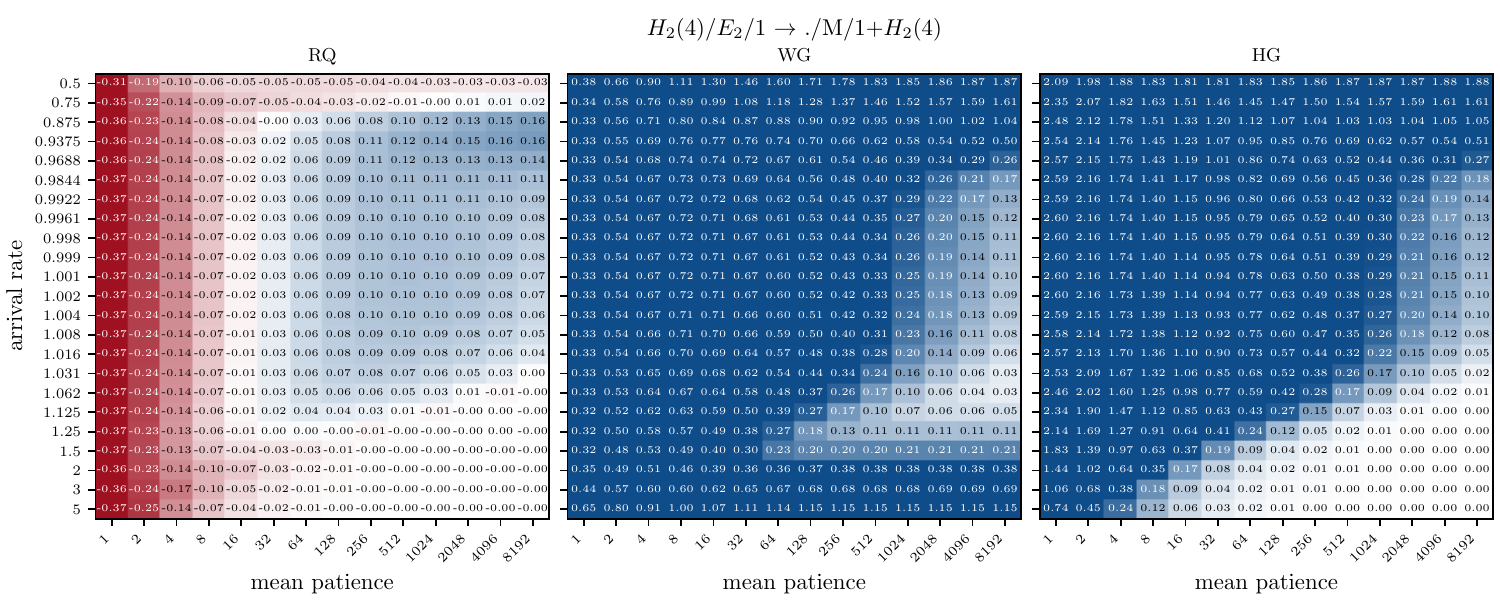}
    \includegraphics[width = \textwidth]{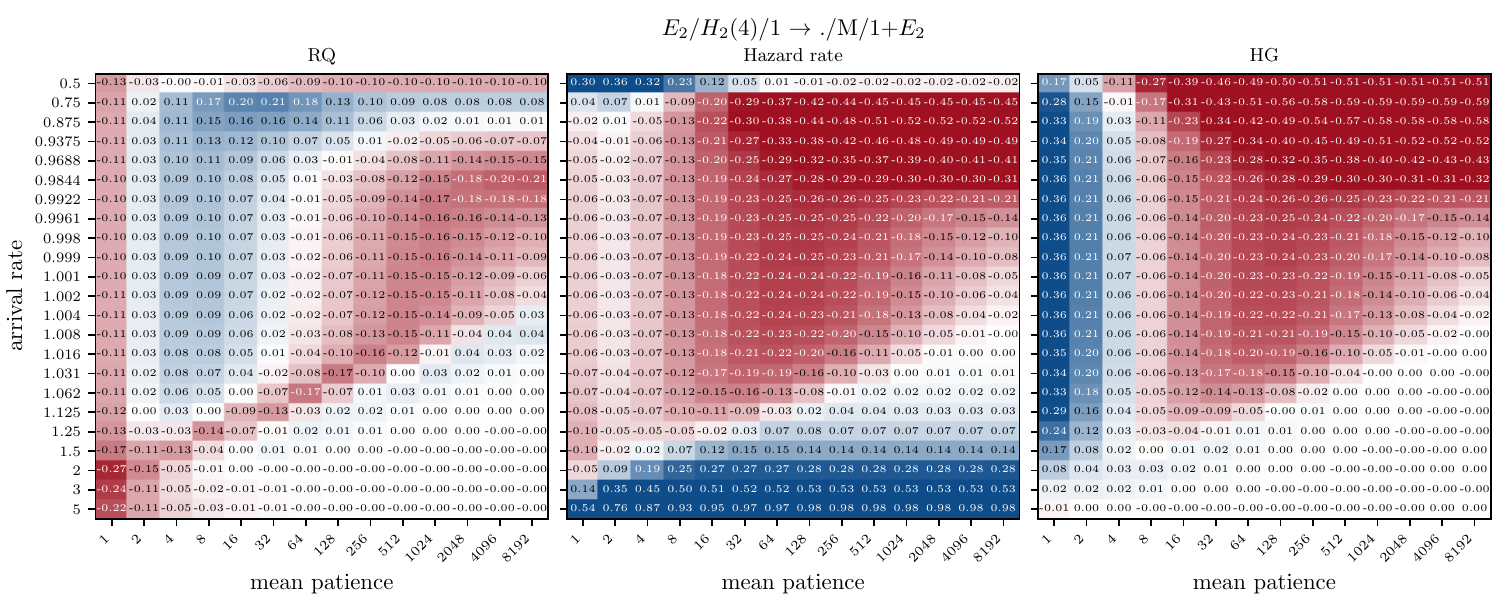}
    \caption{Signed relative errors for the tandem models $H_2(4)/E_2/1\to\cdot/M/1{+}H_2(4)$ (top) and $E_2/H_2(4)/1\to\cdot/M/1{+}E_2$ (bottom). Columns report refined RQ (left), WG for $f(0)>0$ and hazard-rate scaling for $f(0)=0$ (center), and HG (right).}\label{fig:QIS}
\end{figure}

Figure~\ref{fig:QIS} considers
$H_2(4)/E_2/1\to\cdot/M/1{+}H_2(4)$ and
$E_2/H_2(4)/1\to\cdot/M/1{+}E_2$.
Refined RQ remains accurate over most of the displayed grid despite the non-renewal downstream input.
For a mean patience time of at least $4$, the displayed absolute relative error is at most $20\%$ except at one grid point.
At mean patience time $2$, the first tandem model is underestimated by approximately $22\%$--$25\%$ over much of the load grid, consistent with the short-patience limitation observed for single queues.

\section{Conclusion}\label{sec:conclusion}

We developed two Robust Queueing approximations for the mean stationary virtual waiting time in a $GI/GI/1{+}GI$ queue.
Both start from the exact stationary reverse-time supremum and replace each endogenous effective net-input increment by a deterministic mean--standard-deviation envelope evaluated self-consistently at a deterministic trial approximation to the mean stationary virtual waiting time.
The first RQ method uses a deterministic-time-change variance surrogate.
The refined method additionally applies a horizon-dependent factor $w_{c,k}$ that represents the negative feedback from congestion to abandonment and back to effective work.
Both constructions reduce to a unique scalar fixed point.

The analysis distinguishes model identities from approximation steps.
The Poisson effective-arrival compensator and the renewal Palm decomposition are exact in their respective settings; the stationary self-consistency step and both finite-system variance surrogates are approximations.
For the refined variance, the stationary heavy-traffic limit and the fixed-horizon long-patience limit are proved separately, whereas their interpolation in \eqref{eq:IDW_effective_approx} is heuristic.
Heavy-traffic analysis shows that both RQ fixed points recover the underloaded, critical, and overloaded scales, and the refined critical limit retains the variance-reduction function.

Across the numerical grid, refined RQ generally compares favorably with the diffusion benchmarks and remains stable away from critical loading.
Its main systematic weakness is short patience with high abandonment, especially under high primitive variability; this regime is closer to an admission-or-loss mechanism than to accumulated-workload behavior.
The tandem experiment also indicates that IDC propagation can support useful non-renewal-input calculations, but it is an algorithmic extension rather than a proved network approximation.
Natural next steps are a separate loss-regime RQ model, data-driven IDC estimation with uncertainty quantification, and theory for non-renewal inputs and networks with abandonment.

\section*{Acknowledgments}

W. You's research is generously supported by the Hong Kong Research Grants Council [Grant GRF 16212823] and [Theme-based Research Project T32-615/24-R].

\bibliography{refs_RQ_ab}
\bibliographystyle{plainnat}

\clearpage

\appendix

\part*{Appendix}
\addcontentsline{toc}{part}{Appendix}
\etocsetnexttocdepth{subsubsection}

\localtableofcontents

\section{Review of Existing Methods}\label{sec:other_methods}

\paragraph{Exact formula for $M/M/1 + GI$ models.}
Let $H(x)=\int_0^x\bar F_{\alpha}(u)du$; then the single-server specialization of \citet[Eq.~(9.9)]{zeltyn2005call} gives
\begin{align}
    \E[Z] & = \frac{\lambda \int_{0}^{\infty}x \exp(\lambda H(x) - \mu x) dx}{1 + \lambda \int_{0}^{\infty} \exp(\lambda H(x) - \mu x) dx}
    = \frac{ \int_{0}^{\infty}x \exp(\mu \int_0^x (\rho \bar F_{\alpha}(u) - 1) du) dx}{1/\lambda + \int_{0}^{\infty} \exp(\mu \int_0^x (\rho \bar F_{\alpha}(u) - 1) du) dx}. \label{eq:MM1_exact}
\end{align}

\paragraph{Approximation for critically-loaded $GI/GI/1 + GI$ models based on the derivative at $0$.}
Assuming $F'(0)> 0$, from Section 5 of \cite{ward2005diffusion}
\[
    \E[Z] \approx \alpha^{-1/2}\left[\frac{c}{F'(0)} + \frac{\phi\left(-\sqrt{2\mu}c/\sqrt{F'(0)\tilde c_x^2}\right)}{1-\Phi\left(-\sqrt{2\mu}c/\sqrt{F'(0)\tilde c_x^2}\right)}\sqrt{\frac{\tilde c_x^2}{2 \mu F'(0)}}\right], \quad \text{where }\tilde c_x^2 = \rho c_a^2 + (\rho \wedge 1)c_s^2.
\]

\paragraph{Hazard rate scaling approximation for critically-loaded $GI/GI/1 + GI$ models \cite{reed2008approximating}.}

\[
    \E[Z] \approx \frac{\int_0^{\infty}x\exp\left\{\frac{2\mu}{c_x^2}\int_0^x \left[ \log(\bar F_{\alpha}(u)) + (\rho - 1)\right] du\right\} dx}{\int_0^{\infty}\exp\left\{\frac{2\mu}{c_x^2}\int_0^x \left[\log(\bar F_{\alpha}(u)) + (\rho - 1)\right] du\right\} dx},
    \quad \text{where } c_x^2 = c_a^2 + c_s^2.
\]

\paragraph{Universal approximation for $M/GI/1 + GI$ models in \cite{huang2018beyond}.}

\[
    \E[Z] \approx \frac{\int_0^{\infty}x\exp\left\{\frac{2\mu}{(1 + c_s^2)(\rho \wedge 1)}\int_0^x(\rho \bar F_{\alpha}(u) - 1)du\right\} dx}{\int_0^{\infty}\exp\left\{\frac{2\mu}{(1 + c_s^2)(\rho \wedge 1)}\int_0^x(\rho \bar F_{\alpha}(u) - 1)du\right\} dx}.
\]
Comparing with \eqref{eq:MM1_exact}, for this approximation to be exact for the $M/M/1 + GI$ model, one should add an additional constant $1/\lambda$ in the denominator and remove the modifier $\rho\wedge 1$ for the variability parameter $1 + c_s^2$.

\paragraph{Modification of \cite{huang2018beyond} for $GI/GI/1 + GI$ models.}
The formula in \cite{huang2018beyond} can be modified to obtain a naive approximation for $GI/GI/1 + GI$ models with non-Poisson renewal arrival processes.
In particular, this is done by observing that exponential interarrival times have a SCV of $c_a^2 = 1$ and plugging in the corresponding SCV $c_a^2$ of the renewal arrival process.
\[
    \E[Z] \approx \frac{\int_0^{\infty}x\exp\left\{\frac{2\mu}{c_x^2(\rho \wedge 1)}\int_0^x(\rho \bar F_{\alpha}(u) - 1)du\right\} dx}{\int_0^{\infty}\exp\left\{\frac{2\mu}{c_x^2(\rho \wedge 1)}\int_0^x(\rho \bar F_{\alpha}(u) - 1)du\right\} dx},
    \quad \text{where } c_x^2 = c_a^2 + c_s^2.
\]

\section{A Heuristic Tandem-Queue Approximation}\label{sec:appendix_tandem}

This section summarizes how we extend the RQ approximations to a two-node tandem model, in which the downstream queue has abandonment and its input process is the \emph{departure} process from an upstream $GI/GI/1$ queue.
The key point is that, although the downstream arrivals are generally \emph{non-renewal}, the refined RQ approximation only requires the arrival process through its \emph{IDC} function; we therefore approximate the downstream arrival IDC using an IDC-based departure approximation from \cite{whitt2018heavy,whitt2022robust}.

\paragraph{Model.}
Consider two single-server FCFS queues in series.
Queue~1 is a stable $GI/GI/1$ queue without abandonment.
Its stationary departure process is routed to Queue~2, which is a $GI/GI/1{+}GI$ queue with patience-time distribution $F_\alpha$.
The external arrival rate to Queue~1 is $\lambda$, and Queue~1 has traffic intensity $\rho_1<1$; hence the throughput of Queue~1 is $\lambda$ and the arrival rate to Queue~2 is also $\lambda$.
For a stationary counting process $A(\cdot)$ with rate $\lambda$, recall the IDC
\[
    I_A(t)\triangleq\frac{\Var\bigl(A(t)-A(0)\bigr)}{\E\bigl(A(t)-A(0)\bigr)}
    =\frac{\Var\bigl(A(t)-A(0)\bigr)}{\lambda t}, \qquad t>0.
\]
For renewal processes, $I_A(t)$ can be computed numerically from renewal-function or Laplace-transform representations; see, e.g., \cite{whitt2022robust} and references therein.
For general \emph{non-renewal} processes (such as departures from a $GI/GI/1$ queue), we work directly with $I_A(\cdot)$.

\paragraph{Step 1: Approximate the departure IDC from Queue~1.}
Let $I_{a,1}(\cdot)$ be the IDC of the external arrival process to Queue~1 (a renewal process in our experiments), and let $I_{s,1}(\cdot)$ be the IDC of the \emph{equilibrium service renewal process rescaled to rate $\lambda$}.
Let $c_{a,1}^2 \triangleq I_{a,1}(\infty)$ and $c_{s,1}^2 \triangleq I_{s,1}(\infty)$, and define $c_{x,1}^2 \triangleq c_{a,1}^2 + c_{s,1}^2$.

Following \cite{whitt2022robust}, we approximate the stationary departure IDC from Queue~1 by the convex combination
\begin{equation}\label{eq:IDC_dep_app}
    I_{d,1}(t)\approx w_{\rho_1}(t) I_{a,1}(t) + \bigl(1-w_{\rho_1}(t)\bigr) I_{s,1}(t), \qquad t\ge 0,
\end{equation}
where the weight function is
\[
    w_{\rho_1}(t)\triangleq w^*\left(\frac{(1-\rho_1)^2 \lambda t}{\rho_1 c_{x,1}^2}\right),
\]
and $w^*(\cdot)$ is the heavy-traffic limiting weight derived from the canonical RBM correlation structure; an explicit closed form is available (see \cite{whitt2022robust}): for $u>0$,
\begin{equation}\label{eq:wstar_app}
    w^*(u) = \frac{1}{2u} \left( (u^2+2u-1)\bigl(2\Phi(\sqrt{u})-1\bigr) + 2\phi(\sqrt{u})\sqrt{u}(1+u) - u^2 \right),
\end{equation}
where $\phi$ and $\Phi$ are the standard normal density and distribution functions, respectively.
The function $w^*(u)$ is increasing and satisfies $0\le w^*(u)\le 1$ \cite{whitt2022robust}, so \eqref{eq:IDC_dep_app} interpolates smoothly between service-scale variability (small $t$) and arrival-scale variability (large $t$).

\paragraph{Step 2: Use the departure IDC as the downstream arrival IDC.}
Because Queue~2 receives the departures from Queue~1, its true arrival IDC equals the true departure IDC from Queue~1. Using the approximation in \eqref{eq:IDC_dep_app}, we set
\begin{equation}\label{eq:IDC_arrival2_app}
    I_{a,2}(t)\approx I_{d,1}(t), \qquad t\ge 0.
\end{equation}

\paragraph{Step 3: The IDW input to the refined RQ approximation at Queue~2.}
In the refined RQ approximation for the $GI/GI/1{+}GI$ model, the arrival process enters through the \emph{effective IDW} (see \eqref{eq:IDW_refined_RQ}--\eqref{eq:IDW_effective_approx}).
For the tandem system, we use \eqref{eq:IDC_arrival2_app} and set
\begin{equation}\label{eq:IDW_tandem_app}
    \hat I_{w,2}(t) \triangleq \frac{I_{a,2}(t)}{\rho_2\vee 1} + \left(1-\frac{1}{\rho_2\vee 1}\right) + c_{s,2}^2,
    \qquad t\ge 0,
\end{equation}
where $\rho_2=\lambda/\mu_2$ is the nominal traffic intensity of Queue~2 (ignoring abandonment), and $c_{s,2}^2$ is the service-time SCV at Queue~2 (e.g., $c_{s,2}^2=1$ for exponential service).
The abandonment-modulated IDW is then obtained as in \eqref{eq:IDW_effective_approx} by multiplying $\hat I_{w,2}(t)$ with the abandonment factor $w_{\tilde c_\alpha,k}(\alpha^{2h}\tau t)$ from the refined RQ algorithm, using $c_{a,2}^2\approx\lim_{t\to\infty}I_{d,1}(t)$ in $c_{x,2}^2=c_{a,2}^2+c_{s,2}^2$, $\tilde c_\alpha$, and $\tau$.

To approximate the mean stationary virtual waiting time in Queue~2 for a tandem system: (i) compute/approximate $I_{a,1}(\cdot)$ and $I_{s,1}(\cdot)$, (ii) approximate $I_{d,1}(\cdot)$ via \eqref{eq:IDC_dep_app}--\eqref{eq:wstar_app}, (iii) set $I_{a,2}(\cdot)\approx I_{d,1}(\cdot)$ and $c_{a,2}^2\approx\lim_{t\to\infty}I_{d,1}(t)$, and (iv) run the refined RQ procedure for Queue~2 using \eqref{eq:IDW_tandem_app} in place of \eqref{eq:IDW_refined_RQ}.
This is precisely the IDC-based propagation mechanism advocated in \cite{whitt2022robust}.

\section{Heuristic Extensions for Other Performance Measures}\label{sec:heuristic}

Additional steady-state performance measures can be approximated by combining the RQ approximation for the mean virtual waiting time with standard identities for $GI/GI/1{+}GI$ queues.
We describe three such extensions: the abandonment probability, the mean waiting time of served customers, and the effective queue length.
These extensions are heuristic in nature, and we do not provide theoretical guarantees for their accuracy.

Let $Z_{\mathrm{RQ}}$ denote the steady-state RQ approximation of the mean virtual waiting time, i.e., the solution of \eqref{eq:RQ_ab_1} or \eqref{eq:RQ_ab_2} depending on the choice of RQ algorithm.

\paragraph{Abandonment probability.}
Let $p_{\mathrm{ab}}$ denote the steady-state probability that an arriving customer abandons.
In steady state, $p_{\mathrm{ab}} = \E\left[F_\alpha(W)\right]$ and $W=Z(T_i-)$, where $W$ is the offered waiting time seen by an arrival.
As a first-order approximation, we replace $W$ by its RQ mean and set
\begin{equation}\label{eq:ab_approx}
    p_{\mathrm{ab}} \approx F_\alpha(Z_{\mathrm{RQ}}).
\end{equation}

\paragraph{Mean waiting time of served customers.}
For the $GI/GI/1$ model without abandonment, it is well-known that the mean steady-state workload (virtual waiting time) and the mean steady-state waiting time are connected by Pollaczek-Khintchine formula
\[
    \E[Z_{GI/GI/1}] = \rho\left(\E[W_{GI/GI/1}] + \frac{c_s^2 + 1}{2\mu}\right).
\]
For the $GI/GI/1 + GI$ model, \cite{baccelli1984single} derives the corresponding extension for the waiting time of a customer \emph{conditional on being served} in the $GI/GI/1{+}GI$ model:
\begin{equation}\label{eq:PK_abandonment}
    \E[Z] = (1-p_{\mathrm{ab}})\rho\left(\E[W] + \frac{c_s^2 + 1}{2\mu}\right),
\end{equation}
where $p_{\mathrm{ab}}$ is the steady-state probability of abandonment.
Combining \eqref{eq:ab_approx} and \eqref{eq:PK_abandonment}, we approximate
\begin{equation}\label{eq:Wsvc_approx}
    \E[W] \approx \max\left\{0, \frac{Z_{\mathrm{RQ}}}{\rho\bigl(1-F_\alpha(Z_{\mathrm{RQ}})\bigr)} - \frac{c_s^2+1}{2\mu}\right\}.
\end{equation}

\paragraph{Effective queue length.}
Let $Q_0$ denote the steady-state number of customers \emph{waiting} who will eventually enter service (i.e., the queue length associated with the \emph{effective} arrival stream).
Little's law applied to the effective stream gives $\E[Q_0] = \lambda(1-p_{\mathrm{ab}}) \E[W]$.
Using \eqref{eq:ab_approx} and \eqref{eq:Wsvc_approx} yields
\[
    \E[Q_0] \approx \max\left\{0, \mu Z_{\mathrm{RQ}} - \rho\bigl(1-F_\alpha(Z_{\mathrm{RQ}})\bigr)\frac{c_s^2+1}{2}\right\}.
\]

\section{Proofs}\label{sec:proof}

\subsection{Proof of Lemma~\ref{lm:drift_poisson}}\label{sec:proof_mean_arrival}

\begin{proof}
Taking expectations of the increment
\[
    A_0(t)-A_0(t-s)=M_0(t)-M_0(t-s)+\Lambda_0(t)-\Lambda_0(t-s)
\]
gives
\[
    \E[A_0(t)-A_0(t-s)] = \lambda \E\left[\int_{t-s}^{t}\bar F_\alpha(Z(u-))du\right].
\]
For the effective work input, let $M_Y(t)\triangleq Y(t)-\Lambda_Y(t)$, where $\Lambda_Y$ is the compensator of $Y(\cdot)$ $ \Lambda_Y(t)\triangleq \frac{\lambda}{\mu}\int_0^t \bar F_{\alpha} \bigl(Z(u-)\bigr) du. $
Then $M_Y$ is an integrable martingale, because the service marks are independent of the arrival process and $\E[Y(t)]\le \E[\sum_{i=1}^{A(t)}V_i]=\lambda t/\mu<\infty$.
Taking expectations of
\[
    Y(t)-Y(t-s)=M_Y(t)-M_Y(t-s)+\Lambda_Y(t)-\Lambda_Y(t-s)
\]
gives
\[
    \E[Y(t)-Y(t-s)] = \frac{\lambda}{\mu} \E\left[\int_{t-s}^{t}\bar F_\alpha(Z(u-))du\right].
\]
Since $N(t)-N(t-s)=Y(t)-Y(t-s)-s$, the two nonstationary identities follow.

Assume now that $Z(\cdot)$ is strictly stationary.
Because $0\le \bar F_\alpha(Z(u-))\le1$, Tonelli's theorem yields
\[
    \E\left[\int_{t-s}^{t}\bar F_\alpha(Z(u-))du\right] = \int_{t-s}^{t}\E[\bar F_\alpha(Z(u-))]du.
\]
At every deterministic $u$, the stationary simple arrival process has no point at $u$ a.s., so $Z(u-)=Z(u)$ a.s.
Stationarity therefore gives
\[
    \E\left[\int_{t-s}^{t}\bar F_\alpha(Z(u-))du\right] = s\E[\bar F_\alpha(Z(0))].
\]
Substituting this identity into the two nonstationary formulas proves the stationary formulas.
\end{proof}

\subsection{Proofs of Lemma~\ref{lm:correction} and Proposition~\ref{prop:palm_optimizer_scales}}

The lemma follows from Campbell--Mecke and one-cycle Palm inversion, and the proposition applies the resulting uniform $O(\alpha s)$ bound at the three RQ horizon scales.

\begin{proof}[Proof of Lemma~\ref{lm:correction}]
Campbell--Mecke \citep[Eq.~(21)]{bremaud1993stationary} and time stationarity give
\[
    \E[\delta_\alpha(s)]=\lambda s\{\E^0[\bar F_\alpha(Z_{\mathrm{pre}})]-\E[\bar F_\alpha(Z(0))]\},
\]
which proves \eqref{eq:palm_delta}.
One-cycle Palm inversion and invariance of the arrival Palm law under the next-arrival shift \citep[Eqs.~(22)--(23)]{bremaud1993stationary} give
\[
    \E[\bar F_\alpha(Z(0))]
    =\lambda\E^0\left[\int_0^U\bar F_\alpha\bigl((Z_{\mathrm{post}}-u)^+\bigr)du\right],
    \qquad
    Z_{\mathrm{pre}}\stackrel{d}{=}(Z_{\mathrm{post}}-U)^+.
\]
For $x,u\ge0$,
\[
    \bar F_\alpha\bigl((x-u)^+\bigr)=\bar F_\alpha(x)+\alpha\int_0^{u\wedge x}f\bigl(\alpha(x-v)\bigr)dv.
\]
Substitution into these identities and Fubini's theorem yield
\[
    \Delta_\alpha
    =\E^0\left[(1-\lambda U)\bar F_\alpha(Z_{\mathrm{post}})
    +\alpha\int_0^{U\wedge Z_{\mathrm{post}}}\{1-\lambda(U-v)\}f\bigl(\alpha(Z_{\mathrm{post}}-v)\bigr)dv\right].
\]
Under the renewal Palm law, $U$ has the ordinary interarrival distribution and is independent of the past and the marks at time $0$, hence it is independent of $Z_{\mathrm{post}}$ and $\E^0[U]=1/\lambda$.
The first term therefore has mean zero, proving \eqref{eq:Delta_alpha_density}.
Finally, boundedness of $f$, \eqref{eq:palm_delta}, and $\lambda^2\E^0[U^2]=1+c_a^2$ give
\[
    |\E[\delta_\alpha(s)]|
    \le\lambda\alpha s\|f\|_\infty\E^0\left[U+\frac{\lambda U^2}{2}\right]
    =\frac{3+c_a^2}{2}\|f\|_\infty\alpha s,
\]
which proves \eqref{eq:palm_linear_bound}.
\end{proof}

\begin{proof}[Proof of Proposition~\ref{prop:palm_optimizer_scales}]
Let $C=(3+c_a^2)\|f\|_\infty/2$.
Equation~\eqref{eq:palm_linear_bound} gives
\begin{align*}
    (1-\rho_\alpha)\sup_{s\le M(1-\rho_\alpha)^{-2}}|\E[\delta_\alpha(s)]|&\le CM\alpha/(1-\rho_\alpha),\\
    \alpha^h\sup_{s\le M\alpha^{-2h}}|\E[\delta_\alpha(s)]|&\le CM\alpha^{1-h},\\
    \alpha^{1-\gamma/k}\sup_{s\le M\alpha^{-1-(k-1)\gamma/k}}|\E[\delta_\alpha(s)]|&\le CM\alpha^{1-\gamma}.
\end{align*}
The first bound vanishes because $1-\rho_\alpha\sim(-c)\alpha^\gamma$ and $\gamma<h<1$.
The other two vanish because $h<1$ and $0<\gamma<h<1$.
\end{proof}

\subsection{General Heavy-Traffic Lemmas}

We repeatedly use the following uniform small-argument expansion to control $F(\alpha x)$, $\bar F(\alpha x)$, and $1/\bar F(\alpha x)$ as $\alpha\downarrow 0$.

\begin{lemma}[Uniform small-argument expansion]\label{lem:uniform_F_expand}
Suppose Assumption~\ref{assumption:F} holds with index $k\ge 1$.
Then for every $R>0$,
\begin{equation}\label{eq:F_uniform_expand}
    \sup_{0\le x\le R} \left| \frac{F(\alpha x)}{\alpha^k} - \beta x^k \right| \longrightarrow 0
    \qquad\text{as }\alpha\downarrow 0.
\end{equation}
Consequently, for every $R>0$,
\begin{align}
    \bar F(\alpha x) & = 1-\beta \alpha^k x^k + o(\alpha^k),
    \qquad\text{uniformly for }x\in[0,R], \label{eq:Fbar_uniform_expand}\\
    \frac{1}{\bar F(\alpha x)} & = 1+\beta \alpha^k x^k + o(\alpha^k),
    \qquad\text{uniformly for }x\in[0,R]. \label{eq:invFbar_uniform_expand}
\end{align}
\end{lemma}

\begin{proof}
Fix $R>0$.
By Taylor's theorem with the mean-value remainder, for each $y\in[0,\alpha R]$ there exists $\theta=\theta(y)\in(0,1)$ such that
\[
    F(y) = \sum_{j=0}^{k-1}\frac{F^{(j)}(0)}{j!}y^j + \frac{F^{(k)}(\theta y)}{k!}y^k.
\]
Assumption~\ref{assumption:F} implies $F^{(j)}(0)=0$ for $j<k$, hence $F(y)=F^{(k)}(\theta y)y^k/k!$.
Taking $y=\alpha x$ with $x\in[0,R]$ gives
\[
    \left| \frac{F(\alpha x)}{\alpha^k}-\beta x^k \right| = \frac{x^k}{k!} \bigl|F^{(k)}(\theta\alpha x)-F^{(k)}(0)\bigr|.
\]
Therefore,
\[
    \sup_{0\le x\le R} \left| \frac{F(\alpha x)}{\alpha^k}-\beta x^k \right|
    \le \frac{R^k}{k!}\sup_{0\le z\le \alpha R}\bigl|F^{(k)}(z)-F^{(k)}(0)\bigr| \longrightarrow 0,
\]
since $F^{(k)}$ is continuous at $0$.
This proves \eqref{eq:F_uniform_expand}.
Equation \eqref{eq:Fbar_uniform_expand} follows from $\bar F=1-F$.

To obtain \eqref{eq:invFbar_uniform_expand}, note that \eqref{eq:Fbar_uniform_expand} implies $\inf_{0\le x\le R}\bar F(\alpha x)\to 1$, so for sufficiently small $\alpha$ we have $\inf_{0\le x\le R}\bar F(\alpha x)\ge 1/2$.
Then, uniformly over $x\in[0,R]$,
\[
    \frac{1}{\bar F(\alpha x)} =\frac{1}{1-F(\alpha x)} =1+F(\alpha x)+O\bigl(F(\alpha x)^2\bigr)
    =1+\beta\alpha^k x^k + o(\alpha^k),
\]
using \eqref{eq:F_uniform_expand} and the fact that $F(\alpha x)=O(\alpha^k)$ uniformly on $[0,R]$.
\end{proof}

\begin{lemma}[Laplace concentration with a decreasing tail]\label{lem:laplace_concentration_tail}
Let $r_\alpha\to\infty$.
Let $G_\alpha$ be continuously differentiable functions on $[0,\infty)$, and let $G$ be continuous.
Assume that $G_\alpha\to G$ uniformly on compact subsets of $[0,\infty)$.
Assume that $G$ has a unique maximizer $s^*\in(0,\infty)$ and $G(s^*)>0$.
Assume further that there exist $S>s^*$ and $\eta>0$ such that $G_\alpha'(s)\le-\eta$ for all $s\ge S$ and all sufficiently small $\alpha$.
Then the probability measures
\[
    \nu_\alpha(ds) = \frac{\exp\{r_\alpha G_\alpha(s)\}ds} {\int_0^\infty\exp\{r_\alpha G_\alpha(u)\}du}
\]
concentrate weakly at $s^*$.
Consequently,
\[
    \frac{\int_0^\infty s\exp\{r_\alpha G_\alpha(s)\}ds} {\int_0^\infty\exp\{r_\alpha G_\alpha(s)\}ds} \to s^* .
\]
\end{lemma}

\begin{proof}
Fix $\varepsilon>0$.
Choose $S>s^*+\varepsilon$ as in the statement.
By uniqueness of the maximizer and compactness of $[0,S]$, there exists $\zeta>0$ such that
\[
    \sup_{\substack{0\le s\le S\\ |s-s^*|\ge\varepsilon}}G(s) \le G(s^*)-3\zeta .
\]
Uniform convergence on $[0,S]$ gives the same bound with $G_\alpha$ and gap $2\zeta$ for all sufficiently small $\alpha$.
Also, for some interval $I\subset(s^*-\varepsilon,s^*+\varepsilon)$ of positive length, $G_\alpha(s)\ge G(s^*)-\zeta$ on $I$ for all sufficiently small $\alpha$.
Thus the mass outside the $\varepsilon$-neighborhood of $s^*$ but inside $[0,S]$ is exponentially negligible relative to the mass on $I$.

For $s\ge S$, the derivative bound gives $G_\alpha(s)\le G_\alpha(S)-\eta(s-S)$.
The compact-uniform convergence and the fact that $s^*$ is the unique maximizer give $G_\alpha(S)\le G(s^*)-2\zeta$ for all sufficiently small $\alpha$.
Hence
\[
    \int_S^\infty (1+s)e^{r_\alpha G_\alpha(s)}ds
    \le e^{r_\alpha(G(s^*)-2\zeta)} \int_S^\infty(1+s)e^{-r_\alpha\eta(s-S)}ds,
\]
which is also exponentially negligible relative to $\int_I e^{r_\alpha G_\alpha(s)}ds$.
This proves concentration at $s^*$ and uniform integrability of the first moment under $\nu_\alpha$.
The ratio convergence follows.
\end{proof}

\subsection{Proof of Theorem~\ref{Thm:HT_exact}}

\begin{proof}
The proof starts from the exact stationary formula and applies one ratio identity under the three natural spatial scalings.
Let $c_\alpha\triangleq\alpha^{-\gamma}(\rho_\alpha-1)$, so $c_\alpha\to c$ and $\lambda_\alpha\to\mu$, and define
\[
    K_\alpha(x)\triangleq\lambda_\alpha\int_0^x\bar F(\alpha u)du-\mu x,
    \qquad
    J_{\alpha,m}\triangleq\int_0^\infty x^m e^{K_\alpha(x)}dx,
    \quad m=0,1.
\]
The single-server specialization of \citet[Eq.~(9.9)]{zeltyn2005call} gives $\E[Z_\alpha]=\lambda_\alpha J_{\alpha,1}/(1+\lambda_\alpha J_{\alpha,0})$.
For any $a_\alpha>0$, set $L_\alpha(s)\triangleq K_\alpha(s/a_\alpha)$.
Changing variables gives
\begin{equation}\label{eq:MMG_scaled_ratio}
    a_\alpha\E[Z_\alpha]
    =
    \frac{\int_0^\infty s e^{L_\alpha(s)}ds}{\int_0^\infty e^{L_\alpha(s)}ds}
    \frac{\lambda_\alpha J_{\alpha,0}}{1+\lambda_\alpha J_{\alpha,0}}.
\end{equation}
Thus each regime reduces to a scaled integral ratio and the verification that $J_{\alpha,0}\to\infty$.

\paragraph{Underloaded ($c<0$ and $\gamma<h$).}
Since $c<0$, we have $d_\alpha\triangleq1-\rho_\alpha>0$ for all sufficiently small $\alpha$.
Set $a_\alpha\triangleq\mu d_\alpha$, so $a_\alpha\sim\mu(-c)\alpha^\gamma$ and
\[
    L_\alpha(s)=-s-\lambda_\alpha\int_0^{s/a_\alpha}F(\alpha u)du\le-s.
\]
For each fixed $s$, Lemma~\ref{lem:uniform_F_expand}, applied with small parameter $\alpha/a_\alpha$, gives
\[
    \lambda_\alpha\int_0^{s/a_\alpha}F(\alpha u)du
    =\frac{\lambda_\alpha\beta}{k+1}\frac{\alpha^k}{a_\alpha^{k+1}}s^{k+1}(1+o(1))
    \longrightarrow0,
\]
because $k-(k+1)\gamma>0$ when $\gamma<h$.
Dominated convergence yields $\int_0^\infty s^m e^{L_\alpha(s)}ds\to m!$ for $m=0,1$.
Since $J_{\alpha,0}=a_\alpha^{-1}(1+o(1))\to\infty$, \eqref{eq:MMG_scaled_ratio} gives $a_\alpha\E[Z_\alpha]\to1$.
Using $\E[Z_{M/M/1}]=\rho_\alpha/(\mu d_\alpha)$ and $\rho_\alpha\to1$ proves both underloaded limits.

\paragraph{Critically loaded ($\gamma\ge h$).}
Set $a_\alpha\triangleq\alpha^h$.
Then
\[
    L_\alpha(s)
    =\mu c_\alpha\alpha^{\gamma-h}s
    -\lambda_\alpha\alpha^{-h}\int_0^sF(\alpha^{1-h}v)dv.
\]
Lemma~\ref{lem:uniform_F_expand}, $k(1-h)=h$, and $\lambda_\alpha\to\mu$ imply compact-uniform convergence to
\[
    c\mu s\mathds{1}\{\gamma=h\}-\frac{\mu\beta}{k+1}s^{k+1}.
\]
Choose $\delta>0$ such that $F(y)\ge(\beta/2)y^k$ on $[0,\delta]$, and set $R_\alpha\triangleq\delta\alpha^{-(1-h)}$.
For some positive constants $C,c_0,c_1,c_2$ and all sufficiently small $\alpha$,
\[
    L_\alpha(s)\le Cs-c_0s^{k+1}\quad(0\le s\le R_\alpha),
    \qquad
    L_\alpha(R_\alpha)\le-\frac{c_1}{\alpha},
    \qquad
    L_\alpha'(s)\le-c_2\alpha^{-h}\quad(s\ge R_\alpha).
\]
The first bound follows from $k(1-h)=h$, the second from $R_\alpha^{k+1}=\delta^{k+1}/\alpha$, and the third from $F(\alpha^{1-h}s)\ge F(\delta)$ for $s\ge R_\alpha$.
These bounds provide the integrable envelope $e^{Cs-c_0s^{k+1}}$ before $R_\alpha$ and imply $\int_{R_\alpha}^\infty(1+s)e^{L_\alpha(s)}ds\to0$.
Dominated convergence therefore gives, for $m=0,1$,
\[
    \int_0^\infty s^m e^{L_\alpha(s)}ds
    \longrightarrow
    \int_0^\infty s^m
    \exp\left\{c\mu s\mathds{1}\{\gamma=h\}-\frac{\mu\beta}{k+1}s^{k+1}\right\}ds.
\]
The limiting denominator is positive, so $J_{\alpha,0}\to\infty$, and \eqref{eq:MMG_scaled_ratio} gives \eqref{eq:HT_exact}.

\paragraph{Overloaded ($c>0$ and $\gamma<h$).}
Set $a_\alpha\triangleq\alpha^{1-\gamma/k}$ and $r_\alpha\triangleq\alpha^{\gamma/h-1}$, so $r_\alpha\to\infty$.
Then $L_\alpha=r_\alpha G_\alpha$, where
\[
    G_\alpha(s)
    \triangleq
    \mu c_\alpha s
    -\lambda_\alpha\int_0^s\alpha^{-\gamma}F(\alpha^{\gamma/k}v)dv.
\]
Lemma~\ref{lem:uniform_F_expand} gives compact-uniform convergence to $G(s)=\mu cs-\mu\beta s^{k+1}/(k+1)$.
The function $G$ has the unique maximizer $s^*=(c/\beta)^{1/k}$, and $G(s^*)=\mu kcs^*/(k+1)>0$.
Choose $S>s^*$ so that $\beta S^k>4c$, and choose $\delta>0$ so that $F(y)\ge(\beta/2)y^k$ on $[0,\delta]$.
For some $\eta>0$ and all sufficiently small $\alpha$,
\[
    G_\alpha'(s)
    \le
    \begin{cases}
        \mu c_\alpha-\lambda_\alpha\beta s^k/2, & S\le s\le\delta\alpha^{-\gamma/k},\\
        \mu c_\alpha-\lambda_\alpha\alpha^{-\gamma}F(\delta), & s\ge\delta\alpha^{-\gamma/k},
    \end{cases}
    \le-\eta.
\]
Lemma~\ref{lem:laplace_concentration_tail} gives
\[
    \frac{\int_0^\infty s e^{L_\alpha(s)}ds}{\int_0^\infty e^{L_\alpha(s)}ds}\longrightarrow s^*.
\]
Compact-uniform convergence gives an interval around $s^*$ on which $G_\alpha\ge G(s^*)/2$, so $J_{\alpha,0}\to\infty$.
Equation~\eqref{eq:MMG_scaled_ratio} now yields $\alpha^{1-\gamma/k}\E[Z_\alpha]\to(c/\beta)^{1/k}$.
This completes the proof.
\end{proof}

\subsection{Proof of Theorem~\ref{thm:HT_limit}}

Fix $T>0$ and restrict all processes to $[0,T]$.
The proof has three steps.
First, a pathwise reflection lemma gives compact containment and the continuity needed for the final continuous-mapping argument.
Second, a single thinning lemma controls the abandonment and service martingales and proves $C$-tightness\footnote{A sequence is $C$-tight if it is tight in the Skorokhod $J_1$ topology and every subsequential weak limit is supported on continuous paths.} of the workload.
Third, the abandonment compensator converges to the nonlinear drift, after which the theorem follows from the reflected integral equation.

Let $\D([0,T],\R)$ denote the space of real-valued right-continuous functions with left limits on $[0,T]$, endowed with the Skorokhod $J_1$ topology.
For $x\in\D([0,T],\R)$, write $\|x\|_T\triangleq\sup_{0\le t\le T}|x(t)|$ and $\omega_T(x,\delta)\triangleq\sup\{|x(t)-x(s)|:0\le s\le t\le T,\ t-s\le\delta\}$.

\paragraph{Reflection map.}
For $y\in\D([0,T],\R)$, define
\[
    \Gamma(y)(t)\triangleq y(t)-\inf_{0\le s\le t}(y(s)\wedge0).
\]
If $y(0)\ge0$, then $(\Gamma(y),\Gamma(y)-y)$ is the unique solution of the one-dimensional Skorokhod problem with input $y$.

\begin{lemma}[Reflection and polynomial drift]\label{lem:Gamma_increment_monotone}
For $y_1,y_2\in\D([0,T],\R)$,
\[
    \|\Gamma(y_1)-\Gamma(y_2)\|_T\le2\|y_1-y_2\|_T,
    \qquad
    \omega_T(\Gamma(y_1),\delta)\le2\omega_T(y_1,\delta).
\]
If $d$ is nonincreasing with $d(0)\le0$, then $\Gamma(y+d)\le\Gamma(y)$, while the reverse inequality holds if $d$ is nondecreasing with $d(0)\ge0$.
For every $y\in\D([0,T],\R)$ with $y(0)\ge0$, there is a unique pair $(x,\ell)$ satisfying
\begin{equation}\label{eq:reflected_polynomial_map}
    x(t)=y(t)-\beta\int_0^t x(s)^kds+\ell(t),
    \qquad t\in[0,T],
\end{equation}
where $x\ge0$, $\ell$ is nondecreasing with $\ell(0)=0$, and $\int_0^T\mathds{1}\{x(t)>0\}d\ell(t)=0$.
If $y_n\to y$ uniformly, then the corresponding solutions converge uniformly, and the solution is continuous whenever $y$ is continuous.
Consequently, the solution map in \eqref{eq:reflected_polynomial_map} is $J_1$-continuous at every continuous input.
\end{lemma}

\begin{proof}
The explicit formula
\[
    \Gamma(z)(t)=\max\left\{z(t),\sup_{0\le s\le t}\bigl(z(t)-z(s)\bigr)\right\}
\]
gives the two Lipschitz bounds and the stated increment comparison directly.
Set $g(x)=-\beta x^k$ and define $x^{(0)}=\Gamma(y)$ and
\[
    x^{(n+1)}=\Gamma\left(y+\int_0^\cdot g(x^{(n)}(s))ds\right).
\]
The increment comparison gives $0\le x^{(n)}\le\Gamma(y)$ for every $n$.
If $R=\|\Gamma(y)\|_T$ and $L_R$ is a Lipschitz constant of $g$ on $[0,R]$, then
\[
    \|x^{(n+1)}-x^{(n)}\|_t
    \le2L_R\int_0^t\|x^{(n)}-x^{(n-1)}\|_sds,
    \qquad t\le T.
\]
Iteration gives
\[
    \|x^{(n+1)}-x^{(n)}\|_T
    \le \|x^{(1)}-x^{(0)}\|_T\frac{(2L_RT)^n}{n!},
\]
so $x^{(n)}$ converges uniformly to a solution of \eqref{eq:reflected_polynomial_map}.
The same inequality and Gronwall's lemma give uniqueness.
If $y_n\to y$ uniformly, then $\Gamma(y_n)$ is uniformly bounded and the same Gronwall estimate gives $x_n\to x$ uniformly.
The identity $\ell_n=x_n-y_n+\beta\int_0^\cdot x_n(s)^kds$ then gives $\ell_n\to\ell$ uniformly.
Continuity follows because reflection preserves continuity.
Finally, $J_1$ convergence to a continuous path is uniform on $[0,T]$, which proves the last assertion.
\end{proof}

\paragraph{Thinning decomposition.}
Set $q_i^\alpha\triangleq F_\alpha(W_i^\alpha)$ and
\[
    \mathcal H_i^\alpha
    \triangleq\sigma(Z^\alpha(0))\vee\sigma\{T_j^\alpha:j\ge1\}\vee\sigma\{(D_j^\alpha,V_j^\alpha):1\le j<i\}.
\]
Then $W_i^\alpha$ is $\mathcal H_i^\alpha$-measurable and $q_i^\alpha=\E[\mathds{1}\{D_i^\alpha\le W_i^\alpha\}\mid\mathcal H_i^\alpha]$.
Define
\[
    C_{\mathrm{ab}}^\alpha(t)\triangleq\sum_{i=1}^{A^\alpha(t)}q_i^\alpha,
    \qquad
    M_{\mathrm{ab}}^\alpha(t)\triangleq A^\alpha(t)-A_0^\alpha(t)-C_{\mathrm{ab}}^\alpha(t),
\]
and set $\tilde C_{\mathrm{ab}}^\alpha(t)=\alpha^hC_{\mathrm{ab}}^\alpha(\alpha^{-2h}t)$ and $\tilde M_{\mathrm{ab}}^\alpha(t)=\alpha^hM_{\mathrm{ab}}^\alpha(\alpha^{-2h}t)$.
Also define
\[
    \tilde S_{\mathrm{fl}}^\alpha(t)
    \triangleq
    \alpha^h\sum_{i=1}^{A^\alpha(\alpha^{-2h}t)}
    \left(V_i^\alpha-\frac1{\mu_\alpha}\right)
    \mathds{1}\{D_i^\alpha>W_i^\alpha\}
\]
and
\[
    \Delta_S^\alpha(t)
    \triangleq
    \alpha^h\sum_{i=1}^{A^\alpha(\alpha^{-2h}t)}
    \left(V_i^\alpha-\frac1{\mu_\alpha}\right)
    \mathds{1}\{D_i^\alpha\le W_i^\alpha\}.
\]
The exact decompositions are
\begin{equation}\label{eq:HT_thinning_decompositions}
    \tilde A^\alpha-\tilde A_0^\alpha
    =\tilde C_{\mathrm{ab}}^\alpha+\tilde M_{\mathrm{ab}}^\alpha,
    \qquad
    \tilde S_{\mathrm{fl}}^\alpha
    =\tilde S^\alpha\circ\bar A^\alpha-\Delta_S^\alpha.
\end{equation}
By \citet[Th\'eor\`eme~I, p.~172]{lenglart1977relation}, applied to the nonnegative submartingale $M^2$ dominated by its predictable bracket $\langle M\rangle$, we have
\begin{equation}\label{eq:Lenglart_used}
    \mathbb P(\|M\|_T>\varepsilon)
    \le\frac{\eta}{\varepsilon^2}+\mathbb P(\langle M\rangle(T)>\eta),
    \qquad \varepsilon,\eta>0.
\end{equation}
The ucp implication used below is \citet[Corollaire~I, p.~173]{lenglart1977relation}.

\begin{lemma}[Tightness and thinning reduction]\label{lem:HT_input_reduction}
For every $T>0$,
\[
    \|\tilde M_{\mathrm{ab}}^\alpha\|_T+\|\Delta_S^\alpha\|_T\Rightarrow0,
\]
and
\[
    (\tilde A^\alpha,\tilde S_{\mathrm{fl}}^\alpha,\tilde Z^\alpha(0))
    \Rightarrow
    \left(c_aB_a\circ(\mu e),\mu^{-1}c_sB_s\circ(\mu e),Z^*(0)\right).
\]
Moreover, $\{\tilde Z^\alpha\}$ is $C$-tight, meaning that it is tight and every subsequential limit is continuous.
\end{lemma}

\begin{proof}
The identity $\bar A^\alpha=\lambda_\alpha e+\alpha^h\tilde A^\alpha$ and \eqref{eq:FCLT_primitives} imply $\|\bar A^\alpha-\mu e\|_T\Rightarrow0$.
The time changes are nondecreasing and converge uniformly to the continuous map $\mu e$, while the outer Brownian limit in \eqref{eq:FCLT_primitives} is continuous.
The composition theorem \citep[Theorem~13.2.1, p.~518]{whitt2002stochastic} therefore gives the joint convergence
\begin{equation}\label{eq:HT_unthinned_service_limit}
    (\tilde A^\alpha,\tilde S^\alpha\circ\bar A^\alpha,\tilde Z^\alpha(0))
    \Rightarrow
    \left(c_aB_a\circ(\mu e),\mu^{-1}c_sB_s\circ(\mu e),Z^*(0)\right).
\end{equation}
Let $\tilde Z_{\mathrm{tot}}^\alpha$ be the scaled workload in the coupled queue that accepts every arrival.
With $c_\alpha=\alpha^{-h}(\rho_\alpha-1)$,
\[
    \tilde Z_{\mathrm{tot}}^\alpha
    =\Gamma\left(\tilde Z^\alpha(0)+\frac1{\mu_\alpha}\tilde A^\alpha+\tilde S^\alpha\circ\bar A^\alpha+c_\alpha e\right).
\]
Equation \eqref{eq:HT_unthinned_service_limit} and Lemma~\ref{lem:Gamma_increment_monotone} imply that $\tilde Z_{\mathrm{tot}}^\alpha$ is $C$-tight.
The rejected workload is nondecreasing, so the increment comparison in Lemma~\ref{lem:Gamma_increment_monotone} gives $0\le\tilde Z^\alpha\le\tilde Z_{\mathrm{tot}}^\alpha$ pathwise.
Consequently, $\|\tilde Z^\alpha\|_T=O_p(1)$.

Let $\tau_R^\alpha\triangleq\inf\{t\in[0,T]:\tilde Z^\alpha(t)>R\}\wedge T$.
Since $k(1-h)=h$, Assumption~\ref{assumption:F} gives, for each fixed $R$, a constant $C_R$ such that
\begin{equation}\label{eq:HT_local_q_bound}
    q_i^\alpha\le C_R\alpha^h
\end{equation}
for every arrival up to scaled time $\tau_R^\alpha$ and all sufficiently small $\alpha$.
Indeed, before $\tau_R^\alpha$, $\alpha W_i^\alpha\le\alpha^{1-h}R$ and $F(x)\le C_Rx^k$ near zero.
The customer-indexed sums stopped after $A^\alpha(\alpha^{-2h}\tau_R^\alpha)$ terms are square-integrable martingales whose predictable brackets satisfy
\[
    \langle\tilde M_{\mathrm{ab}}^\alpha\rangle(\tau_R^\alpha)
    \le C_R\alpha^h\bar A^\alpha(T),
    \qquad
    \langle\Delta_S^\alpha\rangle(\tau_R^\alpha)
    \le C_R\alpha^h\bar A^\alpha(T).
\]
The second estimate uses $\sup_\alpha\Var(V_i^\alpha)<\infty$ and the independence of $V_i^\alpha$, $D_i^\alpha$, and $\mathcal H_i^\alpha$.
Equation \eqref{eq:Lenglart_used} makes both stopped martingales vanish uniformly in probability.
Since $\|\tilde Z^\alpha\|_T=O_p(1)$, letting $R\to\infty$ gives
\begin{equation}\label{eq:HT_martingales_negligible}
    \|\tilde M_{\mathrm{ab}}^\alpha\|_T+\|\Delta_S^\alpha\|_T\Rightarrow0.
\end{equation}
Combining \eqref{eq:HT_thinning_decompositions}, \eqref{eq:HT_unthinned_service_limit}, and \eqref{eq:HT_martingales_negligible} proves the stated joint convergence.

For $0\le s\le t\le T$, \eqref{eq:HT_local_q_bound} gives
\[
    0\le\tilde C_{\mathrm{ab}}^\alpha(t\wedge\tau_R^\alpha)-\tilde C_{\mathrm{ab}}^\alpha(s\wedge\tau_R^\alpha)
    \le C_R\bigl(\bar A^\alpha(t\wedge\tau_R^\alpha)-\bar A^\alpha(s\wedge\tau_R^\alpha)\bigr).
\]
The modulus bounds for $\bar A^\alpha$ show that the stopped compensators are $C$-tight.
Compact containment and $R\to\infty$ show that $\tilde C_{\mathrm{ab}}^\alpha$ is $C$-tight.
Equations \eqref{eq:HT_thinning_decompositions} and \eqref{eq:HT_martingales_negligible} then show that $\tilde A_0^\alpha$ is $C$-tight.
The exact identities
\begin{equation}\label{eq:HT_exact_scaled_balances}
    \tilde Y^\alpha=\frac1{\mu_\alpha}\tilde A_0^\alpha+\tilde S_{\mathrm{fl}}^\alpha,
    \qquad
    \tilde Z^\alpha=\Gamma\left(\tilde Z^\alpha(0)+\tilde Y^\alpha+c_\alpha e\right)
\end{equation}
complete the $C$-tightness proof by the modulus bound in Lemma~\ref{lem:Gamma_increment_monotone}.
\end{proof}

\begin{lemma}[Abandonment compensator]\label{lem:HT_abandonment_compact}
For every $T>0$,
\begin{equation}\label{eq:HT_abandonment_coupling}
    \left\|
    \tilde A^\alpha-\tilde A_0^\alpha
    -\lambda_\alpha\beta\int_0^\cdot(\tilde Z^\alpha(s))^kds
    \right\|_T
    \Rightarrow0.
\end{equation}
\end{lemma}

\begin{proof}
Define $\widehat F_\alpha(x)\triangleq\alpha^{-h}F(\alpha^{1-h}x)$ for $x\ge0$.
Since $W_i^\alpha=Z^\alpha(T_i^\alpha-)$,
\[
    \tilde C_{\mathrm{ab}}^\alpha(t)
    =\int_0^t\widehat F_\alpha(\tilde Z^\alpha(s-))d\bar A^\alpha(s).
\]
Lemma~\ref{lem:uniform_F_expand} and $k(1-h)=h$ give
\begin{equation}\label{eq:HT_scaled_F_limit}
    \sup_{0\le x\le R}|\widehat F_\alpha(x)-\beta x^k|\to0
\end{equation}
for every $R<\infty$.

Take an arbitrary subsequence.
Lemma~\ref{lem:HT_input_reduction} and $\|\bar A^\alpha-\lambda_\alpha e\|_T\Rightarrow0$ yield a further subsequence on which
\[
    (\tilde Z^\alpha,\bar A^\alpha-\lambda_\alpha e)\Rightarrow(Z,0)
\]
for a continuous process $Z$.
The product $J_1$ space is Polish, so the Skorokhod representation theorem \citep[Theorem~6.7, p.~70]{billingsley1999convergence} permits a common realization with almost-sure $J_1$ convergence.
Since the limit is continuous, $\tilde Z^\alpha\to Z$ uniformly on $[0,T]$ and $\sup_{s\le T}|\tilde Z^\alpha(s-)-Z(s)|\to0$.
Equation \eqref{eq:HT_scaled_F_limit} therefore gives
\[
    \widehat F_\alpha(\tilde Z^\alpha(\cdot-))\to\beta Z(\cdot)^k
\]
uniformly.
We use the elementary deterministic fact that if nondecreasing $a_n$ satisfy $a_n(0)=0$, $\lambda_n\to\lambda<\infty$, $\|a_n-\lambda_ne\|_T\to0$, and $g_n\to g$ uniformly for continuous $g$, then
\[
    \left\|\int_0^\cdot g_n(s)da_n(s)-\lambda_n\int_0^\cdot g_n(s)ds\right\|_T\to0.
\]
For a step function $g=\sum_{j=1}^m b_j\mathds{1}_{(t_{j-1},t_j]}$, the displayed norm is at most $2\|a_n-\lambda_ne\|_T\sum_{j=1}^m|b_j|$, and uniform step approximation together with bounded total masses proves the claim.
Applying the claim with $a_n=\bar A^\alpha$ and $g_n=\widehat F_\alpha(\tilde Z^\alpha(\cdot-))$ first replaces $d\bar A^\alpha$ by $\lambda_\alpha ds$.
Both $g_n$ and $\beta(\tilde Z^\alpha)^k$ converge uniformly to $\beta Z^k$, so
\[
    \left\|
    \tilde C_{\mathrm{ab}}^\alpha
    -\lambda_\alpha\beta\int_0^\cdot(\tilde Z^\alpha(s))^kds
    \right\|_T\to0
\]
almost surely along the coupled subsequence.
Equation \eqref{eq:HT_martingales_negligible} and the first identity in \eqref{eq:HT_thinning_decompositions} prove \eqref{eq:HT_abandonment_coupling} by the subsequence criterion.
\end{proof}

\begin{proof}[Proof of Theorem~\ref{thm:HT_limit}]
Let
\[
    \tilde R_{\mathrm{ab}}^\alpha(t)
    \triangleq
    \frac1{\mu_\alpha}\bigl(\tilde A^\alpha(t)-\tilde A_0^\alpha(t)\bigr)
    -\beta\int_0^t(\tilde Z^\alpha(s))^kds.
\]
Lemma~\ref{lem:HT_input_reduction} gives $\int_0^T(\tilde Z^\alpha(s))^kds=O_p(1)$.
Equation~\eqref{eq:HT_abandonment_coupling} and $\lambda_\alpha/\mu_\alpha=\rho_\alpha\to1$ therefore imply $\|\tilde R_{\mathrm{ab}}^\alpha\|_T\Rightarrow0$.
Using \eqref{eq:HT_exact_scaled_balances}, the workload equation becomes
\[
    \tilde Z^\alpha(t)
    =\tilde y^\alpha(t)
    -\beta\int_0^t(\tilde Z^\alpha(s))^kds
    +\tilde L^\alpha(t),
\]
where
\[
    \tilde y^\alpha(t)
    \triangleq
    \tilde Z^\alpha(0)
    +\frac1{\mu_\alpha}\tilde A^\alpha(t)
    +\tilde S_{\mathrm{fl}}^\alpha(t)
    +c_\alpha t
    -\tilde R_{\mathrm{ab}}^\alpha(t).
\]
The primitive FCLT and Lemma~\ref{lem:HT_input_reduction} give
\[
    \tilde y^\alpha\Rightarrow y,
    \qquad
    y(t)=Z^*(0)+\mu^{-1}c_aB_a(\mu t)+\mu^{-1}c_sB_s(\mu t)+ct,
\]
jointly with the primitive limits.
Lemma~\ref{lem:Gamma_increment_monotone} and the continuous mapping theorem now give
\[
    (\tilde Z^\alpha,\tilde L^\alpha)\Rightarrow(Z^*,L^*),
\]
where $(Z^*,L^*)$ is the unique reflected solution of \eqref{eq:HT_limit_Zstar} on $[0,T]$.

Equation \eqref{eq:HT_abandonment_coupling} and the continuity of $x\mapsto\int_0^\cdot x(s)^kds$ at continuous paths give
\[
    \tilde A_0^\alpha
    \Rightarrow
    A_0^*,
    \qquad
    A_0^*(t)=c_aB_a(\mu t)-\mu\beta\int_0^t(Z^*(s))^kds.
\]
The first identity in \eqref{eq:HT_exact_scaled_balances} then gives
\[
    \tilde Y^\alpha
    \Rightarrow
    Y^*,
    \qquad
    Y^*(t)=\frac1\mu A_0^*(t)+\frac1\mu c_sB_s(\mu t).
\]
All convergences are joint because the processes are continuous mappings of the jointly convergent primitives and errors that vanish uniformly in probability.
Since $T>0$ is arbitrary, the asserted convergence holds on $\D(\R_+,\R^4)$.
\end{proof}

\subsection{Proof of Lemma~\ref{lm:var_expression}}

\begin{proof}
At $t=0$, both sides are zero, so fix $t>0$.
The proof rescales the stationary diffusion in Theorem~\ref{thm:HT_limit} to the base diffusion \eqref{eq:base_reflected_SDE} and then applies the same scaling to its effective input.
Set
\[
    \theta\triangleq\left(\frac{c_x^2}{2\mu\beta}\right)^{\frac{1}{k+1}}.
\]
The definitions in Lemma~\ref{lm:var_expression} are equivalently
\[
    \tau=\beta\theta^{k-1},
    \qquad
    \tilde c=\frac{c}{\beta\theta^k},
    \qquad
    2\theta^2\tau=\frac{c_x^2}{\mu}.
\]
The Brownian motions in \eqref{eq:HT_limit_Zstar} can be represented as $\sqrt{c_x^2/\mu}\,B(t)$ for a standard Brownian motion $B$.
For $u\ge0$, define
\[
    \widehat Z(u)\triangleq\theta^{-1}Z^*(u/\tau),
    \qquad
    \widehat L(u)\triangleq\theta^{-1}L^*(u/\tau).
\]
By Brownian scaling, we may reuse $B$ for the standard Brownian motion $\sqrt{\tau}B(\cdot/\tau)$, and
\[
    \frac{\sqrt{c_x^2/\mu}}{\theta\sqrt{\tau}}=\sqrt2,
    \qquad
    \frac{\beta\theta^{k-1}}{\tau}=1,
    \qquad
    \frac{c}{\theta\tau}=\tilde c.
\]
Substitution in \eqref{eq:HT_limit_Zstar} therefore gives
\[
    \widehat Z(u)
    =\widehat Z(0)+\sqrt2 B(u)
    +\int_0^u\left(\tilde c-\widehat Z(r)^k\right)dr
    +\widehat L(u).
\]
Positive space and time scaling preserve the reflection conditions.
Changing variables $x=\theta z$ in the stationary density \eqref{eq:stationary_dist} and using $2\mu\beta\theta^{k+1}=c_x^2$ gives $\widehat Z(0)\sim\pi_{\tilde c,k}$.
The uniqueness assertion in Theorem~\ref{thm:HT_limit}, specialized to the base coefficients and the same independent stationary initial law, now identifies $(\widehat Z,\widehat L)$ with $(Z^{\tilde c,k},L^{\tilde c,k})$ in law.
Since $c=\theta\tilde c\tau$, the definitions of $Y^*$ and \eqref{eq:base_Y_def} give
\[
    Y^*(t)
    =\theta\left(\widehat Z(\tau t)-\widehat Z(0)-\tilde c\tau t-\widehat L(\tau t)\right)
    \stackrel{d}{=}\theta Y^{\tilde c,k}(\tau t).
\]
Hence \eqref{eq:w} and $2\theta^2\tau=c_x^2/\mu$ imply
\[
    v(t;\Xi)
    =\theta^2v_{\tilde c,k}(\tau t)
    =2\theta^2\tau t\,w_{\tilde c,k}(\tau t)
    =\frac{c_x^2}{\mu}t\,w_{\tilde c,k}(\tau t).
\]
\end{proof}

\subsection{Proofs for the Response Functions and \texorpdfstring{$w_{c,k}$}{wck}}\label{app:wck_proofs}

We first establish the backward characterization and sensitivity identity, then derive the variance decomposition, and finally prove the qualitative properties of $w_{c,k}$.

\begin{proof}[Proof of Proposition~\ref{prop:psi_h_pde}]
Write $\mathcal Lf=f''+(c-x^k)f'$.
For $x>0$ and $R>x$, let $T_R\triangleq\inf\{s\ge0:Z^{c,k}(s)\ge R\}$.
The strong Markov property applied to \eqref{eq:psi_def} gives, with $\sigma_R\triangleq t\wedge T_0\wedge T_R$,
\[
    \psi_{c,k}(t,x)
    =\E_x\left[
        \exp\left\{-\int_0^{\sigma_R}q_k(Z^{c,k}(r))dr\right\}
        \psi_{c,k}(t-\sigma_R,Z^{c,k}(\sigma_R))
    \right].
\]
The same property applied to \eqref{eq:varphi_def} gives, with $\eta_R\triangleq t\wedge T_R$,
\[
    \varphi_{c,k}(t,x)
    =\E_x\left[
        \int_0^{\eta_R}(Z^{c,k}(r))^kdr
        +\varphi_{c,k}(t-\eta_R,Z^{c,k}(\eta_R))
    \right].
\]
Let $\overline Z^x$ be reflected Brownian motion with constant drift $c$, driven by the same Brownian motion as $Z^{c,k}$.
Since $c-z^k\le c$, the one-dimensional reflected comparison principle gives
\[
    0\le Z^{c,k}(s)\le \overline Z^x(s),
    \qquad s\ge0.
\]
Consequently, for every $p\ge1$ and $t<\infty$,
\[
    \E_x\left[\sup_{0\le s\le t}Z^{c,k}(s)^p\right]<\infty.
\]
In particular,
\[
    \mathbb P_x(T_R\le t)
    \le
    R^{-p}
    \E_x\left[\sup_{0\le s\le t}Z^{c,k}(s)^p\right]
    \longrightarrow0.
\]
The polynomial-growth condition on $\varphi$ also gives
\[
    \left|
    \varphi(t-\eta_R,Z^{c,k}(\eta_R))
    \right|
    \le
    C_t\left(
    1+\sup_{0\le s\le t}Z^{c,k}(s)^k
    \right),
\]
whose right-hand side is integrable.
Therefore dominated convergence permits $R\to\infty$ in both stopped identities.
The functions defined in \eqref{eq:psi_def} and \eqref{eq:varphi_def} satisfy the stopped identities by the strong Markov property.
Conversely, after letting $R\to\infty$, any bounded function satisfying the first identity equals the expectation in \eqref{eq:psi_def}, and any polynomial-growth function satisfying the second identity equals the expectation in \eqref{eq:varphi_def}.
Hence the stopped representations uniquely characterize $\psi_{c,k}$ and $\varphi_{c,k}$ in the stated classes.

The stopped identities also imply the interior backward equations in the viscosity sense.
Indeed, one stops the diffusion upon leaving a small parabolic neighborhood of an interior contact point, applies the stopped identity and It\^o's formula to a $C^{1,2}$ test function touching from above or below, divides by the expected stopping time, and shrinks the neighborhood.
The initial and Dirichlet boundary conditions follow directly from the definitions.
The Neumann boundary condition for $\varphi_{c,k}$ follows from the derivative identity proved below.
Finally, synchronous coupling of diffusions started from nearby initial states, followed by the preceding moment bound and dominated convergence, proves continuity.

It remains to prove the derivative identity.
Choose a bounded nondecreasing $C^\infty$ function $g_R:\R\to\R$ that equals $x^k$ on $[0,R]$, is constant on $[R+1,\infty)$, and satisfies $0\le g_R(x)\le x^k$ for $x\ge0$.
Set $b_R=c-g_R$, let $Z_R^x$ be the one-sided reflected diffusion with drift $b_R$, and define
\[
    \varphi_R(t,x)\triangleq\E_x\left[\int_0^t g_R(Z_R^x(s))ds\right].
\]
The coefficients $b_R$ and $g_R$ are bounded and Lipschitz.
By \citet[Proposition~2.8]{bossy2011stochastic}, the weak derivative of the reflected flow is
\[
    \partial_x Z_R^x(s)
    =\exp\left\{-\int_0^s g_R'(Z_R^x(r))dr\right\}\mathds{1}\{s<T_0^R\},
\]
where $T_0^R$ is the first hitting time of zero by $Z_R^x$.
For a compactly supported smooth test function, Fubini's theorem and the Sobolev chain rule therefore identify the weak derivative
\[
    \partial_x\varphi_R(t,x)
    =\E_x\left[\int_0^t
        g_R'(Z_R^x(s))
        \exp\left\{-\int_0^s g_R'(Z_R^x(r))dr\right\}
        \mathds{1}\{s<T_0^R\}ds
    \right].
\]
The integral inside the expectation equals
\[
    1-\exp\left\{-\int_0^{t\wedge T_0^R}g_R'(Z_R^x(r))dr\right\},
\]
so $0\le\partial_x\varphi_R\le1$.
Before $T_R$, the localized and original diffusions coincide and $g_R'=q_k$.
The moment bound, $\mathbb P_x(T_R\le t)\to0$, and the preceding uniform derivative bound imply, locally uniformly in $x$,
\[
    \varphi_R(t,x)\longrightarrow\varphi_{c,k}(t,x),
    \qquad
    \partial_x\varphi_R(t,x)\longrightarrow1-\psi_{c,k}(t,x).
\]
Passing to distributional derivatives gives $\partial_x\varphi_{c,k}=1-\psi_{c,k}$.
Since $\psi_{c,k}$ is continuous, the derivative has a continuous version on $[0,\infty)$, and its boundary value is zero because $\psi_{c,k}(t,0)=1$.
\end{proof}

\begin{proof}[Proof of Lemma~\ref{lm:w_variance_rep}]
Fix $c\in\R$ and an integer $k\ge1$.
Use the bounded localizations $g_R$, $b_R$, $Z_R$, and $\varphi_R$ from the preceding proof, with $Z_R(0)=Z^{c,k}(0)$ and the same Brownian motion.
For $n\ge1$, let
\[
    Z_{R,n}(s)=Z_R(0)+\int_0^s b_R(Z_{R,n}(r))dr+\sqrt2B(s)
    +n\int_0^s(Z_{R,n}(r))^-dr,
\]
and define $\varphi_{R,n}(t,x)\triangleq\E_x[\int_0^t g_R(Z_{R,n}(s))ds]$.
After replacing $x^-$ by smooth monotone approximations, the classical Feynman--Kac formula and It\^o's formula give
\[
    \int_0^t g_R(Z_{R,n}(s))ds
    =\varphi_{R,n}(t,Z_R(0))
    +\sqrt2\int_0^t\partial_x\varphi_{R,n}(t-s,Z_{R,n}(s))dB(s).
\]
The derivative of the penalized flow lies in $[0,1]$ because both $b_R'$ and the derivative of the penalty are nonpositive.

After passing to the limit in the smooth monotone approximations of
$x^-$, write $Z_{R,n}^x$ and $Z_R^x$ for the penalized and reflected
flows started from $x$.  Their flow-derivative representations give,
for $0\le u\le t$,
\[
\partial_x\varphi_{R,n}(u,x)
=
\E_x\!\left[
    \int_0^u
    g_R'(Z_{R,n}^x(r))
    \partial_x Z_{R,n}^x(r)\,dr
\right],
\]
and
\[
\partial_x\varphi_R(u,x)
=
\E_x\!\left[
    \int_0^u
    g_R'(Z_R^x(r))
    \partial_x Z_R^x(r)\,dr
\right].
\]
The path convergence in
\citet[Proposition~2.5]{bossy2011stochastic}, the reflected-flow
derivative representation in
\citet[Proposition~2.8]{bossy2011stochastic}, and the strong
$L^2$ convergence of the penalized flow derivatives in
\citet[Lemma~3.3]{bossy2011stochastic}, together with the boundedness of
$g_R'$ and the uniform derivative bound, imply by localization in the
state variable and dominated convergence that
\begin{equation}
\label{eq:penalized-integrand-convergence}
\lim_{n\to\infty}
\E\!\left[
\int_0^t
\left|
\partial_x\varphi_{R,n}(t-s,Z_{R,n}(s))
-
\partial_x\varphi_R(t-s,Z_R(s))
\right|^2 ds
\right]
=0.
\end{equation}
Consequently, It\^o's isometry yields
\[
\begin{aligned}
&\E\!\left[
\left|
\sqrt{2}\int_0^t
\partial_x\varphi_{R,n}(t-s,Z_{R,n}(s))\,dB(s)
-
\sqrt{2}\int_0^t
\partial_x\varphi_R(t-s,Z_R(s))\,dB(s)
\right|^2
\right]
\\
&\qquad =
2\E\!\left[
\int_0^t
\left|
\partial_x\varphi_{R,n}(t-s,Z_{R,n}(s))
-
\partial_x\varphi_R(t-s,Z_R(s))
\right|^2 ds
\right]
\longrightarrow 0.
\end{aligned}
\]
The boundedness and Lipschitz continuity of $g_R$, together with the
path convergence of $Z_{R,n}$ to $Z_R$, similarly imply
\[
\int_0^t g_R(Z_{R,n}(s))\,ds
\longrightarrow
\int_0^t g_R(Z_R(s))\,ds
\]
and
\[
\varphi_{R,n}(t,Z_R(0))
\longrightarrow
\varphi_R(t,Z_R(0))
\]
in $L^2$.  Passing to the limit in the penalized identity therefore
gives
\[
\int_0^t g_R(Z_R(s))\,ds
=
\varphi_R(t,Z_R(0))
+
\sqrt{2}\int_0^t
\partial_x\varphi_R(t-s,Z_R(s))\,dB(s).
\]
The reflected identity has no local-time term because $\partial_x\varphi_R(t,0)=0$.
The comparison bound, $0\le g_R(x)\le x^k$, and localization imply that the first two terms converge in $L^2$ to their untruncated counterparts as $R\to\infty$.
The derivative formula in the preceding proof, its bound by one, and It\^o's isometry give convergence of the stochastic integrals in $L^2$.
Consequently,
\[
    \int_0^t (Z^{c,k}(s))^kds
    =\varphi_{c,k}(t,Z^{c,k}(0))
    +\sqrt2\int_0^t
    \partial_x\varphi_{c,k}(t-s,Z^{c,k}(s))dB(s).
\]
Substituting this identity into \eqref{eq:base_Y_def} and using \eqref{eq:varphi_x_psi} yields
\[
    Y^{c,k}(t)-\E[Y^{c,k}(t)]
    =-\left(\varphi_{c,k}(t,Z^{c,k}(0))
    -\E_{\pi_{c,k}}[\varphi_{c,k}(t,Z)]\right)+\sqrt2\int_0^t
    \psi_{c,k}(t-s,Z^{c,k}(s))dB(s).
\]
We take the stationary initial state to be independent of the future Brownian increments.
The first term is measurable at time zero, while the stochastic integral has conditional mean zero given the initial state.
Their covariance is therefore zero.
It\^o's isometry and stationarity give
\[
    v_{c,k}(t)
    =\Var_{\pi_{c,k}}\left(\varphi_{c,k}(t,Z)\right)
    +2\int_0^t\E_{\pi_{c,k}}\left[\psi_{c,k}(u,Z)^2\right]du.
\]
Dividing by $2t$ proves \eqref{eq:w_variance_rep}.
\end{proof}

\begin{proof}[Proof of Proposition~\ref{prop:w}]
Fix $c\in\R$ and an integer $k\ge1$.
For this proof only, write $\pi=\pi_{c,k}$, $\pi_0=\pi_{c,k}(0)$, $\overline\Pi(x)=\int_x^\infty\pi(y)dy$, $\varphi_t(x)=\varphi_{c,k}(t,x)$, and $\psi_t(x)=\psi_{c,k}(t,x)$.
Here $\pi_0$ is the value of the stationary density at zero, not an atom.
The representation \eqref{eq:psi_def} gives $0<\psi_t(x)\le1$ and shows that $t\mapsto\psi_t(x)$ is nonincreasing.
For every $x>0$, the decrease is strict because the diffusion has positive probability of remaining in a compact subinterval of $(0,\infty)$ during any prescribed positive interval.

Let $P_t$ be the stationary Markov semigroup of $Z^{c,k}$ and set $g(x)=x^k$.
The Markov property gives $\varphi_t=\int_0^tP_sg\,ds$ in $L^2(\pi)$.
Jensen's inequality makes $P_t$ a contraction on $L^2(\pi)$, while path continuity gives strong continuity first on bounded continuous functions and then on all of $L^2(\pi)$ by density.
Since $g\in L^2(\pi)$, the map $t\mapsto\varphi_t$ is continuously differentiable in $L^2(\pi)$ with derivative $P_tg$.
Moreover,
\[
    \frac{P_h\varphi_t-\varphi_t}{h}
    =\frac1h\int_t^{t+h}P_sg\,ds-\frac1h\int_0^hP_sg\,ds
    \longrightarrow P_tg-g
\]
in $L^2(\pi)$, so $\varphi_t\in\operatorname{Dom}(\mathcal L)$ and $\mathcal L\varphi_t=P_tg-g$.
The $L^2$ differentiability of $\varphi_t$ and dominated convergence in \eqref{eq:psi_def} justify differentiating the variance representation, which gives
\[
    \frac12v_{c,k}'(t)
    =\Cov_\pi(\varphi_t,P_tg)+\E_\pi[\psi_t(Z)^2].
\]
Define $m_{c,k}(t)\triangleq v_{c,k}'(t)/2$ for $t>0$.
Since $\mathcal Lf=\pi^{-1}(\pi f')'$, the generator identity implies $(\pi\partial_x\varphi_t)'=\pi\mathcal L\varphi_t$ in the distributional sense.
The right-hand side is locally integrable, so $\pi\partial_x\varphi_t$ is locally absolutely continuous and integration by parts gives
\[
    \Cov_\pi(\varphi_t,\mathcal L\varphi_t)
    =-\E_\pi[(\partial_x\varphi_t(Z))^2].
\]
The boundary term is zero at the origin because $\partial_x\varphi_t(0)=0$, and it is zero at infinity because $\partial_x\varphi_t$ is bounded, $\varphi_t$ grows at most linearly by \eqref{eq:varphi_x_psi}, and $\pi$ has a super-exponential tail.
Consequently,
\[
    m_{c,k}(t)
    =-\E_\pi[(\partial_x\varphi_t(Z))^2]
    +\Cov_\pi(\varphi_t(Z),Z^k)
    +\E_\pi[\psi_t(Z)^2].
\]
Since $x^k\pi(x)=c\pi(x)-\pi'(x)$, integration by parts gives
\[
    \E_\pi[\varphi_t(Z)Z^k]
    =c\E_\pi[\varphi_t(Z)]
    +\E_\pi[\partial_x\varphi_t(Z)]
    +\pi_0\varphi_t(0).
\]
The same calculation with $\varphi_t\equiv1$ gives $\E_\pi[Z^k]=c+\pi_0$.
Therefore
\[
    \Cov_\pi(\varphi_t(Z),Z^k)
    =\E_\pi[\partial_x\varphi_t(Z)]
    +\pi_0\left(\varphi_t(0)-\E_\pi[\varphi_t(Z)]\right).
\]
Moreover,
\[
    \E_\pi[\varphi_t(Z)]-\varphi_t(0)
    =\int_0^\infty\partial_x\varphi_t(x)\overline\Pi(x)dx
    =\int_0^\infty(1-\psi_t(x))\overline\Pi(x)dx.
\]
Substitution and \eqref{eq:varphi_x_psi} yield
\begin{equation}\label{eq:w_derivative_compact}
    m_{c,k}(t)
    =\E_\pi[\psi_t(Z)]
    -\pi_0\int_0^\infty(1-\psi_t(x))\overline\Pi(x)dx.
\end{equation}
The first term is strictly decreasing, while the nonnegative quantity subtracted from it is nondecreasing.
Thus $m_{c,k}$ is strictly decreasing.

We next identify the limit of $\psi_t$ without invoking a separate parabolic long-time theorem.
The function $y\mapsto\int_0^y\exp\{-cu+u^{k+1}/(k+1)\}du$ is $\mathcal L$-harmonic and vanishes at zero.
The diffusion coefficient is nonzero and the drift is bounded on $[0,R]$, so $T_0\wedge T_R<\infty$ almost surely.
Indeed, over a sufficiently short fixed interval every starting point in $[0,R]$ has a probability bounded away from zero of exiting, and the strong Markov property then gives a geometric tail for the exit time.
Optional stopping at $T_0\wedge T_R$ therefore gives
\[
    \mathbb P_x(T_R<T_0)
    =\frac{\int_0^x\exp\{-cy+y^{k+1}/(k+1)\}dy}
    {\int_0^R\exp\{-cy+y^{k+1}/(k+1)\}dy}.
\]
The denominator diverges as $R\to\infty$.
On $\{T_0=\infty\}$ the process must exit every bounded interval through its upper endpoint, so the preceding probability implies $T_0<\infty$ almost surely.
Define
\begin{equation}\label{eq:psi_infty_explicit}
    \psi_\infty(x)\triangleq\frac{\pi_0\overline\Pi(x)}{\pi(x)}.
\end{equation}
L'H\^opital's rule gives
\[
    \frac{\overline\Pi(x)}{\pi(x)/(x^k-c)}\longrightarrow1,
    \qquad x\to\infty,
\]
because the derivatives of the numerator and denominator are $-\pi(x)$ and $-\pi(x)(1+o(1))$, respectively.
Thus $\psi_\infty$ is bounded and tends to zero at infinity.
Direct differentiation gives
\[
    \mathcal L\psi_\infty-q_k\psi_\infty=0,
    \qquad
    \psi_\infty(0)=1.
\]
The stopped Feynman--Kac identity on $[0,R]$ gives
\[
    \psi_\infty(x)
    =\E_x\left[
        e^{-\int_0^{T_0}q_k(Z^{c,k}(s))ds};\,T_0<T_R
    \right] +\psi_\infty(R)\E_x\left[
        e^{-\int_0^{T_R}q_k(Z^{c,k}(s))ds};\,T_R<T_0
    \right].
\]
The second term vanishes as $R\to\infty$ because $\psi_\infty(R)\to0$.
Since $T_0<\infty$ almost surely, monotone convergence gives
\[
    \psi_\infty(x)
    =\E_x\left[
        \exp\left\{-\int_0^{T_0}q_k(Z^{c,k}(s))ds\right\}
    \right].
\]
Monotone convergence in \eqref{eq:psi_def} now gives $\psi_t(x)\downarrow\psi_\infty(x)$.
Using \eqref{eq:psi_infty_explicit}, equation \eqref{eq:w_derivative_compact} becomes
\[
    m_{c,k}(t)
    =\E_\pi[\psi_t(Z)-\psi_\infty(Z)+\psi_t(Z)\psi_\infty(Z)].
\]
Hence $m_{c,k}(t)>0$, and dominated convergence gives
\[
    m_{c,k}(t)\downarrow
    \E_\pi[\psi_\infty(Z)^2]
    =\pi_0^2\int_0^\infty\frac{\overline\Pi(x)^2}{\pi(x)}dx.
\]
Its integrand is asymptotic to $\pi(x)/(x^k-c)^2$, so the final integral is finite.
The same formula, $\psi_t\to1$ as $t\downarrow0$, and $\int_0^\infty\overline\Pi(x)dx<\infty$ give $m_{c,k}(t)\to1$.
Set $m_{c,k}(0)\triangleq1$.
Since $v_{c,k}(0)=0$,
\begin{equation}\label{eq:w_average_m}
    w_{c,k}(t)=\frac1t\int_0^t m_{c,k}(u)du.
\end{equation}
The positivity, monotonicity, and endpoint values of $m_{c,k}$ imply $0<w_{c,k}(t)\le1$ and $w_{c,k}(t)\to1$ as $t\downarrow0$.
Because the average of a strictly decreasing function strictly exceeds its value at the right endpoint, \eqref{eq:w_average_m} gives $w_{c,k}'(t)<0$ for $t>0$.
The Ces\`aro limit in \eqref{eq:w_average_m} proves \eqref{eq:w_infty_formula}.

We next prove joint continuity.
Let $c_n\to c$ and $t_n\to t>0$.
For $c_n$ in a fixed compact interval, the densities in \eqref{eq:pi_ck} and all their polynomially weighted versions are dominated by one integrable super-exponential envelope.
Dominated convergence therefore gives total-variation convergence of the stationary laws and convergence of every polynomial moment.
Under the common-uniform quantile coupling, the quantiles converge almost everywhere, and the uniform higher-moment bound gives $Z_n(0)\to Z(0)$ in $L^p$ for every finite $p$.
Drive the reflected diffusions by the same Brownian motion.
Tanaka's formula gives
\[
    (Z_n(s)-Z(s))^+
    \le (Z_n(0)-Z(0))^++s(c_n-c)^+,
\]
because on $\{Z_n>Z\}$ the monotone drift term is nonpositive, the regulator of $Z_n$ is inactive, and the regulator of $Z$ enters with a nonpositive sign.
Interchanging $n$ and the limit process yields, for every $T<\infty$,
\[
    \sup_{0\le s\le T}|Z_n(s)-Z(s)|
    \le |Z_n(0)-Z(0)|+T|c_n-c|.
\]
The inequality $|x^k-y^k|\le k|x-y|(x^{k-1}+y^{k-1})$, the preceding bound, and the stationary moment bounds control the common interval $[0,t_n\wedge t]$.
Stationarity and $|t_n-t|\to0$ control the remaining interval, so
\[
    \int_0^{t_n}Z_n(s)^kds
    \longrightarrow
    \int_0^tZ(s)^kds
    \qquad\text{in }L^2.
\]
Together with $B(t_n)\to B(t)$ in $L^2$, equation \eqref{eq:base_Y_def} gives $Y^{c_n,k}(t_n)\to Y^{c,k}(t)$ in $L^2$.
Thus $w_{c_n,k}(t_n)\to w_{c,k}(t)$ for $t>0$.
Let $C\subset\R$ be compact.
The variance representation, $w_{c,k}(t)\le1$, and $1-e^{-x}\le x$ give
\[
    0\le1-w_{c,k}(t)
    \le\frac1t\int_0^t\E_{\pi_{c,k}}[1-\psi_{c,k}(u,Z)^2]du
    \le t\E_{\pi_{c,k}}[q_k(Z)].
\]
The second inequality uses $1-\psi^2\le2(1-\psi)$, $1-e^{-x}\le x$, and stationarity.
The final expectation is bounded uniformly over $c\in C$ by \eqref{eq:pi_ck}.
Hence $w_{c,k}(t)\to1$ as $t\downarrow0$ uniformly on compact $c$-sets, which completes the proof of joint continuity.

It remains to prove the load endpoints in \eqref{eq:w_infty_formula}.
First let $c=-a$ with $a\to\infty$, and let $Z\sim\pi_{-a,k}$.
The scaled variable $U_a=aZ$ has density proportional to
\[
    \exp\left\{-u-\frac{u^{k+1}}{(k+1)a^{k+1}}\right\},
    \qquad u\ge0,
\]
so its laws are tight and converge to $\operatorname{Exp}(1)$.
Set
\[
    A_a\triangleq\int_0^\infty
    \exp\left\{-v-\frac{v^{k+1}}{(k+1)a^{k+1}}\right\}dv.
\]
Equation \eqref{eq:psi_infty_explicit} gives
\[
    \psi_\infty(u/a)
    =\frac{\exp\left\{u+\frac{u^{k+1}}{(k+1)a^{k+1}}\right\}}{A_a}
    \int_u^\infty
    \exp\left\{-v-\frac{v^{k+1}}{(k+1)a^{k+1}}\right\}dv.
\]
This expression converges to $1$ uniformly for $u$ in compact subsets of $[0,\infty)$.
For every $M<\infty$,
\[
    \E[1-\psi_\infty(Z)^2]
    \le\sup_{0\le u\le M}[1-\psi_\infty(u/a)^2]
    +\mathbb P(U_a>M).
\]
First letting $a\to\infty$ and then $M\to\infty$ proves $w_{-a,k}(\infty)=\E[\psi_\infty(Z)^2]\to1$.

Finally let $c\to\infty$, set $r=c^{1/k}$, and let $Z\sim\pi_{c,k}$.
The density of $U_c=Z/r$ is proportional to
\[
    \exp\left\{c^{(k+1)/k}\left(u-\frac{u^{k+1}}{k+1}\right)\right\},
    \qquad u\ge0.
\]
Lemma~\ref{lem:laplace_concentration_tail} applies because $u-u^{k+1}/(k+1)$ has the unique maximizer $u=1$, and hence $U_c\Rightarrow1$.
Equation \eqref{eq:psi_infty_explicit} and $\overline\Pi(x)\le1$ give
\[
    0<\psi_\infty(x)
    \le\exp\left\{-cx+\frac{x^{k+1}}{k+1}\right\}.
\]
Choose $\varepsilon>0$ and $\eta>0$ such that $u-u^{k+1}/(k+1)\ge\eta$ on $[1-\varepsilon,1+\varepsilon]$.
Then
\[
    \E[\psi_\infty(Z)^2]
    \le e^{-2\eta c^{(k+1)/k}}
    +\mathbb P(|U_c-1|>\varepsilon)
    \longrightarrow0.
\]
This proves the overloaded endpoint and completes the proof.
\end{proof}

\subsection{Proof of Lemma~\ref{lm:var_lim}}

\begin{proof}
Suppress the stationary subscript $e$ and set $p\triangleq1/(\rho\vee1)$.
Let $\mathcal A$ be the sigma-field generated by the stationary arrival process, and set $\mathcal H_i^\alpha\triangleq\mathcal A\vee\sigma\{Z^\alpha(0),(V_j,D_j^\alpha):1\le j<i\}$.
Then $W_i^\alpha=Z^\alpha(T_i-)$ is $\mathcal H_i^\alpha$-measurable, while $(V_i,D_i^\alpha)$ is independent of $\mathcal H_i^\alpha$.
Define
\[
    p_i^\alpha\triangleq\bar F(\alpha W_i^\alpha),
    \qquad
    X_i^\alpha\triangleq V_i\mathds{1}\{D_i^\alpha>W_i^\alpha\},
    \qquad
    M_i^\alpha\triangleq X_i^\alpha-\frac{p_i^\alpha}{\mu}.
\]
The conditional moments are
\[
    \E[M_i^\alpha\mid\mathcal H_i^\alpha]=0,
    \qquad
    \E[(M_i^\alpha)^2\mid\mathcal H_i^\alpha]
    =\frac{(1+c_s^2)p_i^\alpha-(p_i^\alpha)^2}{\mu^2}.
\]
Set $P_\alpha(t)\triangleq\sum_{i=1}^{A(t)}p_i^\alpha$ and $M_\alpha(t)\triangleq\sum_{i=1}^{A(t)}M_i^\alpha$.
Because $\{i\le A(t)\}\in\mathcal A$, orthogonality of the stopped martingale differences and $L^2$ truncation give
\[
    \E[M_\alpha(t)\mid\mathcal A]=0,
    \qquad
    \E[M_\alpha(t)^2]
    =\frac1{\mu^2}\E\left[\sum_{i=1}^{A(t)}\bigl((1+c_s^2)p_i^\alpha-(p_i^\alpha)^2\bigr)\right].
\]

Use the convention that a maximum over no arrivals equals zero.
The workload dynamics imply
\[
    \max_{1\le i\le A(t)}|W_i^\alpha-Z^\alpha(0)|
    \le t+\sum_{i=1}^{A(t)}V_i.
\]
If $\rho<1$, stationary coupling bounds $Z^\alpha(0)$ by the finite stationary workload of the queue that accepts every arrival, so $\alpha Z^\alpha(0)\to0$ in probability.
If $\rho>1$, the asserted convergence follows from the stationary fluid-limit assumption in the lemma.
Since $\bar F$ is globally Lipschitz with constant $\|f\|_\infty$, both cases yield
\[
    \max_{1\le i\le A(t)}|p_i^\alpha-p|
    \le \|f\|_\infty\alpha\left(t+\sum_{i=1}^{A(t)}V_i\right)
    +|\bar F(\alpha Z^\alpha(0))-p|
    \xrightarrow{\mathbb P}0.
\]
Assumption~\ref{assumption} gives $\E[A(t)^2]<\infty$, and
\[
    |P_\alpha(t)-pA(t)|^2
    \le A(t)^2\max_{1\le i\le A(t)}|p_i^\alpha-p|^2.
\]
The right-hand side converges to zero in probability and is dominated by the integrable variable $A(t)^2$.
Uniform integrability therefore gives
\[
    P_\alpha(t)-pA(t)\longrightarrow0\quad\text{in }L^2.
\]
The same maximum estimate applied to $x\mapsto(1+c_s^2)x-x^2$ gives
\[
    \E\left[\sum_{i=1}^{A(t)}\bigl((1+c_s^2)p_i^\alpha-(p_i^\alpha)^2\bigr)\right]
    \longrightarrow \lambda t\bigl((1+c_s^2)p-p^2\bigr).
\]
Consequently,
\[
    \E[M_\alpha(t)^2]
    \longrightarrow \frac{\lambda t}{\mu^2}\bigl(p(1-p)+c_s^2p\bigr).
\]

Since $Y^\alpha(t)=P_\alpha(t)/\mu+M_\alpha(t)$, the $L^2$ convergence above gives
\[
    \E[Y^\alpha(t)]\longrightarrow\frac{p\lambda t}{\mu},
    \qquad
    \Var(Y^\alpha(t))
    \longrightarrow\frac{p^2}{\mu^2}\Var(A(t))
    +\frac{\lambda t}{\mu^2}\bigl(p(1-p)+c_s^2p\bigr).
\]
Indeed, $\Cov(pA(t),M_\alpha(t))=0$ because $\E[M_\alpha(t)\mid\mathcal A]=0$, while the covariance with $P_\alpha(t)-pA(t)$ is $o(1)$ by Cauchy--Schwarz.
Finally, $p\lambda/\mu=\rho\wedge1$ and $I_a(t)=\Var(A(t))/(\lambda t)$, which gives both limits in the lemma.
\end{proof}

\subsection{Auxiliary Lemmas for the Unified RQ Proof}\label{sec:RQ_HT_auxiliary}

The following four lemmas isolate the only analytic ingredients shared by the first and refined RQ arguments.
They are stated immediately before the unified proof so that every dependency used there has already been established.

\begin{lemma}[Convergence of localized suprema]\label{lem:localized_suprema}
Let $K\subset[0,\infty)$ be compact.
Let $\phi_\alpha$ and $\phi$ be real-valued functions on $K\times[0,\infty)$, and suppose that $\phi_\alpha(z,0)=\phi(z,0)=0$ for all $z\in K$.
Assume that, for every $M<\infty$,
\[
    \sup_{z\in K}\sup_{0\le u\le M} |\phi_\alpha(z,u)-\phi(z,u)| \to 0 .
\]
Assume also that there exist constants $\eta>0$, $C<\infty$, and $\alpha_0>0$ such that, for all $\alpha<\alpha_0$, all $z\in K$, and all $u\ge0$,
\[
    \phi_\alpha(z,u)\le -\eta u+C\sqrt u, \qquad \phi(z,u)\le -\eta u+C\sqrt u .
\]
Then
\[
    \sup_{z\in K} \left| \sup_{u\ge0}\phi_\alpha(z,u) - \sup_{u\ge0}\phi(z,u) \right| \to 0 .
\]
\end{lemma}

\begin{proof}
Choose $M$ large enough that $-\eta u+C\sqrt u<0$ for all $u\ge M$.
Since $\phi_\alpha(z,0)=\phi(z,0)=0$, the supremum over $u\ge0$ equals the supremum over $0\le u\le M$, uniformly in $z\in K$ and all sufficiently small $\alpha$.
Hence
\[
    \sup_{u\ge0}\phi_\alpha(z,u)=\sup_{0\le u\le M}\phi_\alpha(z,u),
    \qquad \sup_{u\ge0}\phi(z,u)=\sup_{0\le u\le M}\phi(z,u).
\]
The conclusion follows from the compact-uniform convergence on $K\times[0,M]$.
\end{proof}

\begin{lemma}[IDW convergence under growing time changes]\label{lem:IDW_uniform_time_change}
Assume that $I_w(t)\to c_x^2\in(0,\infty)$ as $t\to \infty$, and that $\|I_w\|_\infty\triangleq\sup_{t\ge0}I_w(t)<\infty$.
Let $r_\alpha\to \infty$, and let $a_\alpha\to a>0$.
Then, for every $M<\infty$,
\[
    \sup_{0\le u\le M} \left| \sqrt{a_\alpha u I_w(r_\alpha u)} - \sqrt{a c_x^2u} \right| \to 0 .
\]
\end{lemma}

\begin{proof}
Fix $M<\infty$ and $\varepsilon>0$.
If $M=0$, the claim is immediate.
Let $\bar a<\infty$ be such that $a_\alpha\le \bar a$ for all sufficiently small $\alpha$ and $a\le\bar a$.
Choose $\eta\in(0,M]$ so small that
\[
    2\sqrt{\bar a\|I_w\|_\infty\eta}<\varepsilon/2 .
\]
For $0\le u\le\eta$, both square-root terms are bounded by $\sqrt{\bar a\|I_w\|_\infty\eta}$ for all sufficiently small $\alpha$.
On $\eta\le u\le M$, $r_\alpha u\to\infty$ uniformly, so $I_w(r_\alpha u)\to c_x^2$ uniformly.
Together with $a_\alpha\to a$, this gives uniform convergence on $[\eta,M]$.
Letting $\eta\downarrow0$ proves the claim.
\end{proof}

\begin{lemma}[Coefficient squeeze for a degenerate RQ supremum]\label{lem:coefficient_squeeze}
Let $0\le H_\alpha(u)\le \bar H<\infty$, let $0\le\varepsilon_\alpha\to0$, and suppose that $x_\alpha = \sup_{u\ge0} \left\{ a_\alpha u+\varepsilon_\alpha\sqrt{uH_\alpha(u)} \right\}$ is finite.
Then $a_\alpha\le0$.
If, in addition, $\liminf_{\alpha\downarrow0}x_\alpha>0$, then $a_\alpha\to 0$.
\end{lemma}

\begin{proof}
If $a_\alpha>0$, the supremum is infinite, so finiteness implies $a_\alpha\le0$.
Suppose that $a_{\alpha_n}\le-\eta$ along a subsequence, for some $\eta>0$.
Then
\[
    x_{\alpha_n} \le \sup_{u\ge0} \left\{ -\eta u+\varepsilon_{\alpha_n}\sqrt{\bar H u} \right\}
    = \frac{\bar H\varepsilon_{\alpha_n}^2}{4\eta} \to 0 .
\]
This contradicts $\liminf_{\alpha\downarrow0}x_\alpha>0$.
Therefore, for every $\eta>0$, $a_\alpha>-\eta$ eventually.
Since $a_\alpha\le0$, this proves $a_\alpha\to 0$.
\end{proof}

\begin{lemma}[Uniform convergence of refined variance terms]\label{lem:refined_variance_uniform}
Under Assumption~\ref{assumption}, let $\rho_{\alpha}\to1$, let $I_{a,\alpha}(t)\triangleq I_a^{(1)}(\lambda_\alpha t)$ be the physical-time IDC in the $\alpha$th system, and write $\hat I_{w,\alpha}(t)$ as \eqref{eq:IDW_refined_RQ} with $I_{a,\alpha}$ and $\rho_\alpha$.
Let $w:[0,\infty)\to[0,\infty)$ be continuous and bounded.
Let $K\subset[0,\infty)$ be compact, let $Q_\alpha:K\to[0,\infty)$ satisfy
\[
    \sup_{z\in K}|Q_\alpha(z)-1|\to0,
\]
let $r_\alpha\to\infty$, and let $\theta_\alpha\to\theta\ge0$.
Then, for every $M<\infty$,
\[
    \sup_{z\in K}\sup_{0\le u\le M} \left| \sqrt{ \rho_{\alpha} Q_\alpha(z)u \hat I_{w,\alpha}(r_\alpha u) w(\theta_\alpha u) } - \sqrt{ c_x^2 u w(\theta u) } \right|
    \to0 .
\]
\end{lemma}

\begin{proof}
Fix $M<\infty$.
The claim is trivial if $M=0$, so assume $M>0$.
Let $\bar H<\infty$ be a common upper bound for $\rho_{\alpha} Q_\alpha(z)\hat I_{w,\alpha}(t)w(\theta_\alpha u)$ and $c_x^2w(\theta u)$ for all sufficiently small $\alpha$, all $z\in K$, all $t\ge0$, and all $0\le u\le M$.
Such a bound exists by the boundedness assumptions.
Fix any $\eta\in(0,M]$.
For $0\le u\le\eta$, both square-root terms are bounded by $\sqrt{\bar H\eta}$, uniformly in $z\in K$ and $\alpha$.
For $\eta\le u\le M$, we have $r_\alpha u\ge r_\alpha\eta\to\infty$.
Choose $\lambda_*>0$ such that $\lambda_\alpha\ge\lambda_*$ for all sufficiently small $\alpha$.
The definition of $\hat I_{w,\alpha}$ gives
\[
    \sup_{u\in[\eta,M]}|\hat I_{w,\alpha}(r_\alpha u)-c_x^2|
    \le
    \sup_{s\ge\lambda_*r_\alpha\eta}|I_a^{(1)}(s)-c_a^2|+|1-c_a^2|\left|1-\frac{1}{\rho_\alpha\vee1}\right|
    \longrightarrow0.
\]
Also, $\rho_{\alpha} Q_\alpha(z)\to1$ uniformly over $z\in K$, and continuity of $w$ gives $w(\theta_\alpha u)\to w(\theta u)$ uniformly over $u\in[\eta,M]$.
Combining these facts gives uniform convergence on $K\times[\eta,M]$.
Letting $\eta\downarrow0$ proves the result.
\end{proof}

\subsection{Unified Proof of Theorems~\ref{Thm:RQ_HT} and~\ref{Thm:RQ_refined_HT}}\label{sec:proof_RQ_HT}

We prove the two RQ heavy-traffic theorems together because their scaled drift coefficients have the same limit in all three regimes.
The proof first rewrites the two fixed-point equations in a common form and then treats the underloaded, critically loaded, and overloaded regimes in that order.
In the first two regimes, compact-uniform convergence and Lemma~\ref{lem:localized_suprema} identify the limiting supremum, whereas in the overloaded regime Lemma~\ref{lem:coefficient_squeeze} identifies the limiting fluid balance.

Let $c_\alpha\triangleq\alpha^{-\gamma}(\rho_\alpha-1)$, so $c_\alpha\to c$, and let $\lambda_\alpha=\rho_\alpha\mu$.
Let $I_{a,\alpha}(t)\triangleq I_a^{(1)}(\lambda_\alpha t)$ and let $\hat I_{w,\alpha}$ denote the function in \eqref{eq:IDW_refined_RQ} evaluated with $I_{a,\alpha}$ and $\rho=\rho_\alpha$.
The index $j=1$ denotes the first RQ algorithm, and $j=2$ denotes the refined RQ algorithm.
For $j\in\{1,2\}$, write $Z_{j,\alpha}\triangleq Z_{\mathrm{RQ}_j,b}^\alpha$ and define
\[
\begin{aligned}
    a_{\alpha,1}(z)
    &\triangleq \rho_\alpha-\frac{1}{\bar F(\alpha z)},
    &q_{\alpha,1}(z,s)
    &\triangleq \frac{\rho_\alpha}{\mu}I_w(\lambda_\alpha s),\\
    a_{\alpha,2}(z)
    &\triangleq \rho_\alpha\bar F(\alpha z)-1,
    &q_{\alpha,2}(z,s)
    &\triangleq \frac{\rho_\alpha\bar F(\alpha z)}{\mu}\hat I_{w,\alpha}(s)w_{\tilde c_\alpha,k}(\alpha^{2h}\tau s).
\end{aligned}
\]
The change of variables used in \eqref{eq:RQ_ab_1} and the refined equation \eqref{eq:RQ_ab_2} give, for both $j=1$ and $j=2$,
\begin{equation}\label{eq:RQ_common_for_proof}
    Z_{j,\alpha}
    =\sup_{s\ge0}\left\{a_{\alpha,j}(Z_{j,\alpha})s+b\sqrt{s q_{\alpha,j}(Z_{j,\alpha},s)}\right\}.
\end{equation}
The functions $q_{\alpha,j}$ are bounded uniformly over all sufficiently small $\alpha$, all $z\ge0$, and all $s\ge0$.
For each fixed $\alpha$, their large-$s$ limits are strictly positive at every finite $z$, by $I_w(s)\to c_x^2>0$, Proposition~\ref{prop:w}, and the definition of $\hat I_{w,\alpha}$.
Consequently, finiteness of \eqref{eq:RQ_common_for_proof} forces $a_{\alpha,j}(Z_{j,\alpha})<0$.

\paragraph{Underloaded case $(c<0,\gamma<h)$.}

Let $d_\alpha\triangleq1-\rho_\alpha$, so $d_\alpha>0$ eventually and $d_\alpha\sim(-c)\alpha^\gamma$.
Set $\widehat Z_{j,\alpha}\triangleq d_\alpha Z_{j,\alpha}$.
Multiplying \eqref{eq:RQ_common_for_proof} by $d_\alpha$ and setting $u=d_\alpha^2s$ gives
\begin{equation}\label{eq:RQ_common_underloaded_scaled}
    \widehat Z_{j,\alpha}
    =\sup_{u\ge0}\left\{
        \frac{a_{\alpha,j}(d_\alpha^{-1}\widehat Z_{j,\alpha})}{d_\alpha}u
        +b\sqrt{u q_{\alpha,j}(d_\alpha^{-1}\widehat Z_{j,\alpha},d_\alpha^{-2}u)}
    \right\}.
\end{equation}
For both algorithms, $d_\alpha^{-1}a_{\alpha,j}(d_\alpha^{-1}z)\le-1$.
The uniform boundedness of $q_{\alpha,j}$ therefore implies $\widehat Z_{j,\alpha}\le\sup_{u\ge0}\{-u+C\sqrt u\}$, so both families of scaled solutions are bounded.

Let $K\subset[0,\infty)$ be a compact interval containing $\widehat Z_{j,\alpha}$ for both $j$ and all sufficiently small $\alpha$.
Applying Lemma~\ref{lem:uniform_F_expand} with its small parameter replaced by $\alpha/d_\alpha$ gives
\[
    \sup_{z\in K}\left|
        \frac{a_{\alpha,j}(d_\alpha^{-1}z)}{d_\alpha}+1
    \right|
    \longrightarrow0,
    \qquad j\in\{1,2\},
\]
because $F(\alpha z/d_\alpha)/d_\alpha=O(\alpha^k/d_\alpha^{k+1})\to0$ when $\gamma<h=k/(k+1)$.
For the first algorithm, Lemma~\ref{lem:IDW_uniform_time_change} gives, for every $M<\infty$,
\[
    \sup_{z\in K}\sup_{0\le u\le M}
    \left|
        \sqrt{u q_{\alpha,1}(d_\alpha^{-1}z,d_\alpha^{-2}u)}
        -\sqrt{\frac{c_x^2}{\mu}u}
    \right|
    \longrightarrow0.
\]
For the refined algorithm, \eqref{eq:ctilde_alpha} gives $\tilde c_\alpha\to-\infty$.
The time monotonicity and underloaded endpoint in Proposition~\ref{prop:w} give
\[
    \sup_{t\ge0}|w_{\tilde c_\alpha,k}(t)-1|
    \le1-w_{\tilde c_\alpha,k}(\infty)
    \longrightarrow0,
\]
so $w_{\tilde c_\alpha,k}(\alpha^{2h}\tau d_\alpha^{-2}u)$ may be replaced by $1$ uniformly for all $u\ge0$.
Lemma~\ref{lem:refined_variance_uniform}, applied with $r_\alpha=d_\alpha^{-2}$ and $Q_\alpha(z)=\bar F(\alpha z/d_\alpha)$, then gives the same compact-uniform variance convergence for $j=2$.
Thus, for either algorithm, the objective in \eqref{eq:RQ_common_underloaded_scaled} converges uniformly on $K\times[0,M]$ to
\[
    -u+b\sqrt{\frac{c_x^2}{\mu}u}.
\]
The common bound $-u+C\sqrt u$ localizes every supremum, so Lemma~\ref{lem:localized_suprema} yields
\[
    d_\alpha Z_{j,\alpha}
    \longrightarrow
    \sup_{u\ge0}\left\{-u+b\sqrt{\frac{c_x^2}{\mu}u}\right\}
    =\frac{b^2c_x^2}{4\mu},
    \qquad j\in\{1,2\}.
\]
Since $d_\alpha\sim(-c)\alpha^\gamma$ and $\E[Z_{M/M/1}]=\rho_\alpha/(\mu d_\alpha)$, this proves the underloaded conclusions of both theorems.

\paragraph{Critically loaded case $(\gamma\ge h)$.}

Set $\widehat Z_{j,\alpha}\triangleq\alpha^hZ_{j,\alpha}$.
Multiplying \eqref{eq:RQ_common_for_proof} by $\alpha^h$ and setting $u=\alpha^{2h}s$ gives
\begin{equation}\label{eq:RQ_common_critical_scaled}
    \widehat Z_{j,\alpha}
    =\sup_{u\ge0}\left\{
        \alpha^{-h}a_{\alpha,j}(\alpha^{-h}\widehat Z_{j,\alpha})u
        +b\sqrt{u q_{\alpha,j}(\alpha^{-h}\widehat Z_{j,\alpha},\alpha^{-2h}u)}
    \right\}.
\end{equation}
The scaled solutions are bounded for both algorithms.
Indeed, the boundedness of $q_{\alpha,j}$ and the strict negativity of the linear coefficient imply
\[
    \widehat Z_{j,\alpha}
    \le
    \frac{C^2}{4[-\alpha^{-h}a_{\alpha,j}(\alpha^{-h}\widehat Z_{j,\alpha})]}
\]
for a common finite constant $C$.
Fix $j\in\{1,2\}$ and suppose along a subsequence that $\widehat Z_{j,\alpha}\to\infty$.
Set $y_\alpha\triangleq\alpha^{1-h}\widehat Z_{j,\alpha}$ along that subsequence.
For both choices of $j$, and all sufficiently small $\alpha$,
\begin{equation}\label{eq:RQ_common_critical_drift_lower_bound}
    -\alpha^{-h}a_{\alpha,j}(\alpha^{-h}\widehat Z_{j,\alpha})
    \ge
    -\alpha^{-h}(\rho_\alpha-1)
    +\frac12\alpha^{-h}F(y_\alpha).
\end{equation}
For $j=1$, this follows from $1/\bar F-1\ge F$, and for $j=2$ it follows from $\rho_\alpha\ge1/2$ eventually.
Passing to a further subsequence, either $y_\alpha$ is bounded away from $0$ or $y_\alpha\to0$.
In the first case, the right-hand side of \eqref{eq:RQ_common_critical_drift_lower_bound} diverges to infinity.
In the second case, $F(y_\alpha)\ge(\beta/2)y_\alpha^k$ eventually, and $\alpha^{-h}y_\alpha^k=\widehat Z_{j,\alpha}^k$ because $k(1-h)=h$.
In both cases, $-\alpha^{-h}a_{\alpha,j}(\alpha^{-h}\widehat Z_{j,\alpha})\to\infty$, so the preceding upper bound forces $\widehat Z_{j,\alpha}\to0$, contradicting the assumed divergence.

Let $\widehat Z_{j,\alpha_n}\to z$ be an arbitrary convergent subsequence.
Lemma~\ref{lem:uniform_F_expand} gives, uniformly for $x$ in compact subsets of $[0,\infty)$,
\begin{equation}\label{eq:RQ_common_critical_drift_limit}
    \alpha^{-h}a_{\alpha,j}(\alpha^{-h}x)
    \longrightarrow
    \mathds{1}\{\gamma=h\}c-\beta x^k,
    \qquad j\in\{1,2\}.
\end{equation}
The definition \eqref{eq:ctilde_alpha} also gives
\[
    \tilde c_\alpha
    \longrightarrow
    \mathds{1}\{\gamma=h\}c
    \left(\frac{c_x^2}{2\mu}\right)^{-k/(k+1)}
    \beta^{-1/(k+1)}
    =\tilde c_\gamma.
\]
Define
\[
    \omega_1(u)\triangleq1,
    \qquad
    \omega_2(u)\triangleq w_{\tilde c_\gamma,k}(\tau u).
\]
For $j=1$, Lemma~\ref{lem:IDW_uniform_time_change}, applied with $a_\alpha=\rho_\alpha/\mu$ and $r_\alpha=\lambda_\alpha\alpha^{-2h}$, gives the required variance limit.
For $j=2$, the joint continuity in Proposition~\ref{prop:w} permits uniform replacement of $w_{\tilde c_\alpha,k}(\tau u)$ by $w_{\tilde c_\gamma,k}(\tau u)$ on compact $u$-intervals, and Lemma~\ref{lem:refined_variance_uniform} applies with $Q_\alpha(x)=\bar F(\alpha^{1-h}x)$, $r_\alpha=\alpha^{-2h}$, and $\theta_\alpha=\tau$.
Consequently, for every compact $K\subset[0,\infty)$ and every $M<\infty$,
\begin{equation}\label{eq:RQ_common_critical_variance_limit}
    \sup_{x\in K}\sup_{0\le u\le M}
    \left|
        \sqrt{u q_{\alpha,j}(\alpha^{-h}x,\alpha^{-2h}u)}
        -\sqrt{\frac{c_x^2}{\mu}\omega_j(u)u}
    \right|
    \longrightarrow0,
    \qquad j\in\{1,2\}.
\end{equation}

Finiteness of \eqref{eq:RQ_common_critical_scaled} and \eqref{eq:RQ_common_critical_drift_limit} imply $\mathds{1}\{\gamma=h\}c-\beta z^k\le0$.
Equality is impossible.
If equality held, then evaluating \eqref{eq:RQ_common_critical_scaled} at any fixed $u=U>0$ and using \eqref{eq:RQ_common_critical_variance_limit} would give
\[
    z\ge b\sqrt{\frac{c_x^2}{\mu}\omega_j(U)U}.
\]
Letting $U\to\infty$ gives a contradiction because $\omega_1(U)=1$ and $\omega_2(U)\to w_{\tilde c_\gamma,k}(\infty)>0$ by Proposition~\ref{prop:w}.
Hence $\mathds{1}\{\gamma=h\}c-\beta z^k<0$.

Choose a compact neighborhood $K$ of $z$ and $\eta>0$ such that $\mathds{1}\{\gamma=h\}c-\beta x^k\le-2\eta$ for every $x\in K$.
Equations \eqref{eq:RQ_common_critical_drift_limit} and \eqref{eq:RQ_common_critical_variance_limit} give compact-uniform convergence of the scaled objectives in \eqref{eq:RQ_common_critical_scaled} on $K\times[0,M]$.
The bound $-\eta u+C\sqrt u$ localizes the prelimit and limit suprema uniformly on $K$.
The limiting supremum is continuous in $x$ because the localized objective is jointly continuous.
Lemma~\ref{lem:localized_suprema} therefore yields
\begin{equation}\label{eq:RQ_common_critical_limit_fixed_point}
    z
    =\sup_{u\ge0}\left\{
        \left(\mathds{1}\{\gamma=h\}c-\beta z^k\right)u
        +b\sqrt{\frac{c_x^2}{\mu}\omega_j(u)u}
    \right\}.
\end{equation}
The right-hand side is positive for all sufficiently small positive $u$, so $z>0$.
The right-hand side of \eqref{eq:RQ_common_critical_limit_fixed_point} is nonincreasing in $z$, whereas the left-hand side is strictly increasing, so the fixed-point equation has at most one positive solution.
Every convergent subsequence therefore has the same limit, which proves convergence of the full sequence for each $j$.

For $j=1$, evaluating the elementary supremum in \eqref{eq:RQ_common_critical_limit_fixed_point} gives
\[
    -\mathds{1}\{\gamma=h\}cz+\beta z^{k+1}
    =\frac{b^2c_x^2}{4\mu},
\]
which is \eqref{eq:RQ_HT_equation}.
For $j=2$, equation \eqref{eq:RQ_common_critical_limit_fixed_point} is exactly \eqref{eq:RQ_HT_equation_general}.
This proves the critically loaded conclusions of both theorems.

\paragraph{Overloaded case $(c>0,\gamma<h)$.}

Set $\widehat Z_{j,\alpha}\triangleq\alpha^{1-\gamma/k}Z_{j,\alpha}$.
Multiplying \eqref{eq:RQ_common_for_proof} by $\alpha^{1-\gamma/k}$ and setting $u=\alpha^{1+(k-1)\gamma/k}s$ gives
\begin{equation}\label{eq:RQ_common_overloaded_scaled}
    \widehat Z_{j,\alpha}
    =\sup_{u\ge0}\left\{
        \alpha^{-\gamma}a_{\alpha,j}(\alpha^{-(1-\gamma/k)}\widehat Z_{j,\alpha})u
        +b\alpha^{\frac12[1-(k+1)\gamma/k]}
        \sqrt{u q_{\alpha,j}(\alpha^{-(1-\gamma/k)}\widehat Z_{j,\alpha},\alpha^{-1-(k-1)\gamma/k}u)}
    \right\}.
\end{equation}
The square-root coefficient tends to $0$ because $\gamma<h$, and its remaining factor is uniformly bounded.
Lemma~\ref{lem:coefficient_squeeze} first gives
\[
    \alpha^{-\gamma}a_{\alpha,j}(\alpha^{-(1-\gamma/k)}\widehat Z_{j,\alpha})\le0.
\]
For either $j$, this inequality is equivalent to $\rho_\alpha\bar F(\alpha^{\gamma/k}\widehat Z_{j,\alpha})\le1$, and hence
\begin{equation}\label{eq:RQ_common_overloaded_positive_bound}
    \rho_\alpha\alpha^{-\gamma}F(\alpha^{\gamma/k}\widehat Z_{j,\alpha})
    \ge c_\alpha.
\end{equation}
Equation \eqref{eq:RQ_common_overloaded_positive_bound} implies $\liminf_{\alpha\downarrow0}\widehat Z_{j,\alpha}>0$, because otherwise its left-hand side would converge to $0$ by Lemma~\ref{lem:uniform_F_expand}, while $c_\alpha\to c>0$.
The second conclusion of Lemma~\ref{lem:coefficient_squeeze} now gives
\[
    \alpha^{-\gamma}a_{\alpha,j}(\alpha^{-(1-\gamma/k)}\widehat Z_{j,\alpha})\longrightarrow0.
\]
Since
\[
    \alpha^{-\gamma}\left(\rho_\alpha\bar F(\alpha^{\gamma/k}\widehat Z_{j,\alpha})-1\right)
    =
    \begin{cases}
        \bar F(\alpha^{\gamma/k}\widehat Z_{1,\alpha})\alpha^{-\gamma}a_{\alpha,1}(\alpha^{-(1-\gamma/k)}\widehat Z_{1,\alpha}), & j=1,\\
        \alpha^{-\gamma}a_{\alpha,2}(\alpha^{-(1-\gamma/k)}\widehat Z_{2,\alpha}), & j=2,
    \end{cases}
\]
we obtain, for both algorithms,
\[
    \rho_\alpha\alpha^{-\gamma}F(\alpha^{\gamma/k}\widehat Z_{j,\alpha})
    =c_\alpha-\alpha^{-\gamma}\left(\rho_\alpha\bar F(\alpha^{\gamma/k}\widehat Z_{j,\alpha})-1\right)
    \longrightarrow c.
\]
It follows that $F(\alpha^{\gamma/k}\widehat Z_{j,\alpha})\to0$ and hence $\alpha^{\gamma/k}\widehat Z_{j,\alpha}\to0$ by the full-support assumption.
Assumption~\ref{assumption:F} then gives
\[
    \rho_\alpha\alpha^{-\gamma}F(\alpha^{\gamma/k}\widehat Z_{j,\alpha})
    =\rho_\alpha\beta\widehat Z_{j,\alpha}^k(1+o(1)).
\]
Therefore $\beta\widehat Z_{j,\alpha}^k\to c$, and
\[
    \alpha^{1-\gamma/k}Z_{j,\alpha}\longrightarrow\left(\frac{c}{\beta}\right)^{1/k},
    \qquad j\in\{1,2\}.
\]
Finally,
\[
    \frac{\rho_\alpha F(\alpha Z_{j,\alpha})}{\rho_\alpha-1}
    =\frac{\rho_\alpha\alpha^{-\gamma}F(\alpha^{\gamma/k}\widehat Z_{j,\alpha})}{c_\alpha}
    \longrightarrow1,
\]
which proves the equivalent overloaded statements and completes both proofs.

\end{document}